\documentclass[10pt]{amsart}
\newfont{\cyr}{wncyr10 scaled 1100}
\usepackage[top=1.3in, bottom=1.3in, left=1.1in, right=1.1in]{geometry}
\usepackage{amsthm,amssymb,amsmath,amsfonts,mathrsfs,amscd, graphics}
\usepackage{mathtools}
\usepackage[latin1]{inputenc}
\usepackage{pifont}
\usepackage{circledsteps}
\usepackage[all]{xy}
\usepackage{latexsym}
\usepackage{longtable}
\usepackage{color}
\usepackage{tikz}
\usepackage{tikz-cd}
\usepackage{lua-visual-debug}
\usepackage{dsfont}
\usetikzlibrary{matrix}
\usepackage{hyperref}
\usepackage{enumitem}
\usepackage{setspace}

\newcommand*\ZZ{|[draw,circle]| \Z_2}
\numberwithin{equation}{section}
\theoremstyle{plain}
\newtheorem{theorem}{Theorem}[section]
\newtheorem*{theorem*}{Theorem}
\newtheorem{corollary}[theorem]{Corollary}
\newtheorem{lemma}[theorem]{Lemma}
\newtheorem{proposition}[theorem]{Proposition}
\newtheorem{conjecture}[theorem]{Conjecture}
\numberwithin{equation}{section}
\newtheorem{thm}{Theorem}
\newtheorem{ass}[thm]{Assumption}

\theoremstyle{definition}

\newtheorem{definition}[theorem]{Definition}

\newtheorem{assumption}[theorem]{Assumption}
\newtheorem{remark}[equation]{Remark}

\theoremstyle{remark}
\newtheorem{obswr}[theorem]{Observation}

\newtheorem{intro-definition}[theorem]{Definition}

\newenvironment{myproof}[2] {\paragraph{\emph{Proof of {#1} {#2} }}}{\hfill$\square$}

\def\Gal{\mathrm{Gal}}
\def\GL{\mathrm{GL}}

\def\det{\mathrm{det}}

\def\loc{\mathrm{loc}}
\def\ord{\mathrm{ord}}
\def\im{\mathrm{im}}
\def\coker{\mathrm{coker}}

\def\AJ{\mathrm{AJ}}

\def\sing{\mathrm{sin}}
\def\Frob{\mathrm{Frob}}
\def\new{\mathrm{new}}
\def\ac{\mathrm{ac}}
\def\frakm{\mathfrak{m}}

\def\Qbar{\QQ^{\ac}}

\def\interX{\mathfrak{X}}
\def\ss{\mathrm{ss}}
\def\ur{\mathrm{ur}}
\def\ac{\mathrm{ac}}

\def\ram{\mathrm{ram}}

\def\undlamb{\underline{\lambda}}
\def\triplef{\underline{\mathbf{f}}}

\def\threetensor{\overset{3}{\underset{i=1}\otimes}}
\def\threesum{\overset{3}{\underset{j=1}\oplus} }
\def\triplepnp {{\frakp}^{[p]}_{\triplef, n}}

\DeclareMathOperator{\Sym}{Sym}
\DeclareMathOperator{\Spec}{Spec}

\DeclareMathOperator{\Hom}{Hom}

\DeclareMathOperator{\End}{End}

\DeclareMathOperator{\tr}{tr}

\def\calH{\mathcal{H}}

\def\calO{\mathcal{O}}

\def\frakm{\mathfrak{m}}

\def\frakp{\mathfrak{p}}

\def\Adel{\mathbf{A}}

\def\CC{\mathbf{C}}

\def\FF{\mathbf{F}}

\def\PP{\mathbf{P}}
\def\QQ{\mathbf{Q}}

\def\TT{\mathbf{T}}

\def\ZZ{\mathbf{Z}}

\def\rmA{\mathrm{A}}

\def\rmD{\mathrm{D}}
\def\rmG{\mathrm{G}}
\def\rmH{\mathrm{H}}
\def\rmI{\mathrm{I}}
\def\rmE{\mathrm{E}}
\def\rmF{\mathrm{F}}

\def\rmM{\mathrm{M}}
\def\rmN{\mathrm{N}}

\def\rmX{\mathrm{X}}
\def\rmXbar{\overline{\rmX}}
\def\rmY{\mathrm{Y}}
\def\rmYbar{\overline{\rmY}}
\def\rmU{\mathrm{U}}
\def\rmV{\mathrm{V}}

\def\rmR{\mathrm{R}}
\def\rmS{\mathrm{S}}
\def\rmT{\mathrm{T}}
\def\rmZ{\mathrm{Z}}

\def\bfW{\mathbf{W}}

\newcommand{\Iw}{\mathrm{Iw}}

\newcommand{\Gr}{\mathbf{Gr}}

\include{thebibliography}

\begin{document}

\title[Level raising and Diagonal cycles]
{Arithmetic level raising on triple product of Shimura curves and Gross--Kudla--Schoen Diagonal cycles II: Bipartite Euler system}

\author{Haining Wang }
\address{\parbox{\linewidth}{address:\\Shanghai Center for Mathematical Sciences,\\ Fudan University,\\No,2005 Songhu Road,\\Shanghai,200438, China.~ }}
\email{wanghaining1121@outlook.com}

\begin{abstract}
In this article, we study the Gross--Kudla--Schoen diagonal cycle on the triple product of Shimura curves at a place of good reduction and prove an unramified arithmetic level raising theorem for the cohomology of this triple product. We deduce from it  a reciprocity law which relates the image of the diagonal cycle under the Abel--Jacobi map to certain period integral of Gross--Kudla type. Combing this with the first reciprocity law we proved in a previous work, we show that the Gross--Kudla--Schoen diagonal cycles along with the Gross--Kudla periods form a bipartite Euler system for the symmetric cube motive of a modular form. As an application we provide some evidence for the rank one case of the Bloch--Kato conjecture for the symmetric cube motive of a modular form.
\end{abstract}

\date{\today}

\maketitle
\tableofcontents
\section{Introduction}
In a seminal work of Bertolini--Darmon \cite{BD}, the authors constructed a Kolyvagin type Euler system using Heegner points on various Shimura curves. The cohomology classes in this system satisfy beautiful reciprocity laws that resemble the so-called Jochnowitz's congruence. More precisely, these reciprocity laws relate the Kummer images of the Heegner points to certain toric period integrals via level-raising congruences. One can consider these Heegner point classes as the incarnations of the first derivatives of certain Rankin--Selberg $L$-functions with global root number $-1$ at the central critical points in light of the Gross--Zagier type formula while the toric period integrals are the algebraic parts of the central critical values of the Rankin--Selberg type $L$-functions as shown by the Gross' formula. The reciprocity laws alluded above are reflections of the congruences between the algebraic parts of the first derivatives of certain Rankin--Selberg $L$-functions with global root number $+1$ and  the algebraic parts of the central critical values of the Rankin--Selberg type $L$-functions observed by Jochnowitz using computational method. Using these reciprocity laws,  Bertolini--Darmon are able to verify the one-sided divisibility of the Iwasawa main conjecture for an elliptic curve over the anticyclotomic extension of an imaginary quadratic field. The method of Bertolini--Darmon is axiomatized by \cite{How} and is given the name of a bipartite Euler system. In \cite{How}, it was also mentioned that the cycle on the triple product of Shimura curves studied in \cite{GK} and \cite{GS} known as the Gross--Kudla--Schoen diagonal cycle might form another example of such a bipartite Euler system. The main theme of this article is to investigate Howard's conjecture for the Gross--Kudla--Schoen diagonal cycle.

The present article completes what was started in \cite{Wang} where we verified one of the reciprocity laws (the first reciprocity law). We will show below that the Gross--Kudla--Schoen diagonal cycles and the Gross--Kudla periods on the triple product of Shimura curves satisfy all the reciprocity laws needed to form a bipartite Euler system for the triple product motive of modular forms. However these cycles do not give a bipartite Euler system in the strict sense of \cite{How} for obvious rank reasons.  On the other hand, when we consider the symmetric cube component of the triple tensor product motive, then the image of these Gross--Kudla--Schoen diagonal cycles in the symmetric cube component do form a bipartite Euler system. Using this bipartite Euler system,  we provide some evidence towards the rank one case of the Bloch--Kato conjecture for the symmetric cube motive of a modular form. A refinement of the method in this article should also be able to handle the more general case of the triple product motive of modular forms. We decide to treat this in another occasion where we also plan to extend the results here to an Iwasawa theoretic setting.

Note that the Gross--Kudla--Schoen diagonal cycle is also studied extensively using the $p$-adic method. This was initiated in \cite{DR-1} and \cite{DR-2} where certain $p$-adic reciprocity laws in the style of Perrin-Riou are proved. The reciprocity laws in this article can be viewed as the counter part of those in \cite{DR-1}  and \cite{DR-2} in the level raising situation. We will refer to \cite{BDRSV} and \cite{Ca-H} for further development and arithmetic applications in this direction. It also should be possible to combine the results in this article with the $p$-adic methods to prove interesting Iwasawa theoretic results as was done in \cite{BD}. In a parallel direction, Liu and Tian \cite{Liu-HZ}, \cite{Liu-cubic} and \cite{LT} have constructed the desired bipartite Euler system for the triple product motive realized on certain Hilbert modular threefold and in a recent breakthrough, Liu, Tian, Xiao, Zhang and Zhu have constructed a bipartite Euler system for the rankin-selberg motive of certain unitary groups in \cite{LTX, LTXZZ}.

\subsection{Main results} In order to state our main results, we introduce some notations. Let $\ell\geq 5$ be a prime. Let $\triplef=(f_{1}, f_{2}, f_{3})$ be a triple of normalized newforms in $S^{\new}_{2}(\Gamma_{0}(N))^{3}$. Throughout this article, we assume that we have a factorization $N=N^{+}N^{-}$ such that $(N^{+}, N^{-})=1$ and $N^{-}$ is square-free with \emph{even} number of prime factors. To give a uniform exposition, we give proofs only in the case when $N^{-}\neq 1$. This convenient assumption puts us in the situation where all the Shimura curves considered in this article are proper, but it should be clear to the reader that there is no difficulty extending the results to the case of non-proper Shimura curves and hence to the case $N^{-}=1$ using the same method.  For $i\in\{1, 2, 3\}$, let $E_{i}=\QQ(f_{i})$ be the Hecke field of $f_{i}$ and $\lambda_{i}$ be a place of $E_{i}$ above the prime $\ell$. We denote by $\calO_{\lambda_{i}}$ the valuation ring of the completion $E_{\lambda_{i}}$ of $E$ at $\lambda_{i}$. Let $\varpi_{i}$ be a uniformizer of  $\calO_{\lambda_{i}}$ and $\lambda_{i}=(\varpi_{i})$ be the maximal ideal. We put $\calO_{\lambda_{i}, n}=\calO_{\lambda_{i}}/\lambda_{i}^{n}$ for any integer $n\geq 1$. Let $k_{i}$ be the residue field of $E_{\lambda_{i}}$. 

We denote by ${\rho}_{f_{i}, \lambda_{i}}: \rmG_{\QQ}\rightarrow \GL_{2}(E_{\lambda_{i}})=\GL(\rmV_{f_{i},\lambda_{i}})$ the $\lambda_{i}$-adic representation attached to $f_{i}$ and whose residual representation will be denoted by $\bar{\rho}_{f_{i}, \lambda_{i}}$.  We consider the following sets: $\Sigma^{+}$ is the set of prime divisors of $N^{+}$; $\Sigma^{-}_{\ram}$ is the set of prime divisor $r$ of $N^{-}$ such that $\ell \mid r^{2}-1$. 
\begin{ass}\label{intro-ass}
We make the following assumptions on $\bar{\rho}_{f_{i}, \lambda_{i}}$ and thus on $\frakm_{i}$:
\begin{enumerate}
\item  $\bar{\rho}_{f_{i}, \lambda_{i}}\vert_{\rmG_{\QQ(\zeta_{\ell})}}$ is absolutely irreducible;
\item  $\bar{\rho}_{f_{i}, \lambda_{i}}$ is minimal at primes in $\Sigma^{+}\cup\Sigma^{-}_{\ram}$ and is ramified at primes in $\Sigma^{-}_{\ram}$;
\item The image of $\bar{\rho}_{f_{i}, \lambda_{i}}$ contains $\GL_{2}(\FF_{\ell})$.
\end{enumerate}
\end{ass}

Here $\bar{\rho}_{f_{i}, \lambda_{i}}$ being minimal at primes in $\Sigma^{+}\cup\Sigma^{-}_{\ram}$ means all the lifts of  $\bar{\rho}_{f_{i}, \lambda_{i}}$ are minimally ramified in the sense of \cite[\S 29]{Maz}. We introduce an auxiliary prime $d$ and assume it is clean with respect to $(\overline{\rho}_{f_{1}, \lambda_{1}}, \overline{\rho}_{f_{2}, \lambda_{2}}, \overline{\rho}_{f_{3}, \lambda_{3}})$ in the sense of Definition \ref{clean}.  This auxiliary prime is used to make the moduli problems associated to the Shimura curves in this article representable. Let $B=B_{N^{-}}$ be the indefinite quaternion algebra of discriminant $N^{-}$. Let $\rmX_{d}$ be the Shimura curve over $\QQ$ associated to $B$ with a level structure given by an Eichler order of level $N^{+}$ and an auxiliary level structure at $d$. Let $p$ be a prime away from $Nd$, $\overline{B}=B_{pN^{-}}$ be the definite quaternion algebra over $\QQ$ and $\rmZ_{d}(\overline{B})$ be the Shimura set associated to $\overline{B}$ with a level structure given by an Eichler order of level $N^{+}$ and an auxiliary level structure at $d$. 

For a finite set of primes $S$ away from $Nd$, let $\TT^{[S]}$ be the Hecke algebra containing all the Hecke operators away from $Nd$ and all the primes in $S$. If $S$ is empty, then we will simply denote $\TT^{[S]}$ by $\TT$. The Hecke eigensystem of $f_{i}$ gives rise to a morphism $\phi_{i}: \TT\rightarrow\calO_{\lambda_{i}}$ whose reduction modulo $\lambda^{n}_{i}$ will be denoted by $\phi_{i, n}$. The kernel of this morphism will be denoted by $\frakp_{i, n}$ and  $\frakp_{i, 1}$ will be denoted by $\frakm_{i}$. We will denote by $\frakm_{\triplef}$ the triple of maximal ideals $(\frakm_{1}, \frakm_{2}, \frakm_{3})$ and by $\frakp_{\triplef,n}$ the triple of ideals given by $(\frakp_{1, n}, \frakp_{2, n}, \frakp_{3, n})$.  The triple tensor product Hecke algebras $\TT\otimes\TT\otimes \TT$  acts naturally on the \'etale cohomology of $\rmX^{3}_{d}$ by the K\"{u}nneth formula. Let $\calO_{\undlamb}=\calO_{{\lambda_{1}}}\otimes \calO_{{\lambda_{2}}}\otimes \calO_{{\lambda_{3}}}$ and $\calO_{\undlamb, n}=\calO_{\lambda_{1}, n}\otimes \calO_{\lambda_{2}, n}\otimes \calO_{\lambda_{3}, n}$.  We consider the $\calO_{\underline{\lambda}}[\rmG_{\QQ}]$-module 
\begin{equation*}
\rmM(\triplef, d)(-1)= \rmH^{3}(\rmX^{3}_{d}\otimes\Qbar, \calO_{\undlamb}(2))_{\frakm_{\triplef}}
\end{equation*}
and the $\calO_{\underline{\lambda}, n}[\rmG_{\QQ}]$-module
\begin{equation*}
\rmM_{n}(\triplef, d)(-1) = \rmH^{3}(\rmX^{3}_{d}\otimes {\Qbar}, \calO_{\undlamb}(2))_{/\frakp_{\triplef, n}}.
\end{equation*}

We introduce the following notion of an $n$-admissible prime for the triple $\triplef$. Such a  prime is a common level raising prime for the triple $\triplef$.
\begin{intro-definition}
A prime $p$ is $n$-admissible for the triple $\triplef=(f_{1}, f_{2}, f_{3})$ if
\begin{enumerate}
\item $p\nmid N\ell$;
\item $\ell\nmid p^{2}-1$;
\item $\varpi^{n}_{i}\mid p+1-\epsilon_{p, i}a_{p}(f_{i})$ with $\epsilon_{p,i}\in \{\pm1\}$ for $i\in\{1, 2, 3\}$;
\item $\epsilon_{p, 1}\epsilon_{p, 2}\epsilon_{p, 3}=1$.
\end{enumerate}
\end{intro-definition}
Let $p$ be an $n$-admissible prime for $\triplef$. Let $\frakm^{[p]}_{i}=\frakm_{i}\cap \TT^{[p]}$ and $\frakp^{[p]}_{i, n}=\frakp_{i,n}\cap \TT^{[p]}$ for $i\in\{1, 2, 3\}$. Then our first result is the following arithmetic level raising theorem for the triple product of Shimura curves. 
\begin{thm}\label{intro-level-raise}
Let $p$ be an $n$-admissible prime for $\triplef$. We assume that each maximal ideal in the triple $\frakm_{\triplef}=(\frakm_{1}, \frakm_{2}, \frakm_{3})$ satisfies Assumption \ref{intro-ass}. Then we have the following isomorphism
\begin{equation*}
\Phi_{\triplef, n}:\threesum(\threetensor\Gamma(\rmZ_{d}(\overline{B}),\calO_{\lambda_{i}})_{/\frakp^{[p]}_{i, n}})\cong\rmH^{1}(\FF_{p}, \rmM_{n}(\triplef, d)(-1))
\end{equation*}
between $\calO_{\undlamb, n}$-modules. 
\end{thm}

We consider the diagonal embedding $\theta: \rmX_{d}\rightarrow \rmX_{d}^{3}$ of $\rmX_{d}$ into its triple fiber product $\rmX^{3}_{d}$. The resulting class $\Delta_{d}=\theta_{*}[\rmX_{d}]\in \mathrm{CH}^{2}(\rmX^{3}_{d})$ will be referred to as the {Gross--Kudla--Schoen diagonal cycle} as in the title of this article. Let $n\geq 1$ be an integer, we will introduce the Abel--Jacobi map
\begin{equation*}
\mathrm{AJ}_{\triplef}: \mathrm{CH}^{2}(\rmX^{3}_{d})\rightarrow \rmH^{1}(\QQ,  \rmM(\triplef, d)(-1))
\end{equation*} 
for  $\rmM(\triplef, d)(-1)$ and the Abel--Jacobi map 
\begin{equation*}
\mathrm{AJ}_{\triplef, n}: \mathrm{CH}^{2}(\rmX^{3}_{d})\rightarrow \rmH^{1}(\QQ,  \rmM_{n}(\triplef, d)(-1))
\end{equation*}
for $\rmM_{n}(\triplef, d)(-1)$. We denote by 
\begin{equation*}
\Theta(\triplef, d) \in \rmH^{1}(\QQ,  \rmM(\triplef, d)(-1)) 
\end{equation*}
the image of $\Delta_{d}$ under $\mathrm{AJ}_{\triplef}$ and by 
\begin{equation*}
\Theta_{n}(\triplef, d)\in \rmH^{1}(\QQ,  \rmM_{n}(\triplef, d)(-1)) 
\end{equation*}
the image of $\Delta_{d}$ under $\mathrm{AJ}_{\triplef, n}$. 

We will define a natural bilinear pairing 
\begin{equation*}
 (\hphantom{a}, \hphantom{b}): \threetensor\Gamma(\rmZ_{d}(\overline{B}), \calO_{\lambda_{i}})_{/\frakp^{[p]}_{i, n}}\times  \threetensor\Gamma(\rmZ_{d}(\overline{B}), E_{\lambda_{i}}/\calO_{\lambda_{i}})[\frakp^{[p]}_{i, n}]\rightarrow \calO_{\undlamb, n}
\end{equation*}
induced from the Poincar\'e duality on $\rmZ_{d}(\overline{B})$. The localization $\loc_{p}(\Theta_{n}(\triplef, d))$ of $\Theta_{n}(\triplef, d)$ is an element in $\rmH^{1}(\FF_{p},  \rmM_{n}(\triplef,d)(-1))$ which in turn can be viewed as an element in the space 
\begin{equation*}
\threesum(\threetensor\Gamma(\rmZ_{d}(\overline{B}), \calO_{{\lambda_{i}}})_{/\frakp^{[p]}_{i, n}})
\end{equation*}
via the isomorphism $\Phi_{\triplef, n}$ in the above Theorem. For $j\in\{1, 2, 3\}$, we will denote by 
$$\loc^{(j)}_{p}(\Theta_{n}(\triplef,d))$$ 
the component of $\loc_{p}(\Theta_{n}(\triplef, d))$ in the $j$-th copy of the space 
$\threesum(\threetensor\Gamma(\rmZ_{d}(\overline{B}), \calO_{{\lambda_{i}}})_{/\frakp^{[p]}_{i, n}})$. The following  reciprocity formula relating the Gross--Kudla--Schoen diagonal cycle class $\Theta_{n}(\triplef, d)$ to certain {Gross--Kudla type period integral}. 
\begin{thm}\label{unram-intro}
Let $p$ be an $n$-admissible prime for $\triplef$. Suppose each $\frakm_{i}$ satisfies Assumption \ref{intro-ass}. Then the following formula
\begin{equation*}
(\loc^{(j)}_{p}(\Theta_{n}(\triplef,d)), \phi_{1}\otimes\phi_{2}\otimes\phi_{3})=\sum_{z\in \Delta_{d}(\overline{B})} \phi_{1}(z)\otimes \phi_{2}(z)\otimes \phi_{3}(z)
\end{equation*}
holds for every $\phi_{1}\otimes \phi_{2}\otimes \phi_{2}\in \threetensor\Gamma(\rmZ_{d}(\overline{B}), E_{\lambda_{i}}/\calO_{{\lambda_{i}}})[\frakp^{[p]}_{i, n}]$ and $j\in\{1, 2, 3\}$. 
\end{thm}
This result is the main technical result of this article and its proof consists of  two main ingredients. The first one is a simultaneous refinement of the level raising result of \cite{Ri-ICM} and the arithmetic level raising result of \cite{BD}, see Proposition \ref{new-old-part} . The second one is the study of the monodromy filtration of the neaby cycle cohomology of the triple product of Shimura curves with Iwahori level structure treated in \S $5$.

Next we consider a pair of $n$-admissible primes $(p, q)$ for $\triplef$. Consider the indefinite quaternion algebra $B^{\natural}$ of discriminant $N^{-}pq$. Then we can associate to it a Shimura curve over $\QQ$ denoted by $\rmX^{\natural}_{d}$. 
We have the triple $\frakm^{[pq]}_{\triplef}=(\frakm^{[pq]}_{1},\frakm^{[pq]}_{2}, \frakm^{[pq]}_{3})$ for $\frakm^{[pq]}_{i}=\TT^{[pq]}\cap \frakm_{i}$ and the triple $\frakp^{[pq]}_{\triplef,n}=(\frakp^{[pq]}_{1, n},\frakp^{[pq]}_{2,n}, \frakp^{[pq]}_{3,n})$ for $\frakp^{[pq]}_{i,n}=\frakp_{i, n}\cap\TT^{[pq]}$. We  consider the natural diagonal morphism 
\begin{equation*}
\theta^{\natural}: \rmX^{\natural}_{d}\rightarrow \rmX^{\natural3}_{d}
\end{equation*}
of $\rmX^{\natural}_{d}$ into the triple fiber product $\rmX^{\natural3}_{d}$. We define the $\calO_{\undlamb, n}[\rmG_{\QQ}]$-module $\rmM^{[{pq}]}_{n}(\triplef, d)$ by 
\begin{equation*}
\rmM^{[{pq}]}_{n}(\triplef, d)=\rmH^{3}(\rmX^{\natural3}_{d}\otimes{\QQ^{\ac}}, \calO_{\undlamb}(2))_{/\frakp^{[pq]}_{\triplef, n}}.
\end{equation*}
There is
an Abel--Jacobi map 
\begin{equation*}
\mathrm{AJ}^{[pq]}_{\triplef, n}: \mathrm{CH}^{2}(\rmX^{\natural3}_{d})\rightarrow \rmH^{1}(\QQ,  \rmM^{[pq]}_{n}(\triplef,d)(-1))
\end{equation*}
for $\rmM^{[pq]}_{n}(\triplef,d)(-1)$ constructed similarly for that of $\rmM_{n}(\triplef, d)$. We denote by 
\begin{equation*}
\Theta^{[pq]}_{n}(\triplef, d)\in \rmH^{1}(\QQ,  \rmM^{[pq]}_{n}(\triplef, d)(-1)) 
\end{equation*}
the image of the diagonal cycle $\Delta^{\natural}_{d}=\theta^{\natural}_{*}[\rmX^{\natural}_{d}]$ under $\mathrm{AJ}^{[pq]}_{\triplef, n}$. The main results of our previous work \cite[Theorem 4.7]{Wang} provide an isomorphism 
\begin{equation*}
\threesum(\threetensor\Gamma(\rmZ_{d}(\overline{B}),\calO_{\lambda_{i}})_{/\frakp^{[p]}_{i, n}})\cong\rmH^{1}_{\sing}(\QQ_{q},  \rmM^{[pq]}_{n}(\triplef, d)(-1))
\end{equation*}
under which
\begin{equation}\label{rami-reci}
 (\partial^{(j)}_{q}\Theta^{[pq]} _{n}(\triplef, d), \phi_{1}\otimes \phi_{2}\otimes \phi_{3})=(q+1)^{3}\sum_{z\in \Delta_{d}(\overline{B})} \phi_{1}(z)\otimes \phi_{2}(z)\otimes \phi_{3}(z)
 \end{equation} 
 for any $\phi_{1}\otimes \phi_{2}\otimes \phi_{3}\in \threetensor\Gamma(\rmZ_{d}(\overline{B}),\calO_{\lambda_{i}})[\frakp^{[p]}_{i, n}]$ and $j\in\{1, 2, 3\}$. Here $\partial_{q}\Theta^{[pq]} _{n}(\triplef, d)$ is the singular residue of $\Theta^{[pq]} _{n}(\triplef, d)$ at $q$ and $\partial^{(j)}_{q}\Theta^{[pq]} _{n}(\triplef, d)$ is the $j$-th component of $\partial_{q}\Theta^{[pq]} _{n}(\triplef, d)$ in $\threesum(\threetensor\Gamma(Z_{d}(\overline{B}),\calO_{\lambda_{i}}){/\frakp^{[p]}_{i, n}})$. And we define the {Gross--Kudla period} $\mathbf{I}(\phi_{1}, \phi_{2}, \phi_{3})$  by 
 \begin{equation*}
 \mathbf{I}(\phi_{1}, \phi_{2}, \phi_{3})=\sum_{z\in \Delta_{d}(\overline{B})} \phi_{1}(z)\otimes \phi_{2}(z)\otimes \phi_{3}(z).
 \end{equation*}
Combing Theorem \ref{unram-intro} and \eqref{rami-reci}, we arrive at the following equation
\begin{equation*}
\begin{aligned}
(\partial^{(j)}_{q}\Theta^{ [pq]} _{n}(\triplef, d), \phi_{1}\otimes \phi_{2}\otimes \phi_{3})&=(q+1)^{3}\mathbf{I}(\phi_{1}, \phi_{2}, \phi_{3})=(q+1)^{3}(\loc^{(j)}_{p}(\Theta_{n}(\triplef, d)), \phi_{1}\otimes\phi_{2}\otimes\phi_{3}).\\
\end{aligned}
\end{equation*}
This shows that $(\Theta_{n}(\triplef, d), \Theta^{[pq]}_{n}(\triplef, d))$ along with the Gross--Kudla periods $\mathbf{I}(\phi_{1}, \phi_{2}, \phi_{3})$ form a {bipartite Euler system} in a weaker sense: these classes do satisfy all the required reciprocity laws, however the singular and the finite part are both of rank $3$ as opposed to of rank $1$ as required by the definition in \cite{How}.

\subsection{The symmetric cube motive}  Consider the case when $\triplef=(f, f, f)$ for a single newform $f\in S^{\mathrm{new}}_{2}(\Gamma_{0}(N))$ with $N=N^{+}N^{-}$ such that $(N^{+}, N^{-})=1$ and $N^{-}$ is square-free with {even} number of prime factors. Let $\rmV_{f, \lambda}=\rmV$ be the representation space of  the Galois representation $\rho_{f, \lambda}$ attached to $f$ and $\rmV({\triplef})=\rmV^{\otimes 3}$ be the triple tensor product representation. Then we have the following factorization
\begin{equation*}
\rmV({\triplef})(-1)=\mathrm{Sym}^{3}\rmV(-1)\oplus \rmV\oplus \rmV
\end{equation*}
and we refer to $\rmV^{\diamond}(\triplef)(-1):=\mathrm{Sym}^{3}\rmV(-1)$ as the symmetric cube component of $\rmV({\triplef})(-1)$.  The triple product $L$-function $L(f\otimes f\otimes f, s)$ factors accordingly as 
\begin{equation*}
L(f\otimes f\otimes f, s)=L(\mathrm{Sym}^{3}f, s)L(f, s-1)^{2}.
\end{equation*}
 We project the class  $\Theta(\triplef, d)\in \rmH^{1}(\QQ,  \rmV(\triplef)(-1))$ which is the image of the diagonal cycle $\Delta_{d}=\theta_{*}[\rmX_{d}]$ under the Abel-Jacobi map
\begin{equation*}
 \mathrm{AJ}_{\triplef, \QQ}: \mathrm{CH}^{2}(\rmX^{3}_{d})\rightarrow \rmH^{1}(\QQ,  \rmV(\triplef)(-1))
\end{equation*}
to its symmetic cube component and we obtain thus a class $\Theta^{\diamond}(\triplef, d)\in \rmH^{1}(\QQ, \rmV^{\diamond}({\triplef})(-1))$. In light of the conjectural Gross--Zagier formula for the triple product $L$-function and the conjectural injectivity of the Abel--Jacobi map in this setting, the class  $\Theta^{\diamond}(\triplef, d)$ should be considered as an algebraic incarnation of the first derivative $L^{\prime}(\mathrm{Sym}^{3}f, s)$ at $s=2$ and $L(f, 1)$ is non-vanishing.  Using the reciprocity laws proved in this article and in \cite{Wang}, we prove the following theorem towards the rank $1$ case of the Bloch--Kato conjecture for the symmetric cube motive of the modular form $f$ at the end of this article. 
\begin{thm}\label{intro-rank-1}
Suppose that the modular form $f$ satisfies the following assumptions.
\begin{enumerate}
\item The residual Galois representation $\bar{\rho}_{f, \lambda}\vert_{\rmG_{\QQ(\zeta_{\ell})}}$ is absolutely irreducible;
\item The residual Galois representation $\bar{\rho}_{f, \lambda}$ is minimal at primes in $\Sigma^{+}\cup\Sigma^{-}_{\ram}$. Moreover $\bar{\rho}_{f, \lambda}$ is ramified at primes in $\Sigma^{-}_{\ram}$;
\item The image of $\overline{\rho}_{f, \lambda}$ contains $\GL_{2}(\FF_{\ell})$.
\end{enumerate}
If the class $\Theta^{\diamond}(\triplef,d)\in \rmH^{1}(\QQ, \rmV^{\diamond}({\triplef})(-1))$ is non-zero, then the symmetric cube Bloch--Kato Selmer group 
\begin{equation*}
\rmH^{1}_{f}(\QQ, \mathrm{V}^{\diamond}(\triplef)(-1)) 
\end{equation*}
is of dimension $1$ over $E_{\lambda}$.
\end{thm}

\subsection{Notations and conventions} We will use common notations and conventions in algebraic number theory and algebraic geometry. The cohomologies appeared in this article will be understood as the \'{e}tale cohomologies. For a field $K$, we denote by $K^{\ac}$ the separable closure of $K$ and put $\rmG_{K}:=\Gal(K^{\ac}/K)$ the absolute Galois group of $K$. We let $\Adel$ be the ring of ad\`{e}les over $\QQ$ and $\Adel^{\infty}$ be the subring of finite ad\`{e}les.  For a prime $p$, $\Adel^{(\infty p)}$ is the prime-to-$p$ part of  $\Adel^{\infty}$. 

When $K$ is a local field, we denote by $\calO_{K}$ its valuation ring  and by $k$ its residue field. We let $\rmI_{K}$ be the inertia subgroup of $\rmG_{K}$. For a $\rmG_{K}$-module $\rmM$,  we have the following exact sequence of Galois cohomology groups
\begin{equation}\label{fin-sing}
0\rightarrow \rmH^{1}_{\mathrm{fin}}(K, \rmM)\rightarrow \rmH^{1}(K, \rmM)\xrightarrow{\partial} \rmH^{1}_{\sing}(K, \rmM)\rightarrow 0  
\end{equation}
where  $\rmH^{1}_{\mathrm{fin}}(K, \rmM)=\rmH^{1}(k, \rmM^{\rmI_{K}})$
is called the \emph{unramified} or the \emph{finite} part of the cohomology group $\rmH^{1}(K, \rmM)$ and $\rmH^{1}_{\mathrm{sin}}(K, \rmM)$ is defined as the quotient of $\rmH^{1}(K, \rmM)$ by its finite part is called the \emph{singular quotient} of  $\rmH^{1}(K, \rmM)$. The natural quotient map $\rmH^{1}(K, \rmM)\xrightarrow{\partial} \rmH^{1}_{\sing}(K, \rmM)$ will be referred to as the \emph{singular quotient map}. The element $\partial(x)$ will be referred to as the \emph{singular residue} of $x$ for  $x\in \rmH^{1}(K, \rmM)$. Let $K$ be a number field and $v$ be a place of $K$, suppose $\rmM$ is a $\rmG_{K}$-module, then we call the natural map $\mathrm{loc}_{v}: \rmH^{1}(K, \rmM)\rightarrow \rmH^{1}(K_{v}, \rmM)$ the localization map. The composition of the localization map $\mathrm{loc}_{v}$ with the singular quotient map $\partial$ will be denoted by $\partial_{v}$. The symbol ${\emptyset}$ on the subscript or the superscript means we omit the subscript or the superscript. For example $\rmM_{\emptyset}=\rmM$.

\section{Review of weight spectral sequence}
\subsection{Nearby cycles on semi-stable schemes} Let $K$ be a henselian discrete valuation field with residue field $k$ of characteristic $p$ and valuation ring $\calO_{K}$. We fix a uniformizer $\pi$ of $\calO_{K}$. We set $S=\Spec\calO_{K}$, $s=\Spec k$ and $\eta=\Spec K$. Let $K^{\ac}$ be a separable closure of $K$ and $K_{\ur}$ the maximal unramified extension of $K$ in $K^{\ac}$. We denote by $k^{\ac}$ the residue field of $K_{\ur}$. 
Let $\rmI_{K}=\Gal(K^{\ac}/K_{\ur})\subset \rmG_{K}=\Gal(K^{\ac}/K)$ be the inertia group. Let $\ell$ be a prime different from $p$. We set $t_{\ell}: \rmI_{K}\rightarrow \ZZ_{\ell}(1)$ to be the canonical map given by 
\begin{equation*}
\sigma \mapsto (\sigma(\pi^{1/\ell^{m}})/\pi^{1/\ell^{m}})_{m}
\end{equation*} 
for every $\sigma\in \rmI_{K}$. 

Let $\mathfrak{X}$ be a \emph{strict semi-stable scheme} over $S$ purely of relative dimension $n$ which is also assumed to be proper. This means that $\mathfrak{X}$ is locally of finite presentation and Zariski locally \'{e}tale over $$\Spec(\calO_{K}[X_{1}, \cdots, X_{n}]/(X_{1}\cdots X_{r}-\pi))$$ for some integer $1\leq r\leq n$. We let $\overline{\rmX}$ be the special fiber of $\mathfrak{X}$ and $\overline{\rmX}\otimes{k^{\ac}}$ be its base-change to $k^{\ac}$. Let $\rmX=\mathfrak{X}_{\eta}$ be the generic fiber of $\mathfrak{X}$ and $\rmX\otimes{K_{\ur}}$ be its base-change to $K_{\ur}$. Let $i:\overline{\rmX}\rightarrow \mathfrak{X}$, $j: \rmX\rightarrow \mathfrak{X}$ be the natural inclusions of the special fiber and the generic fiber. Let $\bar{i}: \overline{\rmX}\otimes{k^{\ac}}\rightarrow \mathfrak{X}\otimes{\calO_{K_{\ur}}}$ and $\bar{j}: \rmX\otimes{K_{\ur}}\rightarrow \mathfrak{X}\otimes{\calO_{K_{\ur}}}$ be the base change of $i$ and $j$ to $\calO_{K^{\ur}}$. 

Let $\Lambda$ be a finite extension of $\ZZ_{\ell}$ or $\ZZ/\ell^{v}$. We have the  nearby cycle sheaf 
\begin{equation*}
\rmR^{q}\Psi(\Lambda)= \bar{i}^{*}\rmR^{q}\bar{j}_{*}\Lambda
\end{equation*}
and the {nearby cycle complex} 
\begin{equation*}
\rmR\Psi(\Lambda)= \bar{i}^{*}\rmR\bar{j}_{*}\Lambda.
\end{equation*}
We regard the latter as an object in the derived category $\rmD^{+}(\rmXbar\otimes{k^{\ac}}, \Lambda[\rmI_{K}])$ of sheaves of $\Lambda$-modules with continuous $\rmI_{K}$-actions. By the proper base change theorem, we always have an isomorpshim
\begin{equation*}
\rmH^{*}(\rmXbar\otimes{k^{\ac}}, \rmR\Psi(\Lambda))\cong\rmH^{*}(\rmX\otimes{K^{\ac}}, \Lambda).
\end{equation*} 

Let $\rmD_{1},\cdots, \rmD_{m} $ be the set of irreducible components of $\overline{\rmX}$. For each index set $I\subset \{1, \cdots, m\}$ of cardinality $p$, we set $\overline{\rmX}_{I}=\cap_{i\in I} \rmD_{i}$. This is a smooth scheme of dimension $n-p$. For $1\leq p \leq m-1$, let
\begin{equation}
\overline{\rmX}^{(p)}=\bigsqcup_{I\subset \{1, \cdots, m\}, \mathrm{Card}(I)=p+1} \overline{\rmX}_{I}
\end{equation}
 and 
 \begin{equation}
 a_{p}: \overline{\rmX}^{(p)}\rightarrow \overline{\rmX}
 \end{equation}
 be the projection, we have $a_{p *}\Lambda=\wedge^{p+1}a_{0 *}\Lambda$. Consider the Kummer exact sequence in the case $\Lambda=\ZZ/\ell^{v}$
 \begin{equation}
 0\rightarrow \Lambda(1)\rightarrow \calO^{\times}_{\overline{\rmX}} \rightarrow \calO^{\times}_{\overline{\rmX}}\rightarrow 0.
 \end{equation}
 Let $\partial(\pi)\in i^{*}\rmR^{1}j_{*}\Lambda(1) $ be the image of $\pi$ under the coboundary map by applying $i^{*}\rmR j_{*}$ to the above exact sequence. We let $\theta: \Lambda_{\rmXbar}\rightarrow i^{*}\rmR^{1}j_{*}\Lambda(1)$ be the map sending $1$ to $\partial(\pi)$ and $\delta:\Lambda_{\rmXbar}\rightarrow a_{0*}\Lambda$ be the canonical map. 
\begin{proposition}\label{nearby-cycle-q}
We have the following statements.
\begin{enumerate}
\item There is an isomorphism of exact sequences 
\begin{equation*}
\begin{tikzcd}
\Lambda_{\rmXbar}\arrow{r}{\delta}\arrow{d}&a_{0*}\Lambda\arrow{r}{\delta\wedge}\arrow{d}&\cdots\arrow{r}{\delta\wedge}\arrow{d}&a_{n*}\Lambda\arrow{r}\arrow{d}&0\\
\Lambda_{\rmXbar}\arrow{r}{\theta}&i^{*}\rmR^{1}j_{*}\Lambda(1)\arrow{r}{\theta\cup}&\cdots\arrow{r}{\theta\cup}&i^{*}\rmR^{n+1}j_{*}\Lambda(n+1)\arrow{r}&0.\\
\end{tikzcd}
\end{equation*}
\item For $p\geq 0$, we have an exact sequence
\begin{equation*}
\begin{tikzcd}
\rmR^{p}\Psi(\Lambda)\arrow{r}{\theta\cup}&i^{*}\rmR^{p+1}j_{*}\Lambda(1)\arrow{r}{\theta\cup}&\cdots\arrow{r}{\theta\cup}&i^{*}\rmR^{n+1}j_{*}\Lambda(n+1-p)\arrow{r}&0.\\
\end{tikzcd}
\end{equation*}
\item For $p\geq 0$, we have a quasi-isomorphism of complexes
\begin{equation*}
\rmR^{p}\Psi(\Lambda)[-p]\xrightarrow{\sim} [a_{p*}\Lambda(-p)\xrightarrow{\delta\wedge}\cdots\xrightarrow{\delta\wedge} a_{n*}\Lambda(-p)\rightarrow0].
\end{equation*}
\end{enumerate}
\end{proposition}
\begin{proof}
The first two statements are taken from  \cite[Corollary 1.3]{Saito} and the last one is an immediate consequence of the first two. 
\end{proof}
\subsection{Monodromy filtration and spectral sequence} Suppose $A$ is an object in an abelian category and $N$ is a nilpotent endomorphism on $A$. We define two filtrations on $A$:
\begin{itemize}
\item  the kernel filtration $\rmF_{\bullet}$ by putting $\rmF_{p}A=\ker(N^{p+1}: A\rightarrow A)$ for $p\geq 0$;
\item  the image filtration $\rmG^{\bullet}$ by putting $\rmG^{q}A=\im(N^{q}: A\rightarrow A)$ for $q>0$.
\end{itemize}
Using these two filtrations,  we define the convolution filtration $\rmM_{\bullet}A$ by 
\begin{equation*}
\rmM_{r}A=\oplus_{p-q=r}\rmF_{p}A\cap \rmG^{q}A.
\end{equation*} 

To calculate the graded piece of the convolution filtration, we define the induced $\rmG^{\bullet}$-filtration on the graded piece of $\rmF_{\bullet}$ by $\rmG^{q}\Gr^{\rmF}_{p}A=\im(\rmG^{q}A\cap \rmF_{p}A\rightarrow \Gr^{\rmF}_{p}A)$.  Then we see that the $r$-th graded piece of the convolution filtration $\rmM_{\bullet}A$ is given by
\begin{equation*}
\Gr^{\rmM}_{r}A\cong\bigoplus_{p-q=r}\Gr^{q}_{\rmG}\Gr^{\rmF}_{p}A.
\end{equation*}
The convolution filtration in this case is known as the \emph{Monodromy filtration} and it is characterized by
\begin{enumerate}
\item $\rmM_{-n}A=A$ and $\rmM_{n+1}A=0$.
\item $N: A\rightarrow A$ sends $\rmM_{r}A$ into $\rmM_{r+2}A$ for $r\in\ZZ$.
\item $N^{r}: \Gr^{\rmM}_{-r}A\rightarrow \Gr^{\rmM}_{r}A$ is an isomoprhism. 
\end{enumerate}

Let $A=\rmR\Psi(\Lambda)$ now. Let $T$ be an element in $\rmI_{K}$ such that $t_{l}(T)$ is a generator of $\Lambda(1)$ then $T$ induces a nilpotent operator $T-1$ on $\rmR\Psi(\Lambda)$. Let $N=(T-1)\otimes\breve{T}$ where $\breve{T}\in \Lambda(-1)$ be the dual of $t_{l}(T)$. Then with respect to $N$, we have the following characterization of the Monodromy filtration on $\rmR\Psi(\Lambda)$
\begin{enumerate}
\item $\rmM_{-n}\rmR\Psi(\Lambda)=\rmR\Psi(\Lambda)$ and $\rmM_{n+1}\rmR\Psi(\Lambda)=0$.
\item $N: \rmR\Psi(\Lambda)(1)\rightarrow \rmR\Psi(\Lambda)$ sends $\rmM_{r}\rmR\Psi(\Lambda)(1)$ into $\rmM_{r+2}\rmR\Psi(\Lambda)$ for $r\in\ZZ$.
\item $N^{r}: \Gr^{\rmM}_{-r}\rmR\Psi(\Lambda)(r)\rightarrow \Gr^{M}_{r}\rmR\Psi(\Lambda)$ is an isomoprhism. 
\end{enumerate}
We can use Proposition \ref{nearby-cycle-q} to calculate the monodromy filtration on $\rmR\Psi(\Lambda)$. In addition to Proposition \ref{nearby-cycle-q}, we need the following results in \cite[Lemma 2.5, Corollary 2.6]{Saito}. 
\begin{enumerate}
\item The kernel filtration $\rmF_{p}\rmR\Psi(\Lambda)$ is given by the canonical truncated filtration $\tau_{\leq p}\rmR\Psi(\Lambda)$ and therefore 
\begin{equation*}
\Gr^{F}_{p}\rmR\Psi(\Lambda)\cong \rmR^{p}\Psi(\Lambda)[-p]\cong [a_{p*}\Lambda(-p)\xrightarrow{\delta\wedge}\cdots\xrightarrow{\delta\wedge} a_{n*}\Lambda(-p)\rightarrow0].
\end{equation*}
\item The image filtration $\rmG^{q}\Gr^{\rmF}_{p}\rmR\Psi(\Lambda)$ on $\Gr^{\rmF}_{p}\rmR\Psi(\Lambda)$ is given by the truncation in Proposition \ref{nearby-cycle-q} $(3)$ to $p+q$ position
\begin{equation*}
\rmG^{q}\Gr^{F}_{p}\rmR\Psi(\Lambda)= [a_{p+q*}\Lambda(-p)\xrightarrow{\delta\wedge}\cdots\xrightarrow{\delta\wedge} a_{n*}\Lambda(-p)\rightarrow0].
\end{equation*}
\item Combining the above two results, we have 
\begin{equation*}
\Gr^{q}_{\rmG}\Gr^{\rmF}_{p}\rmR\Psi(\Lambda)=a_{p+q*}\Lambda[-p-q](-p)
\end{equation*}
and therefore we arrive at the following equation 
\begin{equation}\label{Graded-M}
\Gr^{\rmM}_{r}\rmR\Psi(\Lambda)=\bigoplus_{p-q=r}a_{p+q*}\Lambda[-p-q](-p).
\end{equation}
\end{enumerate}
The monodromy filtration induces the \emph{Rapoport--Zink spectral sequence} \cite{RZ} or the \emph{weight spectral sequence}
\begin{equation}\label{wt-seq}
\rmE^{p,q}_{1}= \rmH^{p+q}(\rmXbar\otimes{k^{\ac}}, \Gr^{\rmM}_{-p}\rmR\Psi(\Lambda))\Rightarrow \rmH^{p+q}(\rmXbar\otimes{k^{\ac}}, \rmR\Psi(\Lambda))\cong\rmH^{p+q}(\rmX\otimes{K^{\ac}}, \Lambda).
\end{equation}
The $\rmE_{1}$-term of this spectral sequence can be made explicit by
\begin{equation*}
\begin{aligned}
\rmH^{p+q}(\rmXbar\otimes{k^{\ac}}, \Gr^{\rmM}_{-p}\rmR\Psi(\Lambda))&=\bigoplus_{i-j=-p, i\geq0, j\geq0}\rmH^{p+q-(i+j)}(\rmXbar^{(i+j)}\otimes{k^{\ac}}, \Lambda(-i))\\
&=\bigoplus_{i\geq\mathrm{max}(0, -p)}\rmH^{q-2i}(\rmXbar^{(p+2i)}\otimes{k^{\ac}}, \Lambda(-i)).\\
\end{aligned}
\end{equation*}
We call the induced filtration on $\rmH^{\ast}(\rmXbar\otimes{k^{\ac}}, \rmR\Psi(\Lambda))$ by this spectral sequence the \emph{monodromy filtration} of $\rmH^{\ast}(\rmXbar\otimes{k^{\ac}}, \rmR\Psi(\Lambda))$.

\subsection{Examples in dimension $1$ and $3$} We will make the first page of the weight spectral sequence explicit in dimension $1$ and dimension $3$ which are the only cases that will be used in the computations later. To shorten the notation, we will sometimes write $\rmH^{*}(a_{p*}\Lambda)$ instead of 
\begin{equation*}
\rmH^{*}(\rmXbar\otimes{k^{\ac}}, a_{p*}\Lambda)=\rmH^{*}(\rmXbar^{(p)}\otimes{k^{\ac}}, \Lambda).
\end{equation*}
\subsubsection{The one dimensional case}\label{1-dim}\label{one-mono}: Let $\mathfrak{X}$ be a relative curve over $\Spec(\calO_{K})$. Then we immediately calculate that
\begin{equation}
\begin{aligned}
&\Gr^{\rmM}_{-1}\rmR\Psi(\Lambda)=a_{1*}\Lambda[-1], \\
&\Gr^{\rmM}_{0}\rmR\Psi(\Lambda)=a_{0*}\Lambda,  \\
&\Gr^{\rmM}_{1}\rmR\Psi(\Lambda)=a_{1*}\Lambda[-1](-1). \\
\end{aligned}
\end{equation}
The $\rmE_{1}$-page of the weight spectral sequence is given by
\begin{center}
\begin{tikzpicture}[thick,scale=0.9, every node/.style={scale=0.9}]
  \matrix (m) [matrix of math nodes,
    nodes in empty cells,nodes={minimum width=5ex,
    minimum height=5ex,outer sep=-5pt},
    column sep=1ex,row sep=1ex]{
                &      &     &     & \\
          2     &  \rmH^{0}(a_{1*}\Lambda)(-1) &  \rmH^{2}(a_{0*}\Lambda)  & & \\
          1     &       & \rmH^{1}(a_{0*}\Lambda) &    & \\
          0     &    & \rmH^{0}(a_{0*}\Lambda) &  \rmH^{0}(a_{1*}\Lambda) &\\
    \quad\strut &   -1  &  0  &  1  & \strut \\};
\draw[thick] (m-1-1.east) -- (m-5-1.east) ;
\draw[thick] (m-5-1.north) -- (m-5-5.north) ;
\end{tikzpicture}
\end{center}
and it clearly degenerates at the $\rmE_{2}$-page.  We therefore have the monodromy filtration 
\begin{equation*}
0\subset^{\rmE^{1,0}_{2}} \rmM_{1}\rmH^{1}(\rmXbar\otimes{k^{\ac}}, \rmR\Psi(\Lambda))\subset^{\rmE^{0,1}_{2}} \rmM_{0}\rmH^{1}(\rmXbar\otimes{k^{\ac}}, \rmR\Psi(\Lambda))\subset^{\rmE^{-1,2}_{2}} \rmM_{-1}\rmH^{1}(\rmXbar\otimes{k^{\ac}},\Psi(\Lambda))
\end{equation*}
with graded pieces given by
\begin{equation}\label{curve-grade}
\begin{aligned}
&\Gr^{\rmM}_{-1}\rmH^{1}(\rmXbar\otimes{k^{\ac}}, \rmR\Psi(\Lambda))=\ker[\rmH^{0}(a_{1*}\Lambda(-1))\xrightarrow{\tau} \rmH^{2}(a_{0*}\Lambda)]\\
&\Gr^{\rmM}_{0}\rmH^{1}(\rmXbar\otimes{k^{\ac}}, \rmR\Psi(\Lambda))= \rmH^{1}(a_{0*}\Lambda)\\
&\Gr^{\rmM}_{1}\rmH^{1}(\rmXbar\otimes{k^{\ac}}, \rmR\Psi(\Lambda))= \coker[\rmH^{0}(a_{0*}\Lambda)\xrightarrow{\rho} \rmH^{0}(a_{1*}\Lambda)]\\
\end{aligned}
\end{equation}
where $\tau$ is the generalized {Gysin morphism} and $\rho$ is the generalized {restriction morphism} as in \cite[p. 610]{Saito}. 
Note that we have the following commutative diagram
\begin{equation}\label{picard-lef}
\begin{tikzcd}
\rmH^{1}(\rmXbar\otimes{k^{\ac}},\rmR\Psi (\Lambda(1))) \arrow[r] \arrow[d, "N"] & \ker[\rmH^{0}(a_{1*}\Lambda)\xrightarrow{\tau} \rmH^{2}(a_{0*}\Lambda)(1)] \arrow[d, "N"] \\
\rmH^{1}(\rmXbar\otimes{k^{\ac}}, \rmR\Psi(\Lambda))                  & \coker[\rmH^{0}(a_{0*}\Lambda)\xrightarrow{\rho} \rmH^{0}(a_{1*}\Lambda)] .\arrow[l]
\end{tikzcd}
\end{equation}
In this case, we recover the \emph{Picard--Lefschetz formula} if we identify $\rmH^{0}(a_{1*}\Lambda)(-1)$ with the space of \emph{vanishing cycles} $\underset{x}\oplus\rmR^{1}\Phi_{\{x\}}(\Lambda)$
on $\rmXbar\otimes{k^{\ac}}$ where $x$ runs through the set of singular points $\rmXbar^{(1)}$. 

We explain another way of calculating the monodromy filtration $\rmM_{\bullet}\rmH^{1}(\rmXbar\otimes{k^{\ac}},\rmR\Psi(\Lambda))$
in this case. Consider the distinguished triangle
\begin{equation*}
\Lambda\xrightarrow{sp}\rmR\Psi(\Lambda)\rightarrow \rmR\Phi(\Lambda)\xrightarrow{[+1]} 
\end{equation*}
where $\rmR\Phi(\Lambda)$ is the complex of vanishing cycles and its induced long exact sequence
\begin{equation*}
\rmH^{0}(\rmXbar\otimes k^{\ac}, \rmR\Phi(\Lambda))\rightarrow \rmH^{1}(\rmXbar\otimes k^{\ac},\Lambda)\rightarrow \rmH^{1}(\rmXbar\otimes k^{\ac},\rmR\Psi(\Lambda))\rightarrow \rmH^{1}(\rmXbar\otimes k^{\ac},\rmR\Phi(\Lambda))\rightarrow \rmH^{2}(\rmXbar\otimes k^{\ac}, \Lambda).
\end{equation*}
Since it is well-known that $\rmR\Phi(\Lambda)$ is concentrated in degree one \cite[2.2.5 A.2]{SGA7}, we obtain the exact sequence
\begin{equation*}
0\rightarrow \rmH^{1}(\rmXbar\otimes k^{\ac},\Lambda)\rightarrow \rmH^{1}(\rmXbar\otimes k^{\ac},\rmR\Psi(\Lambda))\rightarrow \rmH^{1}(\rmXbar\otimes k^{\ac},\rmR\Phi(\Lambda))\rightarrow \rmH^{2}(\rmXbar\otimes k^{\ac}, \Lambda).
\end{equation*}
We have a distinguished triangle 
\begin{equation*}
\Lambda\rightarrow a_{0\ast}\Lambda\rightarrow a_{1\ast}\Lambda\xrightarrow{[+1]} 
\end{equation*}
on the \'etale site of $\rmXbar\otimes k^{\ac}$ by Proposition \ref{nearby-cycle-q} (1). Then we obtain from the induced long exact sequence the following exact sequence
\begin{equation}\label{normaliz-exact}
\rmH^{0}(\rmXbar\otimes k^{\ac}, a_{0\ast}\Lambda)\rightarrow \rmH^{0}(\rmXbar\otimes k^{\ac}, a_{1\ast}\Lambda)\rightarrow \rmH^{1}(\rmXbar\otimes k^{\ac}, \Lambda)\rightarrow\rmH^{1}(\rmXbar\otimes k^{\ac}, a_{0\ast}\Lambda)\rightarrow 0
\end{equation}
and an isomorphism $\rmH^{2}(\rmXbar\otimes k^{\ac}, \Lambda)\cong \rmH^{2}(\rmXbar\otimes k^{\ac}, a_{0\ast}\Lambda)$.
Then by comparing with \ref{curve-grade}, it is clear that 
\begin{equation*}
\rmH^{1}(\rmXbar\otimes{k^{\ac}}, \Lambda)=\rmM_{0}\rmH^{1}(\rmXbar\otimes k^{\ac},\rmR\Psi(\Lambda))
\end{equation*}
and we recover the monodromy filtration from the above discussions.  

\subsubsection{Three dimensional case} Let $\mathfrak{X}$ be a relative threefold over $\Spec(\calO_{K})$. We can also easily list the graded pieces of the monodromy filtration on $\rmR\Psi(\Lambda)$ using \ref{Graded-M}
\begin{equation*}
\begin{aligned}
&\Gr^{\rmM}_{-3}\rmR\Psi(\Lambda)=a_{3*}\Lambda[-3], \\
&\Gr^{\rmM}_{-2}\rmR\Psi(\Lambda)=a_{2*}\Lambda[-2], \\
&\Gr^{\rmM}_{-1}\rmR\Psi(\Lambda)=a_{1*}\Lambda[-1]\oplus a_{3*}\Lambda[-3](-1), \\
&\Gr^{\rmM}_{0}\rmR\Psi(\Lambda)=a_{0*}\Lambda\oplus a_{2*}\Lambda[-1](-1),  \\
&\Gr^{\rmM}_{1}\rmR\Psi(\Lambda)=a_{1*}\Lambda[-1](-1)\oplus a_{3*}\Lambda[-3](-2),\\
&\Gr^{\rmM}_{2}\rmR\Psi(\Lambda)=a_{2*}\Lambda[-2](-2), \\
&\Gr^{\rmM}_{3}\rmR\Psi(\Lambda)=a_{3*}\Lambda[-3](-3).\\
\end{aligned}
\end{equation*}
The $\rmE_{1}$-page of the weight spectral sequence is given below

\begin{equation}\label{E1-primitive}
\begin{tikzpicture}[thick,scale=0.65, every node/.style={scale=0.65}]
\matrix (m) [matrix of math nodes,
    nodes in empty cells,nodes={minimum width=5ex,
    minimum height=5ex,outer sep=-5pt},
    column sep=1ex,row sep=1ex]{
                &      &     &     & \\
          6    &\rmH^{0}(a_{3*}\Lambda)(-3) &  \rmH^{2}(a_{2*}\Lambda)(-2)  &\rmH^{4}(a_{1*}\Lambda)(-1) & \rmH^{6}(a_{0*}\Lambda) \\
          5    &    &\rmH^{1}(a_{2*}\Lambda)(-2)  &\rmH^{3}(a_{1*}\Lambda)(-1) &\rmH^{5}(a_{0*}\Lambda)& \\
          4    &       &\rmH^{0}(a_{2*}\Lambda)(-2) &\rmH^{2}(a_{1*}\Lambda)(-1)\oplus\rmH^{0}(a_{3*}\Lambda)(-2) &\rmH^{4}(a_{0*}\Lambda)\oplus\rmH^{2}(a_{2*}\Lambda)(-1)&\rmH^{4}(a_{1*}\Lambda)\\
          3   &        &         &\rmH^{1}(a_{1*}\Lambda)(-1)  &\rmH^{3}(a_{0*}\Lambda)\oplus\rmH^{1}(a_{2*}\Lambda)(-1) &\rmH^{3}(a_{1*}\Lambda)\\
          2   &        &         &\rmH^{0}(a_{1*}\Lambda)(-1)  &\rmH^{2}(a_{0*}\Lambda)\oplus\rmH^{0}(a_{2*}\Lambda)(-1) &\rmH^{2}(a_{1*}\Lambda)\oplus \rmH^{0}(a_{3*}\Lambda)(-1) &\rmH^{2}(a_{2*}\Lambda)\\
           1   &        &         &                             &\rmH^{1}(a_{0*}\Lambda) &\rmH^{1}(a_{1*}\Lambda) &\rmH^{1}(a_{2*}\Lambda)\\
           0   &        &         &       &\rmH^{0}(a_{0*}\Lambda) &\rmH^{0}(a_{1*}\Lambda) &\rmH^{0}(a_{2*}\Lambda)& \rmH^{3}(a_{3*}\Lambda)\\
          \quad\strut & -3  &  -2  &  -1  & 0 &1 &2 &3 & \strut \\};

\draw[thick] (m-1-1.east) -- (m-9-1.east) ;
\draw[thick] (m-9-1.north) -- (m-9-9.north) ;
\end{tikzpicture}
\end{equation}
and this spectral sequence will be used in \S 5.

\section{Arithmetic level raising for Shimura curves}
\subsection{Shimura curves and Shimura sets}
Let $N=N^{+}N^{-}$ be a factorization of $N$ such that $N^{-}$ is square-free and has \emph{even} number of prime divisors. Let $p$ be a prime away from $N$. 
Let $B=B_{N^{-}}$ be the indefinite quaternion algebra over $\QQ$ with discriminant $N^{-}$. Let $\calO_{B}$ be a maximal order of $B$. For an integer $M$ divisible by $N^{+}$ and relatively prime to $N^{-}$, let $\calO_{B, M}$ be an Eichler order of level $M$. Let $\rmG$ be the algebraic group over $\QQ$ defined by $B^{\times}$. Let $m$ be a square free integer such that $(m, pN)=1$, we define the open compact subgroup $K_{M, m}\subset \rmG(\Adel^{\infty})$ by
\begin{equation*}
K_{M, m}= \{g=(g_{v})_{v}\in \widehat{\calO}^{\times}_{B, M}: g_{v}\equiv \begin{pmatrix} \ast & \ast\\ 0& 1\\\end{pmatrix}\mod v \text{ for } v\mid m\}.
\end{equation*}
When $m=1$, then $K_{M, m}$ will simply be denoted by $K_{M}$. There is a Shimura curve $\rmX_{m}$ over $\QQ$ whose complex points are are uniformized by the double coset space
\begin{equation*}
\rmX_{m}(\CC)=\rmG(\QQ)\backslash \calH^{\pm} \times \rmG(\Adel^{\infty})/K_{N^{+},m}.
\end{equation*}
If $m=1$, we simply denote $\rmX_{m}$ by $\rmX$.

We define an integral model $\interX_{m}$ of $\rmX_{m}$ over $\ZZ[1/Nm]$ which represents the following functor. 
\begin{definition}\label{Xm}
Let $S$ be a locally Noetherian  scheme over $\ZZ[1/Nm]$. Then $\interX_{m}(S)$ classifies the isomorphism classes of the tuples $(A, \iota, C_{N^{+}}, \alpha_{m})$ where
\begin{enumerate}
\item $A$ is an $S$-abelian scheme of relative dimension $2$;
\item $\iota: \calO_{B}\hookrightarrow \End_{S}(A)$ is an action of $\calO_{B}$ on $A$;
\item $C_{N^{+}}$ is a finite flat subgroup scheme of $A[N^{+}]$ of order $(N^{+})^{2}$ which is stable and locally cyclic under the action of $\calO_{B}$;
\item $\alpha_{m}: (\underline{\ZZ/m})^{2}_{S}\rightarrow A[m]$ is an $\calO_{B}$-equivariant injection of finite flat group schemes over $S$.
\end{enumerate} 
\end{definition}
This functor is representable by a smooth projective scheme denoted also by $\interX_{m}$ over $\ZZ[1/Nm]$ if $m\geq 4$. In particular, if we omit the data $\alpha_{m}$ from the above moduli problem, we will obtain a coarse moduli space $\interX$ over  $\ZZ[1/N]$ which defines an integral model of the Shimura curve $\rmX$. Suppose $m=d$ is a prime.
Let $\pi_{1,d}: \interX_{d}\rightarrow\interX$ be the degeneracy map given by
\begin{equation}\label{deg-d-1}
(A, \iota, C_{N^{+}}, \alpha_{d}) \mapsto (A, \iota, C_{N^{+}})
\end{equation}
and $\pi_{2,d}: \interX_{d}\mapsto\interX$ be the degeneracy map given by 
\begin{equation}\label{deg-d-2}
(A, \iota, C_{N^{+}}, \alpha_{d})\mapsto (A^{\prime}, \iota^{\prime}, C^{\prime}_{N^{+}}) 
\end{equation}
where $A^{\prime}$ is the quotient of $A$ by the subgroup scheme $C_{d}=\alpha_{d}((\underline{\ZZ/d})^{2})\subset A[d]$, $\iota^{\prime}$ is the induced action of $\calO_{B}$ on $A^{\prime}$ and $C^{\prime}_{N^{+}}=C_{N^{+}}\oplus C_{d}/C_{d}$.

In this article, we will often consider the base-change of $\interX_{m}$ to $\ZZ_{p^{2}}$ for some prime $p\nmid Nm$ and we will denote it by the same notation. It is well known that $\interX_{m}$ is smooth and projective of relative dimension $1$ over $\ZZ_{p^{2}}$. The generic fiber of $\interX_{m}$ will simply be denoted by $\rmX_{m}$ and its special fiber will be denoted by $\overline{\rmX}_{m}$. Let $x=(A, \iota,C_{N^{+}}, \alpha_{m})\in \overline{\rmX}_{m}$ be an $\FF^{\ac}_{p}$-point. Then the $p$-divisible group $A[p^{\infty}]$ of $A$ can be written as  $A[p^{\infty}]=E[p^{\infty}]\times E[p^{\infty}]$ for a $p$-divisible group $E[p^{\infty}]$ associated to an elliptic curve $E$ and $\calO_{B}$ acts naturally via $\calO_{B}\otimes \ZZ_{p}=\rmM_{2}(\ZZ_{p})$. Depending on $E[p^{\infty}]$ is {ordinary} or {supersingular}, we will accordingly call $x$ an ordinary or a supersingular point. Let $\overline{\rmX}^{\ss}_{m}$ be the closed sub-scheme given by those points that are supersingular and  $\overline{\rmX}^{\ord}_{m}=\overline{\rmX}_{m}-\overline{\rmX}^{\ss}_{m}$ be its complement. We will refer to $\overline{\rmX}^{\ss}_{m}$ as the {supersingular locus} and to $\overline{\rmX}^{\ord}_{m}$ as the {ordinary locus}.  Let $\overline{B}=B_{pN^{-}}$ be the definite quaternion algebra with discriminant $pN^{-}$ and $\calO_{\overline{B}}$ be a maximal order. We will write $\calO_{\overline{B}, p}=\calO_{\overline{B}}\otimes\ZZ_{p}$ and define $K_{p}=\calO^{\times}_{\overline{B}, p}$. Note that we can naturally view the prime-to-$p$ part  $K^{(p)}_{N^{+},m}$ of $K_{N^{+},m}$ as an open compact subgroup of $\overline{B}^{\times}(\Adel^{(\infty p)})$. The scheme $\overline{\rmX}^{\ss}_{m}$ is given by a finite set of points naturally defined over $\FF_{p^{2}}$ and we have the following parametrization of  $\overline{\rmX}^{\ss}_{m}$. 

\begin{lemma}
We have an isomorphism 
\begin{equation*}
\overline{\rmX}^{\ss}_{m}\cong \overline{B}^{\times}(\QQ)\backslash \overline{B}^{\times}(\Adel^{\infty})/ K_{N^{+},m}.
\end{equation*}
of schemes over $\FF_{p^{2}}$.
\end{lemma}
\begin{proof}
The lemma is well known and can be proved using essentially the same method of the classical work Deuring and Serre. See \cite[Lemma 9]{DT} for example.
\end{proof}
We will consider the Shimura set associated to the definite quaternion algebra $\overline{B}$ with level $K_{N^{+}, m}$ given by
\begin{equation*}
\rmZ_{m}(\overline{B})=\overline{B}^{\times}(\QQ)\backslash \overline{B}^{\times}(\Adel^{\infty})/ K_{N^{+}, m}.
\end{equation*}
The above lemma gives rise to an isomorphism $\overline{\rmX}^{\ss}_{m}\cong \rmZ_{m}(\overline{B})$ and thus gives rise to a moduli interpretation of this $0$-dimensional scheme. When $m=1$, then by convention, $K_{N^{+}, m}$ agrees with $K_{N^{+}}$ and we write $\rmZ_{1}(\overline{B})$ simply by $\rmZ(\overline{B})$. 

Suppose $m=d$ is a prime. Let $\overline{\pi}_{1,d}:\rmZ_{d}(\overline{B})\rightarrow \rmZ(\overline{B})$ be the degeneracy map induced by $\pi_{1, d}$ as in \ref{deg-d-1} and $\overline{\pi}_{2,d}:\rmZ_{d}(\overline{B})\rightarrow \rmZ(\overline{B})$ be the degeneracy map induced by $\pi_{2, d}$ as in \ref{deg-d-2}. See \cite[A.1.2]{Liu-cubic} for more discussions.
Let $R$ be a ring, we write $\Gamma(\rmZ_{m}(\overline{B}), R)=\rmH^{0}(\rmZ_{m}(\overline{B}), R)$ for the space of $R$-valued functions on the set $\rmZ_{m}(\overline{B})$. 

\subsection{Shimura curves with Iwahori level}
We will consider now the curve $\rmX_{m}(p)$ over $\QQ$ whose complex points are given by
\begin{equation*}
\rmX_{m}(p)(\CC)=\rmG(\QQ)\backslash \calH^{\pm} \times \rmG(\Adel^{\infty})/K_{pN^{+},m}.
\end{equation*}
We define an integral model $\interX_{m}(p)$ over $\ZZ[1/Nm]$ which represents the following functor.  Let $S$ be a locally Noetherian scheme over $\ZZ[1/Nm]$. Then $\interX_{m}(p)(S)$ classifies the isomorphism classes of the tuples $(A, \iota, C_{p}, C_{N^{+}}, \alpha_{m})$ where
\begin{enumerate}
\item $A$ is an $S$-abelian scheme of relative dimension $2$;
\item $\iota: \calO_{B}\hookrightarrow \End_{S}(A)$ is an action of $\calO_{B}$ on $A$;
\item $C_{p}$ is a finite flat subgroup scheme of $A[p]$ that is locally free of rank $p^{2}$ and stable under $\calO_{B}$;
\item $C_{N^{+}}$ is a finite flat subgroup scheme of $A[N^{+}]$ of order $(N^{+})^{2}$ that is locally cyclic under the action of $\calO_{B}$;
\item $\alpha_{m}: (\underline{\ZZ/m})^{2}_{S}\rightarrow A[m]$ is an $\calO_{B}$-equivariant injection of finite flat group schemes over $S$.
\end{enumerate}
By \cite[Theorem 4.7]{Buzzard}, $\interX_{m}(p)$ is regular and proper over $\ZZ[1/Nm]$. Again when we consider the base-change of $\interX_{m}(p)$ to $\ZZ_{p^{2}}$, we use the same notation for this base-change. We denote by $\rmX_{m}(p)$ the generic fiber of $\interX_{m}(p)$ and $\overline{\rmX}_{m}(p)$ the special fiber of $\interX_{m}(p)$. We have the following descriptions of $\overline{\rmX}_{m}(p)$ over $\FF_{p^{2}}$. 

\begin{lemma}\label{gamma0p}
The scheme $\overline{\rmX}_{m}(p)$ consists of two irreducible components $\overline{\rmX}_{+,m}$ and $\overline{\rmX}_{-,m}$, both isomorphic to $\overline{\rmX}_{m}$, crossing transversally at the supersingular locus $\overline{\rmX}_{\pm,m}$ of $\overline{\rmX}_{m}(p)$ which can be identified with the supersingular locus of $\overline{\rmX}_{m}$.
\end{lemma}

\begin{proof}
This is well-known and  see \cite[Theorem 4.7(v)]{Buzzard} for a proof of this.
\end{proof}

Let $\pi_{1, p}:\interX_{m}(p)\rightarrow \interX_{m}$ be the degeneracy morphism given by
\begin{equation*}
(A, \iota, C_{p}, C_{N^{+}}, \alpha_{m}) \mapsto (A, \iota, C_{N^{+}}, \alpha_{m})  
\end{equation*}
and  $\pi_{2, p}:\interX_{m}(p)\rightarrow \interX_{m}$ be the degeneracy morphism given by 
\begin{equation*}
(A, \iota, C_{p}, C_{N^{+}}, \alpha_{m}) \mapsto (A/C_{p}, \overline{\iota}, \overline{C}_{N^{+}}, \overline{\alpha}_{m})
\end{equation*}
where $\overline{\iota}$ is induced action of $\calO_{B}$ on $A/C_{p}$ and $\overline{C}_{N^{+}}=C_{N^{+}}\oplus C_{p}/C_{p}$. We also define the \emph{ Atkin--Lehner involution} $w_{p}$ by the recipe
\begin{equation*}
(A, \iota, C_{p}, C_{N^{+}}, \alpha_{m})\mapsto (A/C_{p},  \iota, A[p]/C_{p}, \overline{C}_{N^{+}}, \overline{\alpha}_{m}). 
\end{equation*}

We can define two closed immersions $i_{1}: \overline{\rmX}_{+,m}\rightarrow \overline{\rmX}_{m}(p)$ and $i_{2}: \overline{\rmX}_{-,m}\rightarrow \overline{\rmX}_{m}(p)$ as in the proof of \cite[Theorem 4.7(v)]{Buzzard} such that  $\pi_{1}\circ i_{1}$ is the identity and $\pi_{2}\circ i_{1}$ is the Frobenius automorphism on the underlying point.

We will need the following result known as  the Ihara's lemma. This is proved for the case of classical modular curves by Ribet \cite{Ri-ICM} and Diamond--Taylor \cite{DT} for Shimura curves.  Recall that a Hecke module is \emph{Eisenstein} if its support consists of Eisenstein maximal ideals. A maximal ideal is Eisenstein if its associated Galois representation is reducible. 
\begin{theorem}[Ihara's lemma] \label{DT}
We have the following statements.
\begin{enumerate}
\item The kernel of the pull-back map
\begin{equation*}
(\pi^{*}_{1, p}+\pi^{*}_{2, p}): \rmH^{1}(\rmX_{m}\otimes \Qbar, k_{\lambda})^{\oplus 2} \rightarrow \rmH^{1}(\rmX_{m}(p)\otimes{\Qbar}, k_{\lambda})
\end{equation*}
is Eisenstein.
\item The cokernel of the push-forwad map
\begin{equation*}
(\pi_{1, p, *}, \pi_{2, p, *}): \rmH^{1}(\rmX_{m}(p)\otimes{\Qbar}, k_{\lambda})\rightarrow \rmH^{1}(\rmX_{m}\otimes{\Qbar}, k_{\lambda})^{\oplus 2}
\end{equation*}
is Eisenstein.
\end{enumerate}
\end{theorem}
\begin{proof}
This is proved in \cite[Theorem 2]{DT}.
\end{proof}

\subsection{Unramified level raising for Shimura curves} Let $f$ be a normalized newform in $S^{\mathrm{new}}_{2}(\Gamma_{0}(N))$. Let $E=\QQ(f)$ be the Hecke field of $f$ and  $\lambda$ be a place of $E$ above $\ell$. We denote by $\calO_{\lambda}$ the valuation ring of $E_{\lambda}$ and by $k_{\lambda}$ its residue field. Let $\varpi$ be a uniformizer of $\calO_{\lambda}$ and $\lambda=(\varpi)$ be the maximal ideal. 
Then by the construction of Eichler--Shimura, we can attach a Galois representation
\begin{equation*}
\rho_{f, \lambda}: \rmG_{\QQ}\rightarrow \GL_{2}({E_{\lambda}})=\GL(\rmV_{f,\lambda})
\end{equation*}
to $f$ such that $\tr\rho_{f,\lambda}(\Frob_{p})=a_{p}(f)$ where $a_{p}(f)$ is the Fourier coefficient of $f$ at $p$ for every $p\nmid N$ and such that the determinant $\det(\rho_{f, \lambda})=\epsilon_{\ell}$ is the $\ell$-adic cyclotomic character.  Let $\overline{\rho}_{f, \lambda}: \rmG_{\QQ}\rightarrow \GL_{2}(k_{\lambda})$
be the residual representation of $\rho_{f, \lambda}$. We fix an auxiliary prime $m=d\geq 4$ considered in the last subsection.

For a finite set $S$ of primes away from $Nd$, we also sometimes denote the product of primes in this set by $S$.  Let $\TT^{[S]}$ be the abstract Hecke algebra unramified away from $SNd$. This means $\TT^{[S]}$ is the restricted tensor product of the spherical Hecke algebra $\TT_{v}$ for $\GL_{2}(\QQ_{v})$ with $v\nmid Nd$. The spherical Hecke algebra $\TT_{v}$ is generated by the operators $\rmT_{v}=[\GL_{2}(\ZZ_{v})\begin{pmatrix}v& 0\\ 0&1\\ \end{pmatrix}\GL_{2}(\ZZ_{v})], \rmS^{\pm 1}_{v}=[\GL_{2}(\ZZ_{v})\begin{pmatrix}v& 0\\ 0&v\\ \end{pmatrix}\GL_{2}(\ZZ_{v})]^{\pm 1}$. If $S$ is the empty set $\emptyset$, then $\TT^{[S]}$ will simply be denoted by $\TT$. The Hecke eigensystem of $f$ gives rise to a map
$\phi_{f}: \TT \rightarrow \calO_{\lambda}$ sending $\rmT_{v}$ to $a_{v}(f)$ and sending $\rmS_{v}$ to $1$ for $v\nmid Nd$.  For any integer $n\geq1$, we define the following ideals of  $\TT^{[S]}$ attached to $f$ by
\begin{equation}\label{ideals}
\begin{aligned}
&\frakm^{[S]}=\mathrm{ker}[\TT \xrightarrow{\phi_{f}} \calO_{\lambda}\rightarrow \calO_{\lambda}/\lambda]\cap\TT^{[S]}\\ 
&\frakp^{[S]}_{n}=\mathrm{ker}[\TT\xrightarrow{\phi_{f}}\calO_{\lambda}\rightarrow \calO_{\lambda}/\lambda^{n}]\cap\TT^{[S]}.
\end{aligned}
\end{equation}
We say $\frakm$ is absolutely irreducible if $\overline{\rho}_{f,\lambda}$ is absolutely irreducible. If $S$ is the empty set $\emptyset$, then we will omit the appearance of $S$ from all the notations.

\begin{definition}\label{clean}
 We say the auxiliary prime $d\nmid N\ell$ is \emph{clean} for $\overline{\rho}_{f,\lambda}$ if the degeneracy maps $(\pi_{1,d}, \pi_{2,d})$ on $\rmX_{d}$ induces an isomorphism
\begin{equation*}
(\pi_{1,d, \ast}, \pi_{2,d, \ast}): \rmH^{1}(\rmX_{d}\otimes\QQ^{\ac},\calO_{\lambda})_{/\frakm}\xrightarrow{\sim}  \rmH^{1}(\rmX\otimes{\QQ^{\ac}},\calO_{\lambda})^{\oplus 2}_{/\frakm}
\end{equation*}
and the degeneracy maps $(\overline{\pi}_{1,d}, \overline{\pi}_{2,d})$ on $\rmZ_{d}(\overline{B})$ induces an isomorphism
\begin{equation*}
(\overline{\pi}_{1,d, \ast}, \overline{\pi}_{2,d, \ast}): \Gamma(\rmZ_{d}(\overline{B}),\calO_{\lambda})_{/\frakm^{[p]}}\xrightarrow{\sim}  \Gamma(\rmZ(\overline{B}),\calO_{\lambda})^{\oplus 2}_{/\frakm^{[p]}}.
\end{equation*}
\end{definition}

\begin{remark}
It is easy to find clean primes for $\overline{\rho}_{f,\lambda}$. Let $\rmX_{0}(d)$ be the Shimura curve by adding an  Iwahori level at $d$ to $\rmX$. Then by Ihara's lemma, 
\begin{equation*}
\rmH^{1}(\rmX_{0}(d)\otimes\QQ^{\ac},\calO_{\lambda})_{/\frakm}\rightarrow  \rmH^{1}(\rmX\otimes{\QQ^{\ac}},\calO_{\lambda})^{\oplus 2}_{/\frakm} 
\end{equation*}
is surjective. Moreover the natural map 
\begin{equation*}
\rmH^{1}(\rmX_{d}\otimes\QQ^{\ac},\calO_{\lambda})_{/\frakm}\rightarrow\rmH^{1}(\rmX_{0}(d)\otimes\QQ^{\ac},\calO_{\lambda})_{/\frakm} 
\end{equation*}
is surjective as $\rmH^{1}(\rmX_{0}(d)\otimes\QQ^{\ac},\calO_{\lambda})_{/\frakm}$ can be identified with the fixed part of $\rmH^{1}(\rmX_{d}\otimes\QQ^{\ac},\calO_{\lambda})_{/\frakm}$ under $(\ZZ/d)^{\times}$. Therefore the map
\begin{equation*}
(\pi_{1,d,\ast}, \pi_{2,d,\ast}): \rmH^{1}(\rmX_{d}\otimes\QQ^{\ac},\calO_{\lambda})_{/\frakm}\rightarrow \rmH^{1}(\rmX\otimes{\QQ^{\ac}},\calO_{\lambda})^{\oplus 2}_{/\frakm} 
\end{equation*}
is surjective.  Suppose that $\ker(\pi_{1,d, \ast}, \pi_{2,d, \ast})$ is non-zero, then there exists a $d$-new form congruent to $f$. Then Carayol's classification, see the list above \cite[Theorem A]{DT}, implies that $d^{2}\equiv 1\mod \ell$ or $a_{d}(f)^{2}\equiv(d+1)^{2}\mod \ell$. Then we simply need to avoid these primes. Similar argument works also for 
\begin{equation*}
(\overline{\pi}_{1,d, \ast}, \overline{\pi}_{2,d, \ast}): \Gamma(\rmZ_{d}(\overline{B}),\calO_{\lambda})_{/\frakm^{[p]}}\xrightarrow{\sim}  \Gamma(\rmZ(\overline{B}),\calO_{\lambda})^{\oplus 2}_{/\frakm^{[p]}}.
\end{equation*}
From here on, we will always fix a clean $d$ for $\overline{\rho}_{f,\lambda}$. Furthermore, we will impose the following assumptions on $\overline{\rho}_{f,\lambda}$.
\end{remark}

\begin{assumption}\label{assump}
We make the following assumptions on $\overline{\rho}_{f, \lambda}$ and thus on the maximal ideal $\frakm$:
\begin{enumerate}
\item $\bar{\rho}_{f, \lambda}\vert_{\rmG_{\QQ(\zeta_{\ell})}}$ is absolutely irreducible;
\item $\bar{\rho}_{f, \lambda}$ is minimally ramified at primes in $\Sigma^{+}\cup\Sigma^{-}_{\ram}$ and is ramified at primes in $\Sigma^{-}_{\ram}$;
\item The image of $\overline{\rho}_{f, \lambda}$ contains $\GL_{2}(\FF_{\ell})$.
\end{enumerate}
\end{assumption}

\begin{proposition}\label{multi-one}
Suppose Assumption \ref{assump} holds,  we have an isomorphism 
\begin{equation*}
\rmH^{1}(\rmX\otimes{\QQ^{\ac}},\calO_{\lambda}(1))_{/\frakm}\cong \overline{\rho}_{f, \lambda} 
\end{equation*}
which induces an isomorphism
\begin{equation*}
\rmH^{1}(\rmX_{d}\otimes\QQ^{\ac}, \calO_{\lambda}(1))_{/\frakm}\cong \overline{\rho}^{\oplus 2}_{f, \lambda}.
\end{equation*}
\end{proposition}
\begin{proof}
By \cite{BLR2}, it follows that $\rmH^{1}(\rmX\otimes{\QQ^{\ac}},\calO_{\lambda}(1))_{/\frakm}$ is isomorphic to $\overline{\rho}^{\oplus d_{\frakm}}_{f,\lambda}$ for some $d_{\frakm}\geq 1$. Therefore $\rmH^{1}(\rmX_{d}\otimes\QQ^{\ac},\calO_{\lambda}(1))_{/\frakm}$ is isomorphic to $2d_{\frakm}$ copies of $\overline{\rho}_{f,\lambda}$ by the cleaness of $d$. On the other hand, we know that $\Gamma(\rmZ(\overline{B}),\calO_{\lambda})_{/\frakm}$ is one dimensional over $k_{\lambda}$ under the Assumption \ref{assump} by \cite[Proposition 6.8]{CH-1} by noting that our assumption is stronger than the (CR+) assumption in \cite[Proposition 6.8]{CH-1}. By the same argument of \cite[Lemma 6.4.2]{LTXZZ}, $d_{\frakm}$ is equal to the dimension of $\Gamma(\rmZ(\overline{B}),\calO_{\lambda})_{/\frakm}$ over $k_{\lambda}$ and hence is equal to $1$.
\end{proof}
\begin{remark}\label{clean-cohomology}
Note that $\rmH^{1}(\rmX\otimes{\QQ^{\ac}},\calO_{\lambda}(1))_{\frakm}$ provide a $\rmG_{\QQ}$-stable $\calO_{\lambda}$-lattice $\rho_{\calO_{\lambda}}$ of $\rho_{f, \lambda}$. Then we have 
\begin{equation*}
\rmH^{1}(\rmX_{d}\otimes\QQ^{\ac},\calO_{\lambda}(1))_{\frakm}\cong \rho^{\oplus 2}_{\calO_{\lambda}}
\end{equation*}
by the above proposition and 
\begin{equation*}
\rmH^{1}(\rmX_{d}\otimes\QQ^{\ac},\calO_{\lambda})_{/\frakp_{n}}\cong \rho^{\oplus 2}_{\calO_{\lambda, n}}
\end{equation*}
for the natural reduction $\rho_{\calO_{\lambda, n}}$ of $\rho_{\calO_{\lambda}}$ by $\lambda^{n}$. 
\end{remark}

Let  $\mathrm{cl}: \mathrm{CH}^{1}(\overline{\rmX}_{d})\rightarrow \rmH^{2}(\overline{\rmX}_{d}, \calO_{\lambda}(1))$ be the cycle class map for $\overline{\rmX}_{d}$.
Suppose that $\frakm$ satisfies Assumption \ref{assump}. Using the Hochschild--Serre spectral sequence and the fact that $\rmH^{0}$ and $\rmH^{2}$ of $\overline{\rmX}_{d}\otimes{\FF^{\ac}_{p}}$ are both Eisenstein \cite[Lemma 3]{DT}, we have an isomorphism 
\begin{equation*}
\rmH^{2}(\overline{\rmX}_{d}, \calO_{\lambda}(1))_{/\frakp^{[p]}_{n}}\cong \rmH^{1}(\FF_{p^{2}}, \rmH^{1}(\overline{\rmX}_{d}\otimes\FF^{\ac}_{p}, \calO_{\lambda}(1))_{/\frakp^{[p]}_{n}}).
\end{equation*}
Then the cycle class map $\mathrm{cl}$ induces the Abel--Jacobi map $\mathrm{AJ}_{f, n}: \mathrm{CH}^{1}(\overline{\rmX}_{d})\rightarrow  \rmH^{1}(\FF_{p^{2}}, \rmH^{1}(\overline{\rmX}_{d}\otimes\FF^{\ac}_{p}, \calO_{\lambda}(1))_{/\frakp^{[p]}_{n}})$.
The Abel--Jacobi image $\mathrm{AJ}_{f, n}([\overline{\rmX}^{\ss}_{d}])$ of the class of the supersingular locus $[\overline{\rmX}^{\ss}_{d}]\in\mathrm{CH}^{1}(\overline{\rmX}_{d})$ can be calculated by the following exact sequence
\begin{equation*}
0\rightarrow\rmH^{1}(\overline{\rmX}_{d}\otimes\FF^{\ac}_{p}, \calO_{\lambda}(1))_{/\frakp^{[p]}_{n}}\rightarrow \rmH^{1}(\overline{\rmX}^{\ord}_{d}\otimes\FF^{\ac}_{p}, \calO_{\lambda}(1))_{/\frakp^{[p]}_{n}} \rightarrow \rmH^{0}(\overline{\rmX}^{\ss}_{d}\otimes\FF^{\ac}_{p}, \calO_{\lambda})_{/\frakp^{[p]}_{n}}\rightarrow 0. 
\end{equation*}
More precisely, the above exact sequence induces the following connecting homomorphism
\begin{equation}\label{Phin}
\Phi_{n}: \rmH^{0}(\overline{\rmX}^{\ss}_{d}\otimes\FF^{\ac}_{p}, \calO_{\lambda})^{\rmG_{\FF_{p^{2}}}}_{/\frakp^{[p]}_{n}}\rightarrow  \rmH^{1}(\FF_{p^{2}}, \rmH^{1}(\overline{\rmX}_{d}\otimes \FF^{\ac}_{p}, \calO_{\lambda}(1))_{/\frakp^{[p]}_{n}}).
\end{equation}
The Abel--Jacobi image $\mathrm{AJ}_{f, n}([\overline{\rmX}^{\ss}_{d}])$ of $[\overline{\rmX}^{\ss}_{d}]$ is represented by the image of the characteristic function $\mathbf{1}_{\overline{\rmX}^{\ss}_{d}}$ of the supersingular locus $\overline{\rmX}^{\ss}_{d}$ under $\Phi_{n}$. On the other hand, we can consider the excision exact sequence 
\begin{equation}\label{Psin}
0\rightarrow\rmH^{0}(\overline{\rmX}^{\ss}_{d}\otimes{\FF^{\ac}_{p}}, \calO_{\lambda})_{/\frakp^{[p]}_{n}}\rightarrow \rmH^{1}_{c}(\overline{\rmX}^{\ord}_{d}\otimes{ \FF^{\ac}_{p}}, \calO_{\lambda})_{/\frakp^{[p]}_{n}}\rightarrow \rmH^{1}(\overline{\rmX}_{d}\otimes{\FF^{\ac}_{p}}, \calO_{\lambda})_{/\frakp^{[p]}_{n}}\rightarrow 0
\end{equation}
whose connecting homomorphism gives the map
\begin{equation*}
\Psi_{n}: \rmH^{1}(\overline{\rmX}_{d}\otimes{\FF^{\ac}_{p}}, \calO_{\lambda})^{\rmG_{\FF_{p^{2}}}}_{/\frakp^{[p]}_{n}}\rightarrow   \rmH^{1}(\FF_{p^{2}}, \rmH^{0}(\overline{\rmX}^{\ss}_{d}\otimes\FF^{\ac}_{p}, \calO_{\lambda})_{/\frakp^{[p]}_{n}}).
\end{equation*}

Since $\overline{\rmX}^{\ss}_{d}$ is naturally defined over $\FF_{p^{2}}$, we can naturally identify the source of $\Phi_{n}$ and the target of $\Psi_{n}$ with $\Gamma(\rmZ_{d}(\overline{B}), \calO_{\lambda})_{/\frakp^{[p]}_{n}}$. It is clear that the two  maps $\Phi_{n}$ and $\Psi_{n}$ are dual to each other under the Tate local duality for $\FF_{p^{2}}$ and the Poincar\'e duality for $\overline{\rmX}_{d}\otimes{\FF^{\ac}_{p}}$. The theorem below is usually referred to as the unramified arithmetic level raising theorem for Shimura curves. It is proved in \cite[Proposition 4.8]{LT} and \cite{Xiao}. We will give a slightly different proof of this theorem following the strategy of \cite{Xiao}.

\begin{proposition}\label{Phi-n}
The map $\Phi_{n}$ is surjective and $\Psi_{n}$ is injective.
\end{proposition}

\begin{proof}
We prove that $\Phi_{n}$ is surjective and the injectivity of $\Psi_{n}$ follows by duality. We consider the weight spectral sequence for $\rmH^{1}(\rmXbar_{d}(p)\otimes{\FF^{\ac}_{p}}, \rmR\Psi(\calO_{\lambda})(1))$ and its induced monodromy filtration for $\rmH^{1}(\rmXbar_{d}(p)\otimes{\FF^{\ac}_{p}}, \rmR\Psi(\calO_{\lambda})(1))_{/\frakp^{[p]}_{n}}$ given by
\begin{equation}\label{mono}
\begin{aligned}
&0\subset^{\rmE^{1,0}_{2,n}} \rmM_{1}\rmH^{1}(\rmXbar_{d}(p)\otimes{\FF^{\ac}_{p}},\rmR\Psi(\calO_{\lambda})(1))_{/\frakp^{[p]}_{n}}\subset^{\rmE^{0,1}_{2,n}} \rmM_{0}\rmH^{1}(\rmXbar_{d}(p)\otimes{\FF^{\ac}}, \rmR\Psi(\calO_{\lambda})(1))_{_{/\frakp^{[p]}_{n}}}\\
&\subset^{\rmE^{-1,2}_{2, n}} \rmM_{-1}\rmH^{1}(\rmXbar_{d}(p)\otimes{\FF^{\ac}}, \rmR\Psi(\calO_{\lambda})(1))_{/\frakp^{[p]}_{n}}.\\
\end{aligned}
\end{equation}
By the discussions in \ref{one-mono} and Lemma \ref{gamma0p},  we have 
\begin{equation*}
\begin{aligned}
&\rmE^{1,0}_{2,n}= \rmH^{0}(\overline{\rmX}^{\ss}_{d}\otimes{\FF^{\ac}_{p}}, \calO_{\lambda}(1))_{/\frakp^{[p]}_{n}};\\
&\rmE^{0,1}_{2, n}=\rmH^{1}(\overline{\rmX}_{d}\otimes {\FF^{\ac}_{p}},\calO_{\lambda}(1))^{\oplus 2}_{/\frakp^{[p]}_{n}};\\
&\rmE^{-1,2}_{2, n}=\rmH^{0}(\overline{\rmX}^{\ss}_{d}\otimes {\FF^{\ac}_{p}}, \calO_{\lambda})_{/\frakp^{[p]}_{n}}.\\
\end{aligned}
\end{equation*}
Next we consider the pushforward map 
\begin{equation*}
(\pi_{1,p,*}, \pi_{2,p,*}): \rmH^{1}(\rmXbar_{d}(p)\otimes{\FF^{\ac}_{p}}, \rmR\Psi(\calO_{\lambda})(1))_{/\frakp^{[p]}_{n}}\rightarrow \rmH^{1}(\rmXbar_{d}\otimes{\FF^{\ac}_{p}}, \rmR\Psi(\calO_{\lambda})(1))_{/\frakp^{[p]}_{n}}^{\oplus 2}.
\end{equation*}
This is surjective by Theorem \ref{DT} (Ihara's lemma). It is well-known that the composite 
\begin{equation*}
\rmE^{1,0}_{2,n}\hookrightarrow \rmH^{1}(\rmX_{d}(p)\otimes{\FF^{\ac}_{p}}, \rmR\Psi(\calO_{\lambda})(1))_{/\frakp^{[p]}_{n}}\xrightarrow{(\pi_{1,p,*}, \pi_{2,p,*})} \rmH^{1}(\rmX_{d}\otimes{\FF^{\ac}_{p}}, \rmR\Psi(\calO_{\lambda})(1))^{\oplus 2}_{/\frakp^{[p]}_{n}}
\end{equation*}
is zero. Indeed, the term $\rmE^{1,0}_{2,n}$ corresponds to the toric part of the Neron model of the Jacobian of $\rmX_{d}(p)$ over $\QQ_{p}$ and projects to zero to the $p$-old part of $\rmX_{d}(p)$, see \cite[Theorem 3.10]{Ri-100}. Therefore we obtain the following commutative diagram 
\begin{equation*}
\begin{tikzcd}
\rmH^{1}(\rmXbar_{d}\otimes{\FF^{\ac}_{p}}, \calO_{\lambda}(1))^{\oplus 2}_{/\frakp^{[p]}_{n}} \arrow[r, "({i_{1*}, i_{2*}})"] \arrow[d, Rightarrow, no head] & \frac{\rmH^{1}(\rmXbar_{d}(p)\otimes{\FF^{\ac}_{p}}, \rmR\Psi(\calO_{\lambda})(1))_{/\frakp^{[p]}_{n}}}{\rmH^{0}(\rmXbar^{\ss}_{d}\otimes{\FF^{\ac}_{p}}, \calO_{\lambda}(1))_{/\frakp^{[p]}_{n}}} \arrow[r] \arrow[d, "({\pi_{1,p,*}, \pi_{2,p,*}})"] & \rmH^{0}(\rmXbar^{\ss}_{d}\otimes{\FF^{\ac}_{p}}, \calO_{\lambda})_{/\frakp^{[p]}_{n}}  \arrow[d, "\widetilde{\Phi}_{n}"] \\
\rmH^{1}(\rmXbar_{d}\otimes{\FF^{\ac}_{p}},\calO_{\lambda}(1))^{\oplus 2}_{/\frakp^{[p]}_{n}} \arrow[r, "\nabla_{n}"]      & \rmH^{1}(\rmXbar_{d}\otimes{\FF^{\ac}_{p}}, \calO_{\lambda}(1))^{\oplus 2}_{/\frakp^{[p]}_{n}} \arrow[r]                & \coker(\nabla_{n})             
\end{tikzcd}
\end{equation*}
where the top row of the diagram is the exact sequence 
\begin{equation*}
0\rightarrow\Gr^{\rmM}_{0}\rmH^{1}_{\Iw}(\calO_{\lambda}(1))_{/\frakp^{[p]}_{n}}=\rmE^{1,0}_{2,n}\rightarrow \frac{\rmH^{1}_{\Iw}(\calO_{\lambda}(1))_{/\frakp^{[p]}_{n}}}{\rmM_{1}\rmH^{1}_{\Iw}(\calO_{\lambda}(1))_{/\frakp^{[p]}_{n}}}\rightarrow \Gr^{\rmM}_{-1}\rmH^{1}_{\Iw}(\calO_{\lambda}(1))_{/\frakp^{[p]}_{n}}=\rmE^{-1,2}_{2, n}\rightarrow 0
\end{equation*}
where  we put $\rmH^{1}_{\Iw}(\calO_{\lambda}(1))_{_{/\frakp^{[p]}_{n}}}=\rmH^{1}(\rmXbar_{d}(p)\otimes{\FF^{\ac}_{p}},\rmR\Psi(\calO_{\lambda})(1))_{/\frakp^{[p]}_{n}}$. The map $\widetilde{\Phi}_{n}$ is the one naturally induced by $({\pi_{1,p,*}, \pi_{2,p,*}})$. The map $\nabla_{n}$ is by definition given by the composite of $({\pi_{1,p,*}, \pi_{2,p,*}})$ with $({i_{1*}, i_{2*}})$. The map $\nabla_{n}$ is given by the matrix
\begin{equation*}
\begin{pmatrix}
\mathrm{id} &   \Frob_{p}\\
\Frob_{p}\rmS_{p} & \mathrm{id}\\ 
\end{pmatrix}=
\begin{pmatrix}
\mathrm{id} &   \Frob_{p}\\
\Frob_{p} & \mathrm{id}\\ 
\end{pmatrix}
\end{equation*}
since the central element $\rmS_{p}$ has trivial action on $\rmH^{1}(\rmXbar_{d}\otimes{\FF^{\ac}_{p}},\rmR\Psi(\calO_{\lambda})(1))_{/\frakp^{[p]}_{n}}$. It follows then by a simple computation that we have an isomorphism
\begin{equation*}
\coker(\nabla_{n})\cong \rmH^{1}(\FF_{p^{2}}, \rmH^{1}(\rmXbar_{d}\otimes{\FF^{\ac}_{p}}, \calO_{\lambda}(1))_{/\frakp^{[p]}_{n}}).
\end{equation*}
Since  $({\pi_{1,p,*}, \pi_{2,p,*}})$ is surjective, the map  $\widetilde{\Phi}_{n}$  is surjective as well. 

Therefore we are left to show that $\widetilde{\Phi}_{n}$ agree with the map $\Phi_{n}$. To show this, we rely on some results proved in \cite[Proposition 1.5]{Illusie2}. More precisely, the natural quotient map 
\begin{equation*}
\rmH^{1}(\rmXbar_{d}(p)\otimes{\FF^{\ac}_{p}}, \rmR\Psi(\calO_{\lambda})(1))\rightarrow \rmH^{0}(\overline{\rmX}^{\ss}_{d}\otimes{\FF^{\ac}_{p}}, \calO_{\lambda})
\end{equation*}
in the monodromy filtration factors through $\rmH^{1}(\overline{\rmX}^{\ord}_{d}\otimes{\FF^{\ac}_{p}}, \calO_{\lambda})$:
\begin{equation*}
\rmH^{1}(\rmX_{d}(p)\otimes{\FF^{\ac}_{p}},\rmR\Psi (\calO_{\lambda})(1))\xrightarrow{i^{*}_{1}}\rmH^{1}(\overline{\rmX}^{\ord}_{d}\otimes{\FF^{\ac}_{p}}, \calO_{\lambda}(1)) \rightarrow \rmH^{0}(\overline{\rmX}^{\ss}_{d}\otimes{\FF^{\ac}_{p}}, \calO_{\lambda})\rightarrow 0
\end{equation*}
where $\rmH^{1}(\overline{\rmX}^{\ord}_{d}\otimes{\FF^{\ac}_{p}}, \calO_{\lambda}(1)) \rightarrow \rmH^{0}(\overline{\rmX}^{\ss}_{d}\otimes{\FF^{\ac}_{p}}, \calO_{\lambda})$ comes from the natural excision exact sequence for  $\rmH^{1}(\overline{\rmX}_{d}\otimes{\FF^{\ac}_{p}}, \calO_{\lambda}(1))$ and the $i^{*}_{1}$ is the pullback of the cohomology of nearby cycles
\begin{equation*}
\rmH^{1}(\overline{\rmX}_{d}(p)\otimes{\FF^{\ac}_{p}}, \rmR\Psi(\calO_{\lambda})(1))\xrightarrow{i^{*}_{1}}\rmH^{1}(\overline{\rmX}_{d}\otimes{\FF^{\ac}_{p}},\rmR\Psi(\calO_{\lambda})(1))\cong\rmH^{1}(\overline{\rmX}^{\ord}_{d}\otimes{\FF^{\ac}_{p}},\calO_{\lambda}(1)).
\end{equation*}
Let $x\in \rmH^{0}(\overline{\rmX}^{\ss}_{d}\otimes{\FF^{\ac}_{p}}, \calO_{\lambda})_{/\frakp^{[p]}_{n}}$ and $\tilde{x}$ be a preimage of $x$ in $\rmH^{1}(\overline{\rmX}^{\ord}\otimes{\FF^{\ac}_{p}}, \calO_{\lambda}(1))_{/\frakp^{[p]}_{n}}$. Since $i^{*}_{1}i_{1*}$ is the identity map, we can  take $i_{1*}(\tilde{x})$ as a preimage of $\tilde{x}$ in $\rmH^{1}(\overline{\rmX}_{d}(p)\otimes{\FF^{\ac}_{p}}, \rmR\Psi(\calO_{\lambda})(1))_{/\frakp^{[p]}_{n}}$. Therefore for $x\in \rmH^{0}(\overline{\rmX}^{\ss}_{d}\otimes{\FF^{\ac}_{p}}, \calO_{\lambda})_{/\frakp^{[p]}_{n}}$, we have $\widetilde{\Phi}(x)=(\pi_{1, p,*}i_{1*}(\tilde{x}), \pi_{2,p,*}i_{2*}(\tilde{x}))=(\tilde{x}, \Frob_{p}(\tilde{x}))$. Since the natural quotient map 
\begin{equation*}
\rmH^{1}(\rmXbar_{d}\otimes{\FF^{\ac}_{p}}, \calO_{\lambda}(1))_{/\frakp^{[p]}_{n}}^{\oplus 2} \rightarrow \coker(\nabla_{n})
\end{equation*}
is given by sending $(x, y)\in \rmH^{1}(\rmXbar_{d}\otimes{\FF^{\ac}_{p}}, \calO_{\lambda})(1))^{\oplus 2}$  to $(x-\Frob_{p}(y))$ by the definition of $\nabla_{n}$, we have 
$\widetilde{\Phi}_{n}(x)=(1-\Frob^{2}_{p})\tilde{x}$. But this is precisely the definition of  $\Phi_{n}(x)$. 
\end{proof}

\begin{remark}\label{jump-the-grading}
By \cite[Lemma 9.1]{BD}, the surjectivity of ${\Phi}_{n}$ implies that $\rmT_{p}$ acts on 
\begin{equation*}
\rmH^{1}(\FF_{p^{2}}, \rmH^{1}(\rmXbar_{d}\otimes{\FF^{\ac}_{p}}, \calO_{\lambda}(1))_{/\frakp^{[p]}_{n}}) 
\end{equation*}
by $\epsilon_{p}(p+1)$ and hence 
\begin{equation*}
\rmH^{1}(\FF_{p^{2}}, \rmH^{1}(\rmXbar_{d}\otimes{\FF^{\ac}_{p}}, \calO_{\lambda}(1))_{/\frakp^{[p]}_{n}})\cong \rmH^{1}(\FF_{p^{2}}, \rmH^{1}(\rmXbar_{d}\otimes{\FF^{\ac}_{p}}, \calO_{\lambda}(1))_{/\frakp_{n}}).
\end{equation*}
Therefore we obtain a surjective map 
\begin{equation*}
\Phi_{n}: \rmH^{0}(\overline{\rmX}^{\ss}_{d}\otimes{\FF^{\ac}_{p}}, \calO_{\lambda})_{/\frakp^{[p]}_{n}}^{\rmG_{\FF_{p^{2}}}}\rightarrow  \rmH^{1}(\FF_{p^{2}}, \rmH^{1}(\overline{\rmX}_{d} \otimes{\FF^{\ac}_{p}}, \calO_{\lambda}(1))_{/\frakp_{n}}). 
\end{equation*}
Similar reasoning applies to $\Psi_{n}$ and we have an injection 
\begin{equation*}
\Psi_{n}: \rmH^{1}(\overline{\rmX}_{d}\otimes \FF^{\ac}_{p}, \calO_{\lambda})^{\rmG_{\FF_{p^{2}}}}_{/\frakp_{n}}\rightarrow   \rmH^{1}(\FF_{p^{2}}, \rmH^{0}(\overline{\rmX}^{\ss}_{d}\otimes\FF^{\ac}_{p}, \calO_{\lambda})_{/\frakp^{[p]}_{n}}).
\end{equation*}
\end{remark}

We recall the definition of an $n$-admissible prime for $f$ following \cite{BD}.
\begin{definition}
Let $n\geq 1$ be an integer. A prime $p$ is called an $n$-admissible prime for $f$ if 
\begin{enumerate}
\item $p\nmid N\ell$;
\item $\ell\nmid p^{2}-1$;
\item $\varpi^{n}\mid p+1-\epsilon_{p}(f)a_{p}(f)$ with $\epsilon_{p}(f)\in\{-1, 1\}$. 
\end{enumerate}
\end{definition} 

The following result will be referred to as the unramified arithmetic level raising theorem for the Shimura curve $\rmX_{d}$ following the terminology of \cite{Xiao} as we are considering a Shimura variety at a prime of good reduction and hence we are concerned with the unramified part of the Galois cohomology at this prime.

\begin{proposition}[unramified level raising for Shimura curve]\label{level-raise-curve}
Let $p$ be an $n$-admissible prime for $f$.  We assume that Assumption \ref{assump} holds for $\overline{\rho}_{f,\lambda}$, then we have the following statements.
\begin{enumerate}
\item We have an isomorphism
\begin{equation*}
\Phi_{n}: \rmH^{0}(\overline{\rmX}^{\ss}_{d}\otimes{\FF^{\ac}_{p}}, \calO_{\lambda})_{/\frakp^{[p]}_{n}}^{\rmG_{\FF_{p^{2}}}}\xrightarrow{\sim}  \rmH^{1}(\FF_{p^{2}}, \rmH^{1}(\overline{\rmX}_{d} \otimes{\FF^{\ac}_{p}}, \calO_{\lambda}(1))_{/\frakp_{n}})
\end{equation*}
which can be identified with an isomorphism
\begin{equation*}
\Phi_{n}: \Gamma(\rmZ_{d}(\overline{B}), \calO_{\lambda})_{/\frakp^{[p]}_{n}}\xrightarrow{\sim}  \rmH^{1}(\FF_{p^{2}}, \rmH^{1}(\overline{\rmX}_{d}\otimes{\FF^{\ac}_{p}}, \calO_{\lambda}(1))_{/\frakp_{n}}).
\end{equation*}

\item Similarly, we have a canonical isomorphism
\begin{equation*}
\Psi_{n}: \rmH^{1}(\overline{\rmX}_{d}\otimes{\FF^{\ac}_{p}}, \calO_{\lambda})^{\rmG_{\FF_{p^{2}}}}_{/\frakp_{n}}\xrightarrow{\sim} \rmH^{1}(\FF_{p^{2}}, \rmH^{0}(\overline{\rmX}^{\ss}_{d}\otimes{\FF^{\ac}_{p}}, \calO_{\lambda})_{/\frakp^{[p]}_{n}})
\end{equation*}
which can be identified with an isomorphism
\begin{equation*}
\Psi_{n}: \rmH^{1}(\overline{\rmX}_{d}\otimes{\FF^{\ac}_{p}}, \calO_{\lambda})_{/\frakp_{n}}^{\rmG_{\FF_{p^{2}}}}\xrightarrow{\sim}  \Gamma(\rmZ_{d}(\overline{B}), \calO_{\lambda})_{/\frakp^{[p]}_{n}}.
\end{equation*}
\end{enumerate}
\end{proposition}

\begin{proof}
We only need to show that the two sides of $\Phi_{n}$ have the same cardinality. But by Proposition \ref{multi-one}, we know that 
\begin{equation*}
\rmH^{1}(\FF_{p^{2}}, \rmH^{1}(\overline{\rmX}_{d}\otimes{\FF^{\ac}_{p}}, \calO_{\lambda}(1))_{/\frakp_{n}})=\rmH^{1}(\FF_{p^{2}}, \overline{\rho}^{\oplus2}_{\lambda, n})
\end{equation*}
is free of rank two over $\calO_{\lambda, n}$. Note Assumption \ref{assump} implies that $\Gamma(\rmZ(\overline{B}), k_{\lambda})_{\frakm^{[p]}}$ is one-dimensional by \cite[Proposition 6.8]{CH-1}. By the cleanness of $d$, it follows that $\Gamma(\rmZ_{d}(\overline{B}), k_{\lambda})_{\frakm^{[p]}}$ is two dimensional and hence $\Gamma(\rmZ_{d}(\overline{B}), \calO_{\lambda})_{/\frakp^{[p]}_{n}}$ is free of rank two over $\calO_{\lambda, n}$. The theorem follows.
\end{proof}
\begin{definition}\label{Phi-ast-n}
Under the assumption of the above theorem, we let 
\begin{equation*}
\Phi^{\ast}_{n}: \rmH^{1}(\FF_{p^{2}}, \rmH^{1}(\overline{\rmX}_{d} \otimes{\FF^{\ac}_{p}}, \calO_{\lambda}(1))_{/\frakp_{n}}) \xrightarrow{\sim}\rmH^{0}(\overline{\rmX}^{\ss}_{d}\otimes{\FF^{\ac}_{p}}, \calO_{\lambda})_{/\frakp^{[p]}_{n}}^{\rmG_{\FF_{p^{2}}}}\cong \Gamma(\rmZ_{d}(\overline{B}), \calO_{\lambda})_{/\frakp^{[p]}_{n}}.
\end{equation*}
be the inverse isomorphism of $\Phi_{n}$.
\end{definition}
\begin{lemma}\label{canonical-decomposition}
Let $p$ be an $n$-admissible prime for $f$ and suppose that Assumption \ref{assump} holds for $\overline{\rho}_{f,\lambda}$, then we have the following exact sequence of $\calO_{\lambda, n}[\rmG_{\FF_{p^{2}}}]$-modules
\begin{equation*}
0\rightarrow\Gamma(\rmZ_{d}(\overline{B}), \calO_{\lambda}(1))_{/\frakp^{[p]}_{n}}\rightarrow \rmH^{1}(\overline{\rmX}_{d}\otimes{\FF^{\ac}_{p}}, \calO_{\lambda}(1))_{/\frakp_{n}}\rightarrow \Gamma(\rmZ_{d}(\overline{B}), \calO_{\lambda})_{/\frakp^{[p]}_{n}}\rightarrow0
\end{equation*}
which splits.
\end{lemma}

\begin{proof}
Since $p$ is $n$-admissible and Proposition \ref{multi-one} holds, we have an exact sequence of $\calO_{\lambda, n}[\rmG_{\FF_{p^{2}}}]$-modules
\begin{equation*}
0\rightarrow \rmH^{1}(\overline{\rmX}_{d}\otimes{\FF^{\ac}_{p}}, \calO_{\lambda})_{/\frakp_{n}}^{\rmG_{\FF_{p^{2}}}}(1)\rightarrow \rmH^{1}(\overline{\rmX}_{d}\otimes{\FF^{\ac}_{p}}, \calO_{\lambda}(1))_{/\frakp_{n}}\rightarrow \rmH^{1}(\FF_{p^{2}}, \rmH^{1}(\overline{\rmX}_{d}\otimes{\FF^{\ac}_{p}}, \calO_{\lambda}(1))_{/\frakp_{n}})\rightarrow 0.
\end{equation*}
This exact sequence clearly splits and we obtain an identification of this exact sequence with 
\begin{equation*}
0\rightarrow\Gamma(\rmZ_{d}(\overline{B}), \calO_{\lambda}(1))_{/\frakp^{[p]}_{n}}\rightarrow\rmH^{1}(\overline{\rmX}_{d}\otimes{\FF^{\ac}_{p}}, \calO_{\lambda}(1))_{/\frakp_{n}}\rightarrow \Gamma(\rmZ_{d}(\overline{B}), \calO_{\lambda})_{/\frakp^{[p]}_{n}}\rightarrow0
\end{equation*}
using $\Psi_{n}$ and $\Phi^{\ast}_{n}$. 
\end{proof}

Next we will give a refinement of the arithmetic level raising results which will play an important role in latter applications. We need the following definition.
\begin{definition}
We put
\begin{equation*}
\rmH^{1}(\rmX_{d}(p)\otimes{\Qbar}, \calO_{\lambda}(1))^{\bullet}_{/\frakp^{[p]}_{n}}=\mathrm{ker}[\rmH^{1}(\rmX_{d}(p)\otimes{\Qbar}, \calO_{\lambda}(1))_{/\frakp^{[p]}_{n}}\xrightarrow{(\pi_{1\ast}, \pi_{2\ast})}\rmH^{1}(\rmX_{d}\otimes{\Qbar}, \calO_{\lambda}(1))_{/\frakp^{[p]}_{n}}^{\oplus2}].
\end{equation*}
We call  $\rmH^{1}(\rmX_{d}(p)\otimes{\Qbar}, \calO_{\lambda}(1))^{\bullet}_{/\frakp^{[p]}_{n}}$ the {new part} of $\rmH^{1}(\rmX_{d}(p)\otimes{\Qbar}, \calO_{\lambda}(1))_{/\frakp^{[p]}_{n}}$ at $p$.
\end{definition}

\begin{proposition}\label{new-old-part}
Let $p$ be an $n$-admissible prime for $f$ and suppose that Assumption \ref{assump} holds for $\overline{\rho}_{f,\lambda}$. Then the new part $\rmH^{1}({\rmX}_{d}(p)\otimes{\QQ^{\ac}}, \calO_{\lambda}(1))^{\bullet}_{/\frakp^{[p]}_{n}}$ is unramified at $p$ and we have the following split exact sequence
\begin{equation*}
0\rightarrow\Gamma(\rmZ_{d}(\overline{B}), \calO_{\lambda}(1))_{/\frakp^{[p]}_{n}}\rightarrow \rmH^{1}({\rmX}_{d}(p)\otimes{\QQ^{\ac}_{p}}, \calO_{\lambda}(1))^{\bullet}_{/\frakp^{[p]}_{n}}\rightarrow \Gamma(\rmZ_{d}(\overline{B}), \calO_{\lambda})_{/\frakp^{[p]}_{n}}\rightarrow0
\end{equation*}
of $\calO_{\lambda, n}[\rmG_{\FF_{p^{2}}}]$-modules.
Moreover the map $\alpha^{\ast}=(\pi^{\ast}_{1,p}, \pi^{\ast}_{2,p})$ induces an isomorphism 
\begin{equation*}
\alpha^{\ast}: \rmH^{1}({\rmX}_{d}\otimes{\Qbar}, \calO_{\lambda}(1))_{/\frakp_{n}}\xrightarrow{\sim} \rmH^{1}(\rmX_{d}(p)\otimes{\Qbar}, \calO_{\lambda}(1))^{\bullet}_{/\frakp^{[p]}_{n}}
\end{equation*}
of $\calO_{\lambda, n}[\rmG_{\QQ}]$-modules.
\end{proposition}

\begin{proof}
We have a commutative diagram
\begin{equation}\label{dia1}
\begin{tikzcd}
\ker(\overline{\alpha}_{\ast}) \arrow[d] \arrow[r]            &\ker({\alpha}_{\ast})   \arrow[d] \arrow[r]          &   \rmH^{1}(\overline{\rmX}_{d}(p)\otimes{\FF^{\ac}_{p}}, \rmR\Phi(\calO_{\lambda})(1))_{/\frakp^{[p]}_{n}}    \arrow[d, Rightarrow, no head]       \\
\rmH^{1}(\overline{\rmX}_{d}(p)\otimes{\FF^{\ac}_{p}}, \calO_{\lambda}(1))_{/\frakp^{[p]}_{n}} \arrow[r, "sp"] \arrow[d, "\overline{\alpha}_{\ast}"]      & \rmH^{1}(\overline{\rmX}_{d}(p)\otimes{\FF^{\ac}_{p}}, \rmR\Psi(\calO_{\lambda})(1))_{/\frakp^{[p]}_{n}} \arrow[r] \arrow[d, "\alpha_{\ast}"] & \rmH^{1}(\overline{\rmX}_{d}(p)\otimes{\FF^{\ac}_{p}}, \rmR\Phi(\calO_{\lambda})(1))_{/\frakp^{[p]}_{n}} \arrow[d] \\
\rmH^{1}(\overline{\rmX}_{d}\otimes{\FF^{\ac}_{p}}, \calO_{\lambda}(1))^{\oplus 2}_{/\frakp^{[p]}_{n}}\arrow[r, Rightarrow, no head]\arrow[d] & \rmH^{1}(\overline{\rmX}_{d}\otimes {\FF^{\ac}_{p}}, \calO_{\lambda}(1))^{\oplus 2}_{/\frakp^{[p]}_{n}}\arrow[r] \arrow[d]     & 0      \\
\coker(\overline{\alpha}_{\ast})\arrow[r] &0\\    
\end{tikzcd}
\end{equation}
where $\overline{\alpha}_{\ast}$ and $\alpha_{\ast}$ are both induced by the degeneracy maps $(\pi_{1,p,\ast}, \pi_{2,p,\ast})$. 

To understand $\ker(\overline{\alpha}_{\ast})$, we consider the following commutative diagram
\begin{equation}\label{dia2}
\begin{tikzcd}
\rmH^{0}(\overline{\rmX}^{\ss}_{d}\otimes{\FF^{\ac}_{p}}, \calO_{\lambda}(1))_{/\frakp^{[p]}_{n}}\arrow[r] \arrow[d, Rightarrow, no head] & \ker(\overline{\alpha}_{\ast}) \arrow[d] \arrow[r] & \ker(\nabla_{n}) \arrow[d] \\
\rmH^{0}(\overline{\rmX}^{\ss}_{d}\otimes{\FF^{\ac}_{p}}, \calO_{\lambda}(1))_{/\frakp^{[p]}_{n}} \arrow[r] \arrow[d, "\overline{\alpha}_{\ast}"] & \rmH^{1}(\overline{\rmX}_{d}(p)\otimes{\FF^{\ac}_{p}}, \calO_{\lambda}(1))_{/\frakp^{[p]}_{n}} \arrow[r] \arrow[d, "\overline{\alpha}_{\ast}"] & \rmH^{1}(\overline{\rmX}_{d}\otimes{\FF^{\ac}_{p}}, \calO_{\lambda}(1))^{\oplus 2}_{/\frakp^{[p]}_{n}} \arrow[d, "\nabla_{n}"] \\
0 \arrow[r]  & \rmH^{1}(\overline{\rmX}_{d}\otimes{\FF^{\ac}_{p}}, \calO_{\lambda}(1))^{\oplus 2}_{/\frakp^{[p]}_{n}} \arrow[r, Rightarrow, no head] \arrow[d] & \rmH^{1}(\overline{\rmX}_{d}\otimes{\FF^{\ac}_{p}}, \calO_{\lambda}(1))^{\oplus2}_{/\frakp^{[p]}_{n}} \arrow[d] \\
          & \coker(\overline{\alpha}_{\ast}) \arrow[r]           & \coker(\nabla_{n})     
\end{tikzcd}
\end{equation}
which implies by the snake lemma that $\ker(\overline{\alpha}_{\ast})$ sits in the exact sequence
\begin{equation}\label{ker-alpha}
0\rightarrow \rmH^{0}(\overline{\rmX}^{\ss}_{d}\otimes{\FF^{\ac}_{p}}, \calO_{\lambda}(1))_{/\frakp^{[p]}_{n}}\rightarrow \ker(\overline{\alpha}_{\ast})\rightarrow \ker(\nabla_{n})\rightarrow0
\end{equation}
and that $\coker(\overline{\alpha}_{\ast})=\coker(\nabla_{n})$. By the proof of Proposition \ref{canonical-decomposition}, we have 
\begin{equation*}
\coker(\nabla_{n})\cong\rmH^{1}(\FF_{p^{2}}, \rmH^{1}(\overline{\rmX}_{d}\otimes{\FF^{\ac}_{p}}, \calO_{\lambda}(1))_{/\frakp^{[p]}_{n}}). 
\end{equation*}
Note the connecting map  
\begin{equation*}
\rmH^{1}(\overline{\rmX}_{d}(p)\otimes{\FF^{\ac}_{p}}, \rmR\Phi(\calO_{\lambda})(1))_{/\frakp^{[p]}_{n}}\cong \Gamma(\rmZ_{d}(\overline{B}), \calO_{\lambda})_{/\frakp^{[p]}_{n}}\rightarrow \coker(\overline{\alpha}_{\ast})\cong \rmH^{1}(\FF_{p^{2}}, \rmH^{1}(\overline{\rmX}_{d}\otimes{\FF^{\ac}_{p}}, \calO_{\lambda}(1))_{/\frakp^{[p]}_{n}})
\end{equation*}
is by definition the same as the map $\Phi_{n}$ and hence is an isomorphism.
Therefore by the snake lemma applied to diagram \ref{dia1}, we have 
\begin{equation*}
\ker(\overline{\alpha}_{\ast})=\ker(\alpha_{\ast}).
\end{equation*} 
Note also that $ \ker(\nabla_{n})\cong \coker(\nabla_{n})$ by the canonical isomorphism 
\begin{equation*}
\rmH^{0}(\FF_{p^{2}}, \rmH^{1}(\overline{\rmX}_{d}\otimes{\FF^{\ac}_{p}}, \calO_{\lambda}(1))_{/\frakp^{[p]}_{n}}\cong \rmH^{1}(\FF_{p^{2}}, \rmH^{1}(\overline{\rmX}_{d}\otimes{\FF^{\ac}_{p}}, \calO_{\lambda}(1))_{/\frakp^{[p]}_{n}}
\end{equation*}
which in turn is also isomorphic to $\Gamma(\rmZ_{d}(\overline{B}), \calO_{\lambda}(1))_{/\frakp^{[p]}_{n}}$ via the map $\Phi^{\ast}_{n}$.
Therefore the exact sequence \ref{ker-alpha} is isomorphic to
 \begin{equation*}
0\rightarrow \Gamma(\rmZ_{d}(\overline{B}), \calO_{\lambda}(1))_{/\frakp^{[p]}_{n}}\rightarrow \ker(\overline{\alpha}_{\ast})\rightarrow \Gamma(\rmZ_{d}(\overline{B}), \calO_{\lambda})_{/\frakp^{[p]}_{n}}\rightarrow0
\end{equation*}
which splits as $\ker(\overline{\alpha}_{\ast})$ is clearly unramified as an $\calO_{\lambda, n}[\rmG_{\FF_{p^{2}}}]$-module. This combines with the isomorphism $\ker(\overline{\alpha}_{\ast})=\ker(\alpha_{\ast})$ gives the exact sequence
\begin{equation*}
0\rightarrow \Gamma(\rmZ_{d}(\overline{B}), \calO_{\lambda}(1))_{/\frakp^{[p]}_{n}}\rightarrow \rmH^{1}({\rmX}_{d}(p)\otimes {\QQ^{\ac}_{p}}, \calO_{\lambda}(1))^{\bullet}_{/\frakp^{[p]}_{n}}\rightarrow \Gamma(\rmZ_{d}(\overline{B}), \calO_{\lambda})_{/\frakp^{[p]}_{n}}\rightarrow0.
\end{equation*}

Since the composite map $\alpha_{\ast}\alpha^{\ast}$
\begin{equation*}
\rmH^{1}({\rmX}_{d}\otimes{\Qbar}, \calO_{\lambda}(1))^{\oplus2}_{/\frakp^{[p]}_{n}}\xrightarrow{\alpha^{\ast}} \rmH^{1}(\rmX_{d}(p)\otimes{\Qbar}, \calO_{\lambda}(1))_{/\frakp^{[p]}_{n}}\xrightarrow{\alpha_{\ast}} \rmH^{1}({\rmX}_{d}\otimes{\Qbar}, \calO_{\lambda}(1))^{\oplus2}_{/\frakp^{[p]}_{n}}
\end{equation*}
is given by the matrix 
\begin{equation*}
\Delta_{n}=
\begin{pmatrix}
p+1 &   \rmT_{p}\\
\rmT_{p} & p+1\\ 
\end{pmatrix}
\end{equation*}
whose kernel is isomorphic to $\rmH^{1}({\rmX}_{d}\otimes{\Qbar}, \calO_{\lambda}(1))_{/\frakp^{[p]}_{n}}$ by the map 
\begin{equation}\label{1/2}
x\mapsto \frac{1}{2}(x, -\epsilon_{p}(f)x), 
\end{equation}
we have a canonical map 
\begin{equation*}
\alpha^{\ast}: \ker(\alpha_{\ast}\alpha^{\ast})\cong\rmH^{1}({\rmX}_{d}\otimes{\Qbar}, \calO_{\lambda}(1))_{/\frakp^{[p]}_{n}} \rightarrow \ker(\alpha_{\ast})=\rmH^{1}(\rmX_{d}(p)\otimes{\Qbar}, \calO_{\lambda}(1))^{\bullet}_{/\frakp^{[p]}_{n}} 
\end{equation*}
which is injective by Ihara's lemma. 

Note that Proposition \ref{Phi-n} and Proposition \ref{canonical-decomposition} furnish a map of exact sequences
\begin{equation*}
\begin{tikzcd}
0 \arrow[r] &\Gamma(\rmZ_{d}(\overline{B}), \calO_{\lambda}(1))_{/\frakp^{[p]}_{n}} \arrow[r] \arrow[d]      & \rmH^{1}({\rmX}_{d}\otimes{\QQ^{\ac}_{p}}, \calO_{\lambda}(1))_{/\frakp^{[p]}_{n}} \arrow[r] \arrow[d, "\alpha^{\ast}"] & \Gamma(\rmZ_{d}(\overline{B}), \calO_{\lambda})_{/\frakp^{[p]}_{n}} \arrow[d] \arrow[r] & 0\\
0 \arrow[r] &\Gamma(\rmZ_{d}(\overline{B}), \calO_{\lambda}(1))_{/\frakp^{[p]}_{n}} \arrow[r] &\rmH^{1}({\rmX}_{d}(p)\otimes {\QQ^{\ac}_{p}}, \calO_{\lambda}(1))^{\bullet}_{/\frakp^{[p]}_{n}}\arrow[r] & \Gamma(\rmZ_{d}(\overline{B}), \calO_{\lambda})_{/\frakp^{[p]}_{n}}\arrow[r] & 0\\  
\end{tikzcd}
\end{equation*}
of $\calO_{\lambda,n}[\rmG_{\FF_{p^{2}}}]$-modules. It follows that 
\begin{equation*}
\rmH^{1}(\rmX_{d}(p)\otimes{\Qbar}, \calO_{\lambda}(1))^{\bullet}_{/\frakp^{[p]}_{n}}\cong \rmH^{1}(\rmX_{d}\otimes{\Qbar}, \calO_{\lambda}(1))_{/\frakp_{n}}
\end{equation*}
as they are of the same cardinality by the above discussion and $\alpha^{\ast}$ is injective.
\end{proof}
\begin{remark}\label{special-new-part}
Let $\rmH^{1}(\overline{\rmX}_{d}(p)\otimes {\FF^{\ac}_{p}}, \calO_{\lambda}(1))^{\bullet}_{/\frakp^{[p]}_{n}}=\ker(\overline{\alpha}_{\ast})$ We note for later applications that we have obtained an isomorphism
\begin{equation*}
\rmH^{1}(\rmX_{d}(p)\otimes{\Qbar_{p}}, \calO_{\lambda}(1))^{\bullet}_{/\frakp^{[p]}_{n}}\cong \rmH^{1}(\overline{\rmX}_{d}(p)\otimes {\FF^{\ac}_{p}}, \calO_{\lambda}(1))^{\bullet}_{/\frakp^{[p]}_{n}}
\end{equation*}
of $\calO_{\lambda}[\rmG_{\FF_{p^{2}}}]$-modules. There is an isomorphism of short exact sequences 
\begin{equation*}
\begin{tikzcd}
&\rmH^{1}(\overline{\rmX}_{d}\otimes{\FF^{\ac}_{p}}, \calO_{\lambda})_{/\frakp_{n}}^{\rmG_{\FF_{p^{2}}}}(1) \arrow[r] \arrow[d]      & \rmH^{1}(\overline{\rmX}_{d}\otimes{\FF^{\ac}_{p}}, \calO_{\lambda}(1))_{/\frakp^{[p]}_{n}} \arrow[r] \arrow[d, "\alpha^{\ast}"] & \rmH^{1}(\FF_{p^{2}}, \rmH^{1}(\overline{\rmX}_{d}\otimes{\FF^{\ac}_{p}}, \calO_{\lambda}(1))_{/\frakp_{n}})\arrow[d] \\
& \rmH^{0}(\overline{\rmX}^{\ss}_{d}\otimes{\FF^{\ac}_{p}}, \calO_{\lambda}(1))_{/\frakp^{[p]}_{n}} \arrow[r] & \rmH^{1}(\overline{\rmX}_{d}(p)\otimes {\FF^{\ac}_{p}}, \calO_{\lambda}(1))^{\bullet}_{/\frakp^{[p]}_{n}}\arrow[r] & \rmH^{1}(\FF_{p^{2}}, \rmH^{1}(\overline{\rmX}_{d}\otimes{\FF^{\ac}_{p}}, \calO_{\lambda}(1))_{/\frakp_{n}})\\  
\end{tikzcd}
\end{equation*}
of $\calO_{\lambda,n}[\rmG_{\FF_{p^{2}}}]$-modules where the first vertical map is nothing but $\Psi_{n}$ and the third vertical map is the identity map.
\end{remark}

Next, we derive a dual version of the previous Proposition \ref{new-old-part}. This result will not be used later but we include it for completeness. 
\begin{definition}
We put  
\begin{equation*}
\rmH^{1}(\rmX_{d}(p)\otimes{\Qbar}, \calO_{\lambda}(1))_{\bullet/\frakp^{[p]}_{n}}:=\mathrm{coker}[\rmH^{1}(\rmX_{d}\otimes {\Qbar}, \calO_{\lambda}(1))^{\oplus 2}_{/\frakp^{[p]}_{n}}\xrightarrow{(\pi^{\ast}_{1,p}, \pi^{\ast}_{2,p})}\rmH^{1}(\rmX_{d}(p)\otimes{\Qbar}, \calO_{\lambda}(1))_{/\frakp^{[p]}_{n}}].
\end{equation*}
We call $\rmH^{1}(\rmX_{d}(p)\otimes{\Qbar}, \calO_{\lambda}(1))_{\bullet/\frakp^{[p]}_{n}}$ the {new quotient}
of  $\rmH^{1}(\rmX_{d}(p)\otimes{\Qbar}, \calO_{\lambda}(1))_{/\frakp^{[p]}_{n}}$ at $p$.
\end{definition}

\begin{proposition}\label{new-old-quotient}
Let $p$ be an $n$-admissible prime for $f$.  We assume that Assumption \ref{assump} holds. Then the Galois module $\rmH^{1}({\rmX}_{d}(p)\otimes{\QQ^{\ac}_{p}}, \calO_{\lambda}(1))_{\bullet/\frakp^{[p]}_{n}}$ is unramified at $p$ and we have the following split exact sequence
\begin{equation*}
0\rightarrow\Gamma(\rmZ_{d}(\overline{B}), \calO_{\lambda}(1))_{/\frakp^{[p]}_{n}}\rightarrow \rmH^{1}({\rmX}_{d}(p)\otimes{\QQ^{\ac}_{p}}, \calO_{\lambda}(1))_{\bullet/\frakp^{[p]}_{n}}\rightarrow \Gamma(\rmZ_{d}(\overline{B}), \calO_{\lambda})_{/\frakp^{[p]}_{n}}\rightarrow0
\end{equation*}
of $\calO_{\lambda, n}[\rmG_{\FF_{p^{2}}}]$-modules. Moreover the degeneracy maps $\alpha_{\ast}=(\pi_{1,p,\ast}, \pi_{2,p,\ast})$ induces an isomorphism 
\begin{equation*}
\alpha_{\ast}: \rmH^{1}(\rmX_{d}(p)\otimes{\Qbar}, \calO_{\lambda}(1))_{\bullet/\frakp^{[p]}_{n}}\xrightarrow{\sim}\rmH^{1}(\overline{\rmX}_{d}\otimes{\Qbar}, \calO_{\lambda}(1))_{/\frakp_{n}}
\end{equation*}
of $\calO_{\lambda, n}[\rmG_{\QQ}]$-modules.
\end{proposition}
\begin{proof}
We identify the cohomology $\rmH^{1}({\rmX}_{d}(p)\otimes{\QQ^{\ac}_{p}}, \calO_{\lambda}(1))_{/\frakp^{[p]}_{n}}$ with $\rmH^{1}(\overline{\rmX}_{d}(p)\otimes{\FF^{\ac}_{p}}, \rmR\Psi(\calO_{\lambda})(1))_{/\frakp^{[p]}_{n}}$.
The specialization exact sequence induces an exact sequence
\begin{equation*}
0\rightarrow \rmH^{1}(\overline{\rmX}_{d}(p)\otimes{\FF^{\ac}_{p}}, \calO_{\lambda}(1))_{/\frakp^{[p]}_{n}}\rightarrow \rmH^{1}(\overline{\rmX}_{d}(p)\otimes{\FF^{\ac}_{p}}, \rmR\Psi(\calO_{\lambda})(1))_{/\frakp^{[p]}_{n}} \rightarrow  \rmH^{1}(\overline{\rmX}_{d}(p)\otimes{\FF^{\ac}_{p}}, \rmR\Phi(\calO_{\lambda})(1))_{/\frakp^{[p]}_{n}} \rightarrow 0.
\end{equation*}
We define the new quotient $\rmH^{1}(\rmX_{d}(p)\otimes{\FF^{\ac}_{p}}, \calO_{\lambda}(1))_{\bullet/\frakp^{[p]}_{n}}$ of $\rmH^{1}(\rmX_{d}(p)\otimes{\FF^{\ac}_{p}}, \calO_{\lambda}(1))_{/\frakp^{[p]}_{n}}$ by
\begin{equation*}
\mathrm{coker}[\rmH^{1}(\rmX_{d}\otimes {\FF}^{\ac}_{p}, \calO_{\lambda}(1))^{\oplus 2}_{/\frakp^{[p]}_{n}}\xrightarrow{\overline{\alpha}^{\ast}=(\pi^{\ast}_{1,p}, \pi^{\ast}_{2,p})}\rmH^{1}(\rmX_{d}(p)\otimes{\FF}^{\ac}_{p}, \calO_{\lambda}(1))_{/\frakp^{[p]}_{n}}].
\end{equation*}

The Ihara's lemma implies that  $\alpha^{\ast}=(\pi^{\ast}_{1,p}, \pi^{\ast}_{2,p})$ and $\overline{\alpha}^{\ast}$ is injective and hence we have an exact sequence
\begin{equation*}
0\rightarrow \rmH^{1}(\overline{\rmX}_{d}(p)\otimes{\FF^{\ac}_{p}}, \calO_{\lambda}(1))_{\bullet/\frakp^{[p]}_{n}} \xrightarrow{sp} \rmH^{1}(\overline{\rmX}_{d}(p)\otimes{\FF^{\ac}_{p}}, \rmR\Psi(\calO_{\lambda})(1))_{\bullet/\frakp^{[p]}_{n}} \rightarrow \rmH^{1}(\overline{\rmX}_{d}(p)\otimes{\FF^{\ac}_{p}}, \rmR\Phi(\calO_{\lambda})(1))_{/\frakp^{[p]}_{n}} \rightarrow 0.
\end{equation*}
We have a canonical isomorphism
\begin{equation*}
\rmH^{1}(\overline{\rmX}_{d}(p)\otimes{\FF^{\ac}_{p}}, \rmR\Phi(\calO_{\lambda})(1))_{/\frakp^{[p]}_{n}}\cong \Gamma(Z_{d}(\overline{B}), \calO_{\lambda})_{/\frakp^{[p]}_{n}} 
\end{equation*}
and hence it rests to understand $\rmH^{1}(\overline{\rmX}_{d}(p)\otimes{\FF^{\ac}_{p}}, \calO_{\lambda}(1))_{\bullet/\frakp^{[p]}_{n}}$.
Consider the following commutative diagram
\begin{equation*}
\begin{tikzcd}
\mathrm{ker}(\nabla_{n}) \arrow[r]\arrow[d]& \rmH^{0}(\rmX^{\ss}_{d}\otimes{\FF^{\ac}_{p}}, \calO_{\lambda}(1))_{/\frakp^{[p]}_{n}} \arrow[d]  \\
\rmH^{1}(\overline{\rmX}_{d}\otimes{\FF^{\ac}_{p}}, \calO_{\lambda}(1))^{\oplus 2}_{/\frakp^{[p]}_{n}} \arrow[r, "\overline{\alpha}^{\ast}"] \arrow[d, "\nabla_{n}"] & \rmH^{1}(\overline{\rmX}_{d}(p)\otimes{\FF^{\ac}_{p}}, \calO_{\lambda}(1))_{/\frakp^{[p]}_{n}} \arrow[r] \arrow[d, "a^{\ast}_{0}"] & \rmH^{1}(\overline{\rmX}_{d}\otimes{\FF^{\ac}_{p}}, \calO_{\lambda}(1))_{\bullet/\frakp^{[p]}_{n}}\\
\rmH^{1}(\overline{\rmX}_{d}\otimes{\FF^{\ac}_{p}}, \calO_{\lambda}(1))^{\oplus 2}_{/\frakp^{[p]}_{n}} \arrow[r]  & \rmH^{1}(\overline{\rmX}_{d}\otimes{\FF^{\ac}_{p}}, \calO_{\lambda}(1))^{\oplus 2}_{/\frakp^{[p]}_{n}}  \\  
\end{tikzcd}
\end{equation*}
where $a^{\ast}_{0}$ is the pull-back of the normalization map $a_{0}:\overline{\rmX}^{(0)}_{d}(p)=\overline{\rmX}_{d}\sqcup\overline{\rmX}_{d}\rightarrow \overline{\rmX}_{d}(p)$ and the middle column is the natural exact sequence as in \ref{normaliz-exact}.

The map $\nabla_{n}$ is given by the composite $ a^{\ast}_{0}\circ\overline{\alpha}^{\ast}$ which can be represented by the matrix
\begin{equation*}
\begin{pmatrix}
\mathrm{id} &   \Frob_{p}\\
\Frob_{p}\rmS_{p} & \mathrm{id}\\ 
\end{pmatrix}=
\begin{pmatrix}
\mathrm{id} &   \Frob_{p}\\
\Frob_{p} & \mathrm{id}\\ 
\end{pmatrix}
\end{equation*}
and thus $\mathrm{ker}(\nabla_{n})\cong\rmH^{1}(\overline{\rmX}_{d}\otimes{\FF^{\ac}_{p}}, \calO_{\lambda})^{\rmG_{\FF_{p^{2}}}}_{/\frakp^{[p]}_{n}}(1)$. Note that the top horizontal map is nothing but $\Psi_{n}$ by the proof of \cite[Proposition 4.8]{LT} and hence is an isomorphism by Theorem \ref{level-raise-curve}. Then a simple diagram chase shows that $\rmH^{1}(\overline{\rmX}_{d}\otimes{\FF^{\ac}_{p}}, \calO_{\lambda}(1))_{\bullet/\frakp^{[p]}_{n}}$ is isomorphic to 
\begin{equation*}
\coker(\nabla_{n})\cong\rmH^{1}(\FF_{p^{2}}, \rmH^{1}(\overline{\rmX}_{d}\otimes{\FF^{\ac}_{p}}, \calO_{\lambda})_{/\frakp^{[p]}_{n}})(1)
\end{equation*}
which in turn is isomorphic to $\rmH^{1}(\overline{\rmX}_{d}\otimes{\FF^{\ac}_{p}}, \calO_{\lambda})^{\rmG_{\FF_{p^{2}}}}_{/\frakp^{[p]}_{n}}(1)$ via the canonical isomorphism 
\begin{equation*}
\rmH^{1}(\overline{\rmX}_{d}\otimes{\FF^{\ac}_{p}}, \calO_{\lambda})^{\rmG_{\FF_{p^{2}}}}_{/\frakp^{[p]}_{n}}(1)\rightarrow \rmH^{1}(\FF_{p^{2}}, \rmH^{1}(\overline{\rmX}_{d}\otimes{\FF^{\ac}_{p}}, \calO_{\lambda})_{/\frakp^{[p]}_{n}})(1).
\end{equation*} 
Since $\Psi_{n}$ induces an isomorphism
$\Gamma(Z_{d}(\overline{B}), \calO_{\lambda})_{/\frakp^{[p]}_{n}}(1)\cong\rmH^{1}(\overline{\rmX}_{d}\otimes{\FF^{\ac}_{p}}, \calO_{\lambda}(1))_{\bullet/\frakp^{[p]}_{n}}$, we have an exact sequence 
\begin{equation}\label{quotient-exact}
0\rightarrow\Gamma(\rmZ_{d}(\overline{B}), \calO_{\lambda}(1))_{/\frakp^{[p]}_{n}}\rightarrow \rmH^{1}({\rmX}_{d}(p)\otimes{\QQ^{\ac}_{p}}, \calO_{\lambda}(1))_{\bullet/\frakp^{[p]}_{n}}\rightarrow \Gamma(\rmZ_{d}(\overline{B}), \calO_{\lambda})_{/\frakp^{[p]}_{n}}\rightarrow0
\end{equation}
of $\calO_{\lambda, n}[\rmG_{\FF_{p^{2}}}]$-modules. This exact sequence is unramified since the monodromy operator acts by zero by the Picard--Lefschetz formula \ref{picard-lef} and therefore it splits.

Since the composite map $\alpha_{\ast}\alpha^{\ast}$
\begin{equation*}
\rmH^{1}({\rmX}_{d}\otimes{\Qbar}, \calO_{\lambda}(1))^{\oplus2}_{/\frakp^{[p]}_{n}}\xrightarrow{\alpha^{\ast}} \rmH^{1}(\rmX_{d}(p)\otimes{\Qbar}, \calO_{\lambda}(1))_{/\frakp^{[p]}_{n}}\xrightarrow{\alpha_{\ast}} \rmH^{1}({\rmX}_{d}\otimes{\Qbar}, \calO_{\lambda}(1))^{\oplus2}_{/\frakp^{[p]}_{n}}
\end{equation*}
is given by the matrix
\begin{equation*}
\Delta_{n}=
\begin{pmatrix}
p+1 &   \rmT_{p}\\
\rmT_{p} & p+1\\ 
\end{pmatrix}
\end{equation*}
whose cokernel is isomorphic to $\rmH^{1}({\rmX}_{d}\otimes{\Qbar}, \calO_{\lambda}(1))_{/\frakp^{[p]}_{n}}$ by the map $(x, y)\mapsto (p+1)x-\rmT_{p}y$, we have a canonical map 
\begin{equation*}
\alpha_{\ast}: \coker(\alpha^{\ast})=\rmH^{1}({\rmX}_{d}\otimes{\Qbar}, \calO_{\lambda}(1))_{\bullet/\frakp^{[p]}_{n}} \rightarrow \coker(\alpha_{\ast}\alpha^{\ast})\cong\rmH^{1}(\rmX_{d}(p)\otimes{\Qbar}, \calO_{\lambda}(1))_{/\frakp^{[p]}_{n}} 
\end{equation*}
which is surjective by Ihara's lemma. Note that Proposition \ref{Phi-n} and Proposition \ref{canonical-decomposition} furnishes a map of exact sequences
\begin{equation*}
\begin{tikzcd}
\Gamma(\rmZ_{d}(\overline{B}), \calO_{\lambda}(1))_{/\frakp^{[p]}_{n}} \arrow[r] \arrow[d]      & \rmH^{1}({\rmX}_{d}(p)\otimes{\QQ^{\ac}_{p}}, \calO_{\lambda}(1))_{\bullet/\frakp^{[p]}_{n}} \arrow[r] \arrow[d, "\alpha_{\ast}"] & \Gamma(\rmZ_{d}(\overline{B}), \calO_{\lambda})_{/\frakp^{[p]}_{n}} \arrow[d] \\
\Gamma(\rmZ_{d}(\overline{B}), \calO_{\lambda}(1))_{/\frakp^{[p]}_{n}} \arrow[r] & \rmH^{1}({\rmX}_{d}\otimes {\QQ^{\ac}_{p}}, \calO_{\lambda}(1))_{/\frakp^{[p]}_{n}}\arrow[r] & \Gamma(\rmZ_{d}(\overline{B}), \calO_{\lambda})_{/\frakp^{[p]}_{n}}\\  
\end{tikzcd}
\end{equation*}
which has to be an isomorphism as $\alpha_{\ast}$ is a surjective morphism of modules of the same cardinality, we see that 
\begin{equation*}
\rmH^{1}(\rmX_{d}(p)\otimes{\Qbar_{p}}, \calO_{\lambda}(1))_{\bullet/\frakp^{[p]}_{n}}\cong \rmH^{1}(\rmX_{d}\otimes{\Qbar_{p}}, \calO_{\lambda}(1))_{/\frakp_{n}}
\end{equation*}
as desired. 
\end{proof}

\begin{remark}
There is an isomorphism of short exact sequences 
\begin{equation*}
\begin{tikzcd}
&\rmH^{1}(\overline{\rmX}_{d}\otimes{\FF^{\ac}_{p}}, \calO_{\lambda})_{/\frakp_{n}}^{\rmG_{\FF_{p^{2}}}}(1) \arrow[r] \arrow[d]      & \rmH^{1}(\overline{\rmX}_{d}(p)\otimes{\FF^{\ac}_{p}}, \calO_{\lambda}(1))_{\bullet/\frakp^{[p]}_{n}} \arrow[r] \arrow[d, "\alpha_{\ast}"] & \rmH^{0}(\rmX^{\ss}_{d}\otimes{\FF^{\ac}_{p}}, \calO_{\lambda})_{/\frakp^{[p]}_{n}}\arrow[d] \\
&\rmH^{1}(\overline{\rmX}_{d}\otimes{\FF^{\ac}_{p}}, \calO_{\lambda})_{/\frakp_{n}}^{\rmG_{\FF_{p^{2}}}}(1) \arrow[r] & \rmH^{1}(\overline{\rmX}_{d}\otimes {\FF^{\ac}_{p}}, \calO_{\lambda}(1))_{/\frakp^{[p]}_{n}}\arrow[r] & \rmH^{1}(\FF_{p^{2}},\rmH^{1}(\overline{\rmX}_{d}\otimes{\FF^{\ac}_{p}}, \calO_{\lambda}(1))_{/\frakp_{n}})\\  
\end{tikzcd}
\end{equation*}
of $\calO_{\lambda,n}[\rmG_{\FF_{p^{2}}}]$-modules where the first vertical map is the identity map and the third vertical map is $\Phi_{n}$.
\end{remark}

\section{Level raising on triple product of Shimura curves}
\subsection{The triple tensor product Galois representation} Let $\underline{\mathbf{f}}=(f_{1}, f_{2}, f_{3})\in S^{\new}_{2}(\Gamma_{0}(N))^{3}$ be a triple of normalized newforms of level $\Gamma_{0}(N)$ with $q$-expansions: 
\begin{equation*}
\begin{aligned}
&f_{1}=\sum_{n\geq 1} a_{n}(f_{1})q^{n};\\
&f_{2}=\sum_{n\geq 1} a_{n}(f_{2})q^{n};\\
&f_{3}=\sum_{n\geq 1} a_{n}(f_{3})q^{n}.\\
\end{aligned}
\end{equation*} 

We assume $N=N^{+}N^{-}$ such that $(N^{+}, N^{-})=1$ and $N^{-}$ is square-free with \emph{even} number of prime divisors.  For $i\in\{1, 2, 3\}$, let $E_{i}=\QQ(f_{i})$ be the Hecke field of $f_{i}$. Let $\lambda_{i}$ be a place in $E_{i}$ above $\ell$ and write $E_{\lambda_{i}}$ for the completion of $E_{i}$ at $\lambda_{i}$. Let $\calO_{\lambda_{i}}$ be the valuation ring of $E_{\lambda_{i}}$. Let $\lambda_{i}$ be the maximal ideal of $\calO_{\lambda_{i}}$ and denote by $k_{\lambda_{i}}$ its residue field. We denote by 
\begin{equation*}
\rho_{f_{i}, \lambda_{i}}:\rmG_{\QQ}\rightarrow \GL_{2}({E_{\lambda_{i}}})=\GL_{2}(\rmV_{f_{i},\lambda_{i}})
\end{equation*}
the Galois representation attached to $f_{i}$ whose residue representation is denoted by $\overline{\rho}_{f_{i},\lambda_{i}}$. Let $d$ be a common clean prime for all $\overline{\rho}_{f_{i},\lambda_{i}}$. We have the morphism $\phi_{i}: \TT\rightarrow \calO_{\lambda_{i}}$ corresponding to the Hecke eigensystem of $f_{i}$. Let $S$ be a set of primes away from $Nd$. For any integer $n\geq 1$,  we define some ideals of $\TT^{[S]}$ associated to $f_{i}$ as in \ref{ideals} by
\begin{equation*}
\begin{aligned}
&\frakm^{[S]}_{i}=\mathrm{ker}[\TT \xrightarrow{\phi_{i}} \calO_{\lambda}\rightarrow \calO_{\lambda}/\lambda]\cap\TT^{[S]}\\ 
&\frakp^{[S]}_{i, n}=\mathrm{ker}[\TT\xrightarrow{\phi_{i}}\calO_{\lambda}\rightarrow \calO_{\lambda}/\lambda^{n}]\cap\TT^{[S]}.\\
\end{aligned}
\end{equation*}
We introduce the triple of ideals $\frakm^{[S]}_{\triplef}=(\frakm^{[S]}_{1}, \frakm^{[S]}_{2}, \frakm^{[S]}_{3})$ and $\frakp^{[S]}_{\triplef, n}=(\frakp^{[S]}_{i, n}, \frakp^{[S]}_{2, n}, \frakp^{[S]}_{3, n})$.
 By Proposition \ref{multi-one}, $\rmH^{1}(\rmX_{d}\otimes {\QQ^{\ac}}, \calO_{\lambda_{i}}(1))_{\frakm_{i}}$ defines a $\rmG_{\QQ}$-stable $\calO_{\lambda_{i}}$-lattice $\rho_{\calO_{\lambda_{i}}}$ in $\rho_{f_i, \lambda_{i}}$. We denote by $\rho_{\calO_{\lambda_{i}, n}}$ the reduction of $\rho_{\calO_{\lambda_{i}}}$ by $\lambda^{n}_{i}$. Then we have $\rmH^{1}(\rmX_{d}\otimes {\QQ^{\ac}}, \calO_{\lambda_{i}}(1))_{\frakm_{i}}\cong \rho^{\oplus 2}_{\calO_{\lambda_{i}}}$ and $\rmH^{1}(\rmX_{d}\otimes{\QQ^{\ac}}, \calO_{\lambda_{i}}(1))_{/\frakp_{i,n}}\cong \rho^{\oplus 2}_{\calO_{\lambda_{i}, n}}$
by Remark \ref{clean-cohomology}.  

We define the $\calO_{\undlamb}[\rmG_{\QQ}]$-module 
$\rmM(\triplef,d)$ by 
\begin{equation*}
\rmM(\triplef, d)=\threetensor\rmH^{1}(\rmX_{d}\otimes \QQ^{\ac}, \calO_{\lambda_{i}}(1))_{\frakm_{i}}. 
\end{equation*}
And for each integer $n\geq 1$, we define the $\calO_{\undlamb, n}[\rmG_{\QQ}]$-module $\rmM_{n}(\triplef,d)$ by
\begin{equation*}
\rmM_{n}(\triplef, d)= \threetensor\rmH^{1}(\rmX_{d}\otimes {\QQ^{\ac}}, \calO_{\lambda_{i}}(1))_{/\frakp_{i, n}} 
\end{equation*}
where we put $\calO_{\undlamb, n}=\calO_{\lambda_{1},n}\otimes \calO_{\lambda_{2}, n}\otimes \calO_{\lambda_{3}, n}$. Similarly for a prime $p\nmid Nd$, we put 
\begin{equation*}
\rmM^{[p]}(\triplef, d)=\threetensor\rmH^{1}(\rmX_{d}(p)\otimes \QQ^{\ac}, \calO_{\lambda_{i}}(1))_{\frakm^{[p]}_{i}}
\end{equation*}
and for each integer $n\geq 1$, we define the $\calO_{\undlamb, n}[\rmG_{\QQ}]$-module $\rmM^{[p]}_{n}(\triplef,d)$ by
\begin{equation*}
\rmM^{[p]}_{n}(\triplef, d)= \threetensor\rmH^{1}(\rmX_{d}(p)\otimes {\QQ^{\ac}}, \calO_{\lambda_{i}}(1))_{/\frakp^{[p]}_{i, n}}.
\end{equation*}

The Galois module $\rmM(\triplef, d)$ appears naturally in the middle degree cohomology $\rmH^{3}(\rmX^{3}_{d}\otimes{\Qbar}, \calO_{\undlamb}(2))$ of the triple fiber product  $\rmX^{3}_{d}$. By the K\"unneth formula,  the triple tensor product Hecke algebra $\TT\otimes \TT\otimes \TT$ acts on $\rmH^{3}(\rmX^{3}_{d}\otimes{\Qbar}, \calO_{\undlamb}(2))$. Therefore we can localize $\rmH^{3}(\rmX^{3}_{d}\otimes{\Qbar}, \calO_{\undlamb}(2))$ at $\frakm_{\triplef}=(\frakm_{1}, \frakm_{2},\frakm_{3})$. Similarly, the  Galois module $\rmM^{[p]}(\triplef, d)$ appears naturally in $\rmH^{3}(\rmX^{3}_{d}(p)\otimes{\Qbar}, \calO_{\undlamb}(2))$. The triple tensor product Hecke algebra $\TT^{[p]}\otimes \TT^{[p]}\otimes \TT^{[p]}$ acts on $\rmH^{3}(\rmX^{3}_{d}(p)\otimes{\Qbar}, \calO_{\undlamb}(2))$ and therefore we can localize it at the triple of maximal ideals $\frakm^{[p]}_{\triplef}=(\frakm^{[p]}_{1}, \frakm^{[p]}_{2}, \frakm^{[p]}_{3})$.
In fact, we have the following lemma.
\begin{lemma}\label{kunneth}
There is an isomorphism of $\calO_{\undlamb}[\rmG_{\QQ}]$-modules
\begin{equation*}
\rmM(\triplef, d)(-1)\cong \rmH^{3}(\rmX^{3}_{d}\otimes{\Qbar}, \calO_{\undlamb}(2))_{\frakm_{\triplef}}
\end{equation*} 
and an isomorphism of $\calO_{\undlamb, n}[\rmG_{\QQ}]$-modules
\begin{equation*}
\rmM_{n}(\triplef, d)(-1)\cong \rmH^{3}(\rmX^{3}_{d}\otimes{\Qbar}, \calO_{\undlamb}(2))_{/\frakp_{\triplef,n}} 
\end{equation*} 
for each $n\geq 1$. Similarly, there is an isomorphism of $\calO_{\undlamb}[\rmG_{\QQ}]$-modules
\begin{equation*}
\rmM^{[p]}(\triplef, d)(-1)\cong \rmH^{3}(\rmX^{3}_{d}(p)\otimes{\Qbar}, \calO_{\undlamb}(2))_{\frakm^{[p]}_{\triplef}}
\end{equation*} 
and an isomorphism of $\calO_{\undlamb, n}[\rmG_{\QQ}]$-modules
\begin{equation*}
\rmM^{[p]}_{n}(\triplef, d)(-1)\cong \rmH^{3}(\rmX^{3}_{d}(p)\otimes{\Qbar}, \calO_{\undlamb}(2))_{/\frakp^{[p]}_{\triplef,n}}.
\end{equation*}
\end{lemma}
\begin{proof}
This follows from an easy application of the K\"{u}nneth formula using the fact that $\rmH^{0}$ and $\rmH^{2}$ of $\rmX_{d}\otimes{\QQ}^{\ac}$ are both Eisenstein as Hecke modules, see \cite[Lemma 3]{DT}. 
\end{proof}
We define the pull-back map
\begin{equation*}
\underline{\mathbf{a}}^{\ast}: \rmM_{n}(\triplef, d)^{\oplus 8}\rightarrow \rmM^{[p]}_{n}(\triplef, d)
\end{equation*}
to be the map given by
\begin{equation*}
\threetensor\alpha^{\ast}_{i}:\threetensor(\rmH^{1}(\rmX_{d}\otimes \QQ^{\ac}, \calO_{\lambda_{i}}(1))^{\oplus 2}_{/\frakp_{i,n}})\rightarrow \threetensor\rmH^{1}(\rmX_{d}(p)\otimes \QQ^{\ac}, \calO_{\lambda_{i}}(1))_{/\frakp^{[p]}_{i, n}}
\end{equation*}
induced by the pull-back maps $\alpha^{\ast}_{i}=\pi^{\ast}_{1,p}+\pi^{\ast}_{2,p}$ on the $i$-th copy of the tensor product $\threetensor(\rmH^{1}(\rmX_{d}\otimes \QQ^{\ac}, \calO_{\lambda_{i}}(1))^{\oplus 2}_{/\frakp_{i,n}})$.  

We define the push-forward map
\begin{equation*}
\underline{\mathbf{a}}_{\ast}:\rmM^{[p]}_{n}(\triplef, d)\rightarrow \rmM_{n}(\triplef, d)^{\oplus 8}
\end{equation*}
to be the map given by
\begin{equation*}
\threetensor\alpha_{i\ast}:\threetensor\rmH^{1}(\rmX_{d}(p)\otimes \QQ^{\ac}, \calO_{\lambda_{i}}(1))_{/\frakp^{[p]}_{i, n}}\rightarrow \threetensor(\rmH^{1}(\rmX_{d}\otimes \QQ^{\ac}, \calO_{\lambda_{i}}(1))^{\oplus 2}_{/\frakp_{i,n}})
\end{equation*}
induced by the push-forward maps $\alpha_{i\ast}=(\pi_{1,p,\ast}, \pi_{2,p,\ast})$ on the $i$-th copy of the tensor product $\threetensor\rmH^{1}(\rmX_{d}(p)\otimes \QQ^{\ac}, \calO_{\lambda_{i}}(1))_{/\frakp^{[p]}_{i, n}}$.  

We define the new part 
$\rmM^{[p]}_{n}(\triplef, d)^{\bullet}$ of $\rmM^{[p]}_{n}(\triplef, d)$ as $\ker(\underline{\mathbf{a}}_{\ast})$. We also define the new part 
\begin{equation*}
\rmH^{3}(\rmX^{3}_{d}(p)\otimes\QQ^{\ac}, \calO_{\undlamb}(2))^{\bullet}_{/\frakp^{[p]}_{\triplef,n}}
\end{equation*}
of $\rmH^{3}(\rmX^{3}_{d}(p)\otimes\QQ^{\ac},  \calO_{\undlamb}(2))_{/\frakp^{[p]}_{\triplef,n}}$ as the kernel of the map
\begin{equation*}
\underset{?_{1}, ?_{2}, ?_{3}\in\{1, 2\}^{3}}\oplus(\pi_{?_{1, p, \ast}}, \pi_{?_{2,p,\ast}}, \pi_{?_{3,p, \ast}}): \rmH^{3}(\rmX^{3}_{d}(p)\otimes\QQ^{\ac}, \calO_{\undlamb}(2))_{/\frakp^{[p]}_{\triplef,n}}\rightarrow \rmH^{3}(\rmX^{3}_{d}\otimes\QQ^{\ac}, \calO_{\undlamb}(2))^{\oplus 8}_{/\frakp_{\triplef,n}}.
\end{equation*}

\begin{lemma}\label{new-part-threefold}
We have an isomorphism
\begin{equation*}
\rmM^{[p]}_{n}(\triplef, d)^{\bullet}\cong \rmH^{3}(\rmX^{3}_{d}(p)\otimes\QQ^{\ac}, \calO_{\undlamb}(2))^{\bullet}_{/\frakp^{[p]}_{\triplef,n}}
\end{equation*}
of $\calO_{\undlamb,n}[\rmG_{\QQ}]$-modules.
\end{lemma}
\begin{proof}
This is a direct consequence of Lemma \ref{kunneth} and the K\"unneth formula.
\end{proof}

\subsection{Unramified level raising for triple product of Shimura curves}\label{unram-raise} We recall next the definition of an \emph{$n$-admissible prime for $\triplef$} introduced in \cite[Definition 4.3]{Wang}. 
\begin{definition}
Let $n\geq 1$ be an integer. We say that a prime $p$ is \emph{$n$-admissible} for $\triplef=(f_{1}, f_{2}, f_{3})$ if 
\begin{enumerate}
\item $p\nmid N\ell$;
\item $\ell\nmid p^{2}-1$;
\item $\varpi_{i}^{n}\mid p+1-\epsilon_{p, i}a_{p}(f_{i})$ with $\epsilon_{p,i}=\pm1$ for $i\in\{1, 2, 3\}$;
\item $\epsilon_{p, 1}\epsilon_{p, 2}\epsilon_{p, 3}=1$. 
\end{enumerate}
\end{definition} 
Under the assumption \ref{assump} for each member of $\triplef$, there are infinitely many $n$-admissible primes for $\triplef$ by \cite[Lemma 5.5]{Wang}. Let $p$ from here on be an $n$-admissible prime for $\triplef$.

Let $i\in\{1, 2, 3\}$,  we denote by $\overline{\rmX}_{i}$  the $i$-th copy of $\overline{\rmX}_{d}$ in the triple product $\overline{\rmX}^{3}_{d}$  and let
\begin{equation*}
\rmT_{i, n}=\rmH^{1}(\overline{\rmX}_{i}\otimes{\QQ^{\ac}_{p}},\calO_{\lambda_{i}}(1))_{/\frakp_{i,n}}. 
 \end{equation*}
 By Proposition \ref{canonical-decomposition}, we have a natural split exact sequence
 \begin{equation}\label{fil-of-T}
 0\rightarrow \rmT^{+}_{i, n}\rightarrow \rmT_{i, n}\rightarrow \rmT^{-}_{i, n}\rightarrow 0
 \end{equation}
 where 
\begin{enumerate}
\item  $\rmT^{+}_{i, n}$ is given by $\rmH^{1}(\overline{\rmX}_{i}\otimes{\FF^{\ac}_{p}}, \calO_{\lambda_{i}})_{/\frakp_{i,n}}^{\rmG_{\FF_{p^{2}}}}(1)$ which is identified with $\Gamma(\rmZ_{d}(\overline{B}),\calO_{\lambda_{i}}(1))_{/\frakp^{[p]}_{i,n}}$ via the  map $\Psi_{i,n}$ in \ref{Psin};  
\item $\rmT^{-}_{i, n}$ is given by $\rmH^{1}(\FF_{p^{2}},\rmH^{1}(\overline{\rmX}_{i}\otimes{\FF^{\ac}_{p}}, \calO_{\lambda_{i}}(1))_{/\frakp_{i,n}})$ which is identified with $\Gamma(\rmZ_{d}(\overline{B}),\calO_{\lambda_{i}})_{/\frakp^{[p]}_{i, n}}$ via the map $\Phi^{\ast}_{i,n}$ in Definition \ref{Phi-ast-n}. 
\end{enumerate}
By smooth base change theorem, we have $\rmT_{i, n}\cong\rmH^{1}(\overline{\rmX}_{i}\otimes{{\FF}^{\ac}}, \calO_{\lambda_{i}}(1))_{/\frakp_{i,n}}$.

Let $i\in\{1, 2, 3\}$, we denote by $\overline{\rmX}_{i}(p)$ the $i$-th copy of $\overline{\rmX}_{d}(p)$ in the triple product $\overline{\rmX}^{3}_{d}(p)$ and let
\begin{equation*}
\rmT^{\bullet}_{i, n}=\rmH^{1}({\rmX}_{d}(p)\otimes{{\QQ}^{\ac}_{p}}, \calO_{\lambda_{i}}(1))^{\bullet}_{/\frakp_{i,n}} 
\end{equation*}
be the new part of $\rmH^{1}({\rmX}_{d}(p)\otimes{\overline{\QQ}^{\ac}_{p}}, \calO_{\lambda_{i}}(1))_{/\frakp_{i,n}}$. By Proposition \ref{new-old-part}, we have a natural split exact sequence
 \begin{equation}\label{fil-T-new}
 0\rightarrow \rmT^{\bullet+}_{i, n}\rightarrow \rmT^{\bullet}_{i, n}\rightarrow \rmT^{\bullet-}_{i, n}\rightarrow 0
 \end{equation}
 where 
 \begin{enumerate}
 \item $\rmT^{\bullet+}_{i, n}$ is given by $\rmH^{0}(\overline{\rmX}^{\ss}_{d}\otimes{\FF^{\ac}_{p}}, \calO_{\lambda_{i}}(1))_{/\frakp^{[p]}_{i,n}}$ which is naturally isomorphic to $\Gamma(\rmZ_{d}(\overline{B}),\calO_{\lambda_{i}}(1))_{/\frakp^{[p]}_{i,n}}$; 
 \item $\rmT^{\bullet-}_{i, n}$ is given by $\coker(\nabla_{n})\cong \rmH^{1}(\FF_{p^{2}},\rmH^{1}(\overline{\rmX}_{i}\otimes{\FF^{\ac}_{p}}, \calO_{\lambda_{i}}(1))_{/\frakp^{[p]}_{i,n}})$ which is identified with $\Gamma(\rmZ_{d}(\overline{B}),\calO_{\lambda_{i}})_{/\frakp^{[p]}_{i, n}}$ via the map $\Phi^{\ast}_{i,n}$ in Definition \ref{Phi-ast-n}. 
 \end{enumerate}
 By Remark \ref{special-new-part}, we have $\rmT^{\bullet}_{i, n}\cong\rmH^{1}(\overline{\rmX}_{i}(p)\otimes{{\FF}^{\ac}_{p}}, \calO_{\lambda_{i}}(1))^{\bullet}_{/\frakp_{i,n}}$ and we have an isomorphism of exact sequences of $\calO_{\lambda_{i}, n}[\rmG_{\FF_{p^{2}}}]$-modules
\begin{equation}\label{comm-T}
\begin{tikzcd}
0 \arrow[r]& \rmT^{+}_{i, n} \arrow[r] \arrow[d]      & \rmT_{i, n}  \arrow[r] \arrow[d, "\alpha^{\ast}"] &  \rmT^{-}_{i, n }\arrow[d] \arrow[r] &0 \\
0 \arrow[r]&\rmT^{\bullet+}_{i, n} \arrow[r] & \rmT^{\bullet}_{i, n}\arrow[r] & \rmT^{\bullet-}_{i, n}\arrow[r] &0\\  
\end{tikzcd}
\end{equation}
where the first vertical map is $\Psi_{i,n}$ and the third vertical map is the identity map on $\rmH^{1}(\FF_{p^{2}},\rmH^{1}(\overline{\rmX}_{i}\otimes{\FF^{\ac}_{p}}, \calO_{\lambda_{i}}(1))_{/\frakp^{[p]}_{i,n}})$.

Since the integral model $\interX^{3}_{d}$ of $\rmX^{3}_{d}$ aquires good reduction at $p$, $\rmM_{n}(\triplef, d)$ is unramified at $p$ as an $\calO_{\undlamb, n}[\rmG_{\QQ}]$-module and therefore it makes sense to consider $\rmH^{1}(\FF_{p^{2}}, \rmM_{n}(\triplef, d)(-1))$. The following is the analogue of Proposition \ref{level-raise-curve} for the threefold $\rmX^{3}_{d}$ which will be referred to as the unramified level raising theorem for the triple product of Shimura curves.
\begin{proposition}[unramified level raising for triple product]\label{level-raising-decomp}
Let $p$ be an \emph{$n$-admissible} prime for the triple $\triplef$. We assume that each maximal ideal in the triple $\frakm_{\triplef}=(\frakm_{1}, \frakm_{2}, \frakm_{3})$ satisfies Assumption \ref{assump}. 
\begin{enumerate}
\item The map $\underline{\mathbf{a}}^{\ast}$ induces an isomorphism
\begin{equation*}
\underline{\mathbf{a}}^{\ast}: \rmM_{n}(\triplef, d)\cong \rmM^{[p]}_{n}(\triplef, d)^{\bullet}
\end{equation*}
of $\calO_{\undlamb, n}[\rmG_{\QQ}]$-modules which are both unramified at $p$.

\item We have isomorphisms $\Phi_{\triplef, n}$ and $\Phi^{[p]}_{\triplef, n}$ of $\calO_{\undlamb, n}[\rmG_{\FF_{p^{2}}}]$ modules which fit in the following commutative diagram
\begin{equation*}
\begin{tikzcd}
\rmH^{1}(\FF_{p^{2}}, \rmM_{n}(\triplef, d)(-1)) \arrow[r, "\Phi_{\triplef, n}"] \arrow[d, "\underline{\mathbf{a}}^{\ast}"] & \threesum(\threetensor\Gamma(\rmZ_{d}(\overline{B}),\calO_{\lambda_{i}})_{/\frakp^{[p]}_{i,n}})\arrow[d, Rightarrow, no head] \\
\rmH^{1}(\FF_{p^{2}}, \rmM^{[p]}_{n}(\triplef, d)^{\bullet}(-1))\arrow[r, "\Phi^{[p]}_{\triplef, n}"] & \threesum(\threetensor\Gamma(\rmZ_{d}(\overline{B}),\calO_{\lambda_{i}})_{/\frakp^{[p]}_{i,n}} ).                      
\end{tikzcd}
\end{equation*}
\end{enumerate}
\end{proposition}

\begin{proof}
The first part follows from Proposition \ref{new-old-part} and Lemma \ref{new-part-threefold}. The second part follows from the split exact sequences in Proposition \ref{canonical-decomposition} and Proposition \ref{new-old-part}. Indeed, for $(?,\square)\in\{(\emptyset, \emptyset), ([p], \bullet)\}$, we have
\begin{equation*}
\begin{aligned}
\rmM^{?}_{n}(\triplef, d)&= \rmT^{\square}_{1, n}\otimes\rmT^{\square}_{2, n}\otimes \rmT^{\square}_{3, n}\\
&\cong\underset{?_{\ast}\in\{\pm\}}\oplus\rmT^{\square?_{1}}_{1, n}\otimes \rmT^{\square?_{2}}_{2, n}\otimes\rmT^{\square?_{3}}_{3, n} \\
&\cong \rmZ_{n}(\triplef, d)\oplus  \rmZ^{\oplus 3}_{n}(\triplef, d)(1) \oplus  \rmZ^{\oplus 3}_{n}(\triplef, d)(2)\oplus  Z_{n}(\triplef, d)(3)\\
\end{aligned}
\end{equation*}
where we put $\rmZ_{n}(\triplef, d)=\threetensor\Gamma(\rmZ_{d}(\overline{B}),\calO_{\lambda_{i}})_{/\frakp^{[p]}_{i,n}}$. It follows then that
\begin{equation*}
\begin{aligned}
\rmH^{1}(\FF_{p^{2}}, \rmM^{?}_{n}(\triplef, d)(-1))&\cong\rmH^{1}(\FF_{p^{2}}, \rmZ_{n}(\triplef, d)(-1)\oplus \rmZ^{\oplus 3}_{n}(\triplef, d) \oplus \rmZ^{\oplus 3}_{n}(\triplef, d)(1)\oplus \rmZ_{n}(\triplef, d)(2))\\ 
&\cong\rmH^{1}(\FF_{p^{2}}, \rmZ^{\oplus 3}_{n}(\triplef, d))\\
&\cong\threesum(\threetensor\Gamma(\rmZ_{d}(\overline{B}),\calO_{\lambda_{i}})_{/\frakp^{[p]}_{i,n}})\\
\end{aligned}
\end{equation*}
and we define the composite of these isomorphisms as $\Phi^{?}_{\triplef, n}$. The commutativity of the diagram follows from the commutativity of \ref{comm-T}.
\end{proof}

For $(?,\square)\in\{(\emptyset, \emptyset), ([p], \bullet)\}$, we formally define the \emph{weight zero part} $\bfW_{0}\rmM^{?}_{n}(\triplef, d)(-1)$ of $\rmM^{?}_{n}(\triplef, d)(-1)=\threetensor\rmT^{\square}_{i,n}(-1)$ using the split exact sequence \ref{fil-of-T} and the split exact sequence \ref{fil-T-new} such that we have
\begin{equation*}
\begin{aligned}
\bfW_{0}\rmM^{?}_{n}(\triplef, d)(-1)&=\rmT^{\square-}_{1, n}\otimes\rmT^{\square-}_{2, n}\otimes\rmT^{\square+}_{3, n}(-1)\\
&\oplus\rmT^{\square-}_{1, n}\otimes\rmT^{\square+}_{2, n}\otimes\rmT^{\square-}_{3, n}(-1)\\
&\oplus\rmT^{\square-}_{1, n}\otimes\rmT^{\square-}_{2, n}\otimes\rmT^{\square+}_{3, n}(-1)\\
&\cong\threesum (\threetensor\Gamma(\rmZ_{d}(\overline{B}),\calO_{\lambda_{i}})_{/\frakp^{[p]}_{i,n}}).\\
\end{aligned}
\end{equation*}

We note that the map $\Phi_{\triplef, n}$ in the second part of the above theorem comes from composing the following series of maps
\begin{equation*}
\begin{aligned}
\rmH^{1}(\FF_{p^{2}}, \rmM_{n}(\triplef, d)(-1))&\rightarrow \rmH^{1}(\FF_{p^{2}}, {\bfW}_{0}\rmM_{n}(\triplef, d)(-1))\\
&\cong \rmT^{-}_{1, n}\otimes\rmT^{-}_{2, n}\otimes \rmT^{+}_{3, n}(-1)\\
&\oplus\rmT^{-}_{1, n}\otimes\rmT^{+}_{2, n}\otimes \rmT^{-}_{3, n}(-1)\\
&\oplus\rmT^{+}_{1, n}\otimes\rmT^{-}_{2, n}\otimes \rmT^{-}_{3, n}(-1)\\
&\cong \threesum(\threetensor\Gamma(\rmZ_{d}(\overline{B}),\calO_{\lambda_{i}})_{/\frakp^{[p]}_{i,n}})\\
\end{aligned}
\end{equation*}
where we note that the last isomorphism is given by $(\Phi^{\ast}_{n}\otimes\Phi^{\ast}_{n}\otimes\Psi_{n})\oplus(\Phi^{\ast}_{n}\otimes\Psi_{n}\otimes \Phi^{\ast}_{n})\oplus(\Psi_{n}\otimes\Phi^{\ast}_{n}\otimes\Phi^{\ast}_{n})$. Similarly, the map $\Phi^{[p]}_{\triplef, n}$ factors as
\begin{equation*}
\begin{aligned}
\rmH^{1}(\FF_{p^{2}}, \rmM^{[p]}_{n}(\triplef, d)^{\bullet}(-1))&\rightarrow \rmH^{1}(\FF_{p^{2}}, {\bfW}_{0}\rmM^{[p]}_{n}(\triplef, d)^{\bullet}(-1))\\
&\cong \rmT^{\bullet-}_{1, n}\otimes\rmT^{\bullet-}_{2, n}\otimes \rmT^{\bullet+}_{3, n}(-1)\\
&\oplus\rmT^{\bullet-}_{1, n}\otimes\rmT^{\bullet+}_{2, n}\otimes \rmT^{\bullet-}_{3, n}(-1)\\
&\oplus\rmT^{\bullet+}_{1, n}\otimes\rmT^{\bullet-}_{2, n}\otimes \rmT^{\bullet-}_{3, n}(-1)\\
&\cong \threesum(\threetensor\Gamma(\rmZ_{d}(\overline{B}),\calO_{\lambda_{i}})_{/\frakp^{[p]}_{i,n}}).\\
\end{aligned}
\end{equation*}
Note that last isomorphism is given by $(\Phi^{\ast}_{n}\otimes\Phi^{\ast}_{n}\otimes 1)\oplus(\Phi^{\ast}_{n}\otimes 1\otimes \Phi^{\ast}_{n})\oplus(1\otimes\Phi^{\ast}_{n}\otimes\Phi^{\ast}_{n})$.

The previous discussion can be summarized in the following commutative diagram
\begin{equation}\label{a-commu}
\begin{tikzcd}
\rmH^{1}(\FF_{p^{2}}, \rmM_{n}(\triplef, d)(-1) \arrow[r, "\underline{\mathbf{a}}^{\ast}"] \arrow[d] & \rmH^{1}(\FF_{p^{2}}, \rmM^{[p]}_{n}(\triplef, d)^{\bullet}(-1) )\arrow[d] \\
\rmH^{1}(\FF_{p^{2}},{\bfW}_{0}\rmM_{n}(\triplef, d)(-1))\arrow[r, "{\bfW}_{0}\underline{\mathbf{a}}^{\ast}"] \arrow[d] & \rmH^{1}(\FF_{p^{2}},{\bfW}_{0}\rmM^{[p]}_{n}(\triplef, d)^{\bullet}(1)) \arrow[d] \\
\threesum(\threetensor\Gamma(\rmZ_{d}(\overline{B}),\calO_{\lambda_{i}})_{/\frakp^{[p]}_{i,n}})   \arrow[r, Rightarrow, no head]       &\threesum(\threetensor\Gamma(\rmZ_{d}(\overline{B}),\calO_{\lambda_{i}})_{/\frakp^{[p]}_{i,n}}) 
\end{tikzcd}
\end{equation}
The composite of the vertical maps on the left is nothing but $\Phi_{\triplef, n}$ whereas the composite of the right vertical map is nothing but $\Phi^{[p]}_{\triplef, n}$.
The horizontal map ${\bfW}_{0}\underline{\mathbf{a}}^{\ast}$ on the second row is thus the isomorphism given by $(1\otimes1\otimes\Psi_{n})\oplus(1\otimes\Psi_{n}\otimes 1)\oplus (\Psi_{n}\otimes1\otimes1)$.

\begin{remark}\label{special-new}
As  an $\calO_{\undlamb, n}[\rmG_{\FF_{p^{2}}}]$-module, we can identify $\rmM_{n}(\triplef, d)(-1)$ with 
$\rmH^{3}(\overline{\rmX}^{3}_{d}\otimes{\FF}^{\ac}_{p}, \calO_{\undlamb}(2))_{/\frakp^{[p]}_{i,n}}$
by the smooth base change theorem. On the other hand, as an $\calO_{\undlamb, n}[\rmG_{\FF_{p^{2}}}]$-module, $\rmM^{[p]}_{n}(\triplef, d)^{\bullet}(-1)$ can be identified with 
\begin{equation*}
\rmH^{3}(\overline{\rmX}^{3}_{d}(p)\otimes{\FF}^{\ac}_{p}, \calO_{\undlamb}(2))^{\bullet}_{/{\frakp}^{[p]}_{\triplef, n}}=\ker[\rmH^{3}(\overline{\rmX}^{3}_{d}(p)\otimes{\FF}^{\ac}_{p}, \calO_{\undlamb}(2))_{/{\frakp}^{[p]}_{\triplef, n}}\xrightarrow{\underline{\mathbf{a}}_{\ast}} \rmH^{3}(\overline{\rmX}^{3}_{d}\otimes{\FF}^{\ac}_{p}, \calO_{\undlamb}(2))^{\oplus 8}_{/{\frakp}_{\triplef, n}}]
\end{equation*}
by  Remark \ref{special-new-part} and Lemma \ref{new-part-threefold}. Therefore the isomorphism $\underline{\mathbf{a}}^{\ast}$ in Theorem \ref{level-raising-decomp} can be regarded as an isomorphism of the cohomologies on the special fibers. 
\end{remark}

\begin{corollary}\label{level-raise-Q}
Let $p$ be an $n$-admissible prime for $\triplef$ and suppose that each maximal ideal in the triple $\frakm_{\triplef}=(\frakm_{1}, \frakm_{2}, \frakm_{3})$ satisfies Assumption \ref{assump}. Then we have the following isomorphism
\begin{equation*}
\Phi_{\triplef, n}: \rmH^{1}_{\mathrm{fin}}(\QQ_{p}, \rmM_{n}(\triplef, d)(-1))\cong \threesum(\threetensor\Gamma(\rmZ_{d}(\overline{B}),\calO_{\lambda_{i}})_{/\frakp^{[p]}_{i, n}}).
\end{equation*}
\end{corollary}
\begin{proof}
Since $\rmM_{n}(\triplef, d)$ is unramified, we have 
\begin{equation*}
\rmH^{1}_{\mathrm{fin}}(\QQ_{p}, \rmM_{n}(\triplef, d)(-1))\cong\rmH^{1}(\FF_{p}, \rmM_{n}(\triplef, d)(-1)).  
\end{equation*}
By the same argument as in \cite[Proposition 3.8]{Ri-100}, the non-trivial element in $\Gal(\FF_{p^{2}}/\FF_{p})$ acts on the space $\Gamma(\rmZ_{d}(\overline{B}),\calO_{\lambda_{i}})_{/\frakp^{[p]}_{i, n}}$ by $\epsilon_{i}$. Therefore the non-trivial element in $\Gal(\FF_{p^{2}}/\FF_{p})$ acts on the space
$\threesum(\threetensor\Gamma(\rmZ_{d}(\overline{B}),\calO_{\lambda_{i}})_{/\frakp^{[p]}_{i, n}})$
by the product of the sign $(\epsilon_{1}, \epsilon_{2}, \epsilon_{3})$ which is $1$ by the definition of an $n$-admissible prime for $\triplef$. Since the map $\Phi_{\triplef, n}$ commutes with the action of $\Gal(\FF_{p^{2}}/\FF_{p})$ by construction, the conclusion immediately follows. 
\end{proof}

\subsection{Reciprocity laws for Gross--Kudla--Schoen diagonal cycles} In this subsection, we will always fix an $n$-admissible prime $p$ for $\triplef$. Let $\frakm_{\triplef}=(\frakm_{1}, \frakm_{2}, \frakm_{3})$ be the triple of maximal ideals that all satisfy Assumption \ref{assump}. Recall we have the Shimura curve $\rmX_{d}$ over $\QQ$ and its integral model $\interX_{d}$ over $\ZZ[1/Nd]$. We consider the diagonal embedding  $\theta: \interX_{d}\rightarrow \interX^{3}_{d}$ of $\interX_{d}$ whose restriction to its generic fiber $\rmX_{d}$ will be denoted by the same notation. This gives a class 
\begin{equation*}
\Delta_{d}=\theta_{*}[\rmX_{d}]\in \mathrm{CH}^{2}(\rmX^{3}_{d}) 
\end{equation*}
which will be referred to as the \emph{Gross--Kudla--Schoen diagonal cycle}. We consider the cycle class map 
$\mathrm{cl}: \mathrm{CH}^{2}(\rmX^{3}_{d})\rightarrow \rmH^{4}(\rmX^{3}_{d}, \calO_{\undlamb}(2))$.
Since $ \rmH^{4}(\rmX^{3}_{d}\otimes{\QQ^{\ac}}, \calO_{\undlamb}(2))_{\frakm_{\triplef}}$ is zero by the K\"{u}nneth formula and our assumption that $\frakm_{i}$ is abosultely irreducible, the cycle class map induces the $\underline{\lambda}$-adic Abel--Jacobi map
\begin{equation}\label{AJ}
\mathrm{AJ}_{\triplef}: \mathrm{CH}^{2}(\rmX^{3}_{d})\rightarrow \rmH^{1}(\QQ,  \rmH^{3}(\rmX^{3}_{d}\otimes{\QQ^{\ac}}, \calO_{\undlamb}(2))_{\frakm_{\triplef}})\cong \rmH^{1}(\QQ,  \rmM(\triplef, d)(-1))
\end{equation} 
and the mod $\underline{\lambda}^{n}$ Abel--Jacobi map
\begin{equation}\label{AJ-n}
\mathrm{AJ}_{\triplef, n}: \mathrm{CH}^{2}(\rmX^{3}_{d})\rightarrow \rmH^{1}(\QQ,  \rmH^{3}(\rmX^{3}_{d}\otimes{\QQ^{\ac}}, \calO_{\undlamb}(2))_{/\frakp_{\triplef,n}})\cong \rmH^{1}(\QQ,  \rmM_{n}(\triplef, d)(-1))
\end{equation} 
via the Hochschild--Serre spectral sequence of $\rmH^{4}(\rmX^{3}_{d}, \calO_{\undlamb}(2))$. Let 
\begin{equation*}
\Theta_{n}(\triplef,d)=\mathrm{AJ}_{\triplef, n}(\Delta_{d})\in \rmH^{1}(\QQ,  \rmM_{n}(\triplef, d)(-1)).
\end{equation*}
These cohomological classes will be referred to as the \emph{Gross--Schoen--Kudla diagonal classess}. 

\begin{lemma}
Let $p$ be an $n$-admissible prime for $\triplef$. The class 
\begin{equation*}
\loc_{p}(\Theta_{n}(\triplef, d))\in \rmH^{1}(\QQ_{p}, \rmM_{n}(\triplef, d)(-1)) 
\end{equation*}
lies in the finite part $\rmH^{1}_{\mathrm{fin}}(\QQ_{p}, \rmM_{n}(\triplef, d)(-1))\cong \rmH^{1}(\FF_{p}, \rmM_{n}(\triplef, d)(-1))$ of  $\rmH^{1}(\QQ_{p}, \rmM_{n}(\triplef, d)(-1))$. 
\end{lemma}
\begin{proof}
This follows immediately from the fact that the threefold $\interX^{3}_{d}$ admits good reduction at the prime $p$. See \cite[Lemma 3.4]{Liu-HZ}. 
\end{proof}
The above lemma allows us to consider the class $\loc_{p}(\Theta_{n}(\triplef, d))$ as an element in $\threesum(\threetensor\Gamma(\rmZ_{d}(\overline{B}), \calO_{\lambda_{i}})_{/\frakp^{[p]}_{i,n}})$
via the isomorphism $\Phi_{\triplef, n}$ in Corollary \ref{level-raise-Q}. We will denote by 
$\loc^{(j)}_{p}(\Theta_{n}(\triplef, d))$
the component of $\loc_{p}(\Theta_{n}(\triplef, d))$ in the $j$-th copy of $\threesum(\threetensor\Gamma(\rmZ_{d}(\overline{B}), \calO_{\lambda_{i}})_{/\frakp^{[p]}_{i, n}})$. 

On the other hand, we have a bilinear pairing 
\begin{equation*}
(\hphantom{a}, \hphantom{b}): \threetensor\Gamma(\rmZ_{d}(\overline{B}), \calO_{\lambda_{i}})\times  \threetensor\Gamma(\rmZ_{d}(\overline{B}), E_{\lambda_{i}}/\calO_{\lambda_{i}})\rightarrow \calO_{\undlamb}
\end{equation*}
given by the formula
\begin{equation*}
(\threetensor\zeta_{i}, \threetensor\phi_{i})=\sum_{(z_{1}, z_{2}, z_{3})\in Z_{d}(\overline{B})^{3}}\zeta_{1}\phi_{1}(z_{1})\otimes\zeta_{2}\phi_{2}(z_{2})\otimes\zeta_{3}\phi_{3}(z_{3})
\end{equation*}
for $\threetensor\zeta_{i}\in  \threetensor\Gamma(\rmZ_{d}(\overline{B}), \calO_{\lambda_{i}})$ and $\threetensor\phi_{i}\in \threetensor\Gamma(\rmZ_{d}(\overline{B}), E_{\lambda_{i}}/\calO_{\lambda_{i}})$ induced by the Poincar\'e duality on $\rmZ_{d}(\overline{B})^{3}$. This paring gives rise naturally to a pairing 
\begin{equation*}
 (\hphantom{a}, \hphantom{b}): \threetensor\Gamma(\rmZ_{d}(\overline{B}), \calO_{\lambda_{i}})_{/\frakp^{[p]}_{i, n}}\times\threetensor\Gamma(\rmZ_{d}(\overline{B}), E_{\lambda_{i}}/\calO_{\lambda_{i}})[\frakp^{[p]}_{i, n}]\rightarrow \calO_{\undlamb, n}.
\end{equation*}

Following the terminology in \cite{BD}, we will refer to the formula in the theorem below as the \emph{second explicit reciprocity law} for the Gross--Kudla--Schoen diagonal classes.
\begin{theorem}[The second reciprocity law]\label{2-law}
Let $p$ be an $n$-admissible prime for $\triplef$. Suppose each maximal ideal in $\frakm_{\triplef}$ satisfies Assumption \ref{assump}. Then the following formula
\begin{equation*}
(\loc^{(j)}_{p}(\Theta_{n}(\triplef, d)), \phi_{1}\otimes\phi_{2}\otimes\phi_{3})=\sum_{z\in \Delta_{d}(\overline{B})} \phi_{1}(z)\otimes \phi_{2}(z)\otimes \phi_{3}(z)
\end{equation*}
holds for any $\phi_{1}\otimes \phi_{2}\otimes \phi_{2}\in \threetensor\Gamma(Z_{d}(\overline{B}),E_{\lambda_{i}}/ \calO_{\lambda_{i}})[\frakp^{[p]}_{i, n}]$ and $j\in\{1, 2, 3\}$. 
\end{theorem}

Let $\overline{\theta}: \overline{\rmX}_{d}\rightarrow \overline{\rmX}^{3}_{d}$ be the map induced by $\theta: \interX_{d}\rightarrow\interX^{3}_{d}$ on the special fiber. Consider the Abel--Jacobi map 
\begin{equation*}
{\AJ}_{\triplef, n}: \mathrm{CH}^{2}(\overline{\rmX}^{3}_{d})\rightarrow \rmH^{1}(\FF_{p^{2}}, \rmM_{n}(\triplef, d)(-1)) 
\end{equation*}
for $\rmXbar^{3}_{d}$ constructed in the same way as in \ref{AJ-n}. Let $\overline{\Delta}_{d}=[\overline{\theta}_{\ast}\overline{\rmX}_{d}]\in  \mathrm{CH}^{2}(\overline{\rmX}^{3}_{d})$ be the diagonal element, we denote by 
\begin{equation}\label{modtheta}
\overline{\Theta}_{n}(\triplef, d) 
\end{equation}
the class given by $\mathrm{AJ}_{\triplef, n}(\overline{\Delta}_{d})\in \rmH^{1}(\FF_{p^{2}}, \rmM_{n}(\triplef,d)(-1))$ and define 
\begin{equation}\label{j-theta}
\overline{\Theta}_{n}^{(j)}(\triplef, d) 
\end{equation}
to be the projection to the $j$-th component of $\threesum(\threetensor\Gamma(\rmZ_{d}(\overline{B}), \calO_{\lambda_{i}})_{/\frakp^{[p]}_{i,n}})$
using the isomorphism $\Phi_{\triplef, n}$ for $j\in\{1, 2, 3\}$.  Since $\interX^{3}_{d}$ has good reduction at $p$ and thus $\rmM_{n}(\triplef,d)(-1)$ is unramified at $p$, the class $\overline{\Theta}_{n}(\triplef, d)$ agrees with $\loc_{p}(\Theta_{n}(\triplef, d))$. Similarly, we can identify  $\overline{\Theta}^{(j)}_{n}(\triplef, d)$ with $\loc^{(j)}_{p}(\Theta_{n}(\triplef, d))$.
To finish the proof, we will show, without loss of generality, 
\begin{equation}\label{1-finish}
\begin{aligned}
(\bar{\Theta}^{(1)}_{n}(\triplef, d), \phi_{1}\otimes\phi_{2}\otimes\phi_{3})=\sum_{z\in \Delta_{d}(\overline{B})} \phi_{1}(z)\otimes \phi_{2}(z)\otimes \phi_{3}(z).
\end{aligned}
\end{equation}

Recall that the class  $\overline{\Theta}^{(1)}_{n}(\triplef, d)$ is the image of $\overline{\Theta}_{n}(\triplef, d)$ under the following maps 
\begin{equation*}
\begin{aligned}
\rmH^{1}(\FF_{p^{2}}, \rmM_{n}(\triplef, d)(-1))&= \rmH^{1}(\FF_{p^{2}}, \rmT_{1, n}\otimes\rmT_{2, n}\otimes \rmT_{3, n}(-1))\\
&\rightarrow\rmH^{1}(\FF_{p^{2}},\rmT^{+}_{1, n}(-1)\otimes \rmT^{-}_{2, n}\otimes\rmT^{-}_{3, n}) \\
&\xrightarrow{\sim} \rmT^{+}_{1, n}(-1)\otimes \rmT^{-}_{2, n}\otimes\rmT^{-}_{3, n} \\
&\xrightarrow{\sim}\Gamma(\rmZ_{d}(\overline{B}),\calO_{\lambda_{1}})_{/\frakp^{[p]}_{1, n}}\otimes \Gamma(\rmZ_{d}(\overline{B}),\calO_{\lambda_{2}})_{/\frakp^{[p]}_{2, n}}\otimes \Gamma(\rmZ_{d}(\overline{B}),\calO_{\lambda_{3}})_{/\frakp^{[p]}_{3, n}}\\
&\xrightarrow{\sim}\Gamma(\rmZ_{d}(\overline{B})^{3}, \calO_{\undlamb})_{/\triplepnp}.
\end{aligned}
\end{equation*}
We note that the second to last isomorphism is given by $\Psi_{n}\otimes\Phi^{\ast}_{n}\otimes\Phi^{\ast}_{n}$.
We would like to prove that the element $\overline{\Theta}^{(1)}_{n}(\triplef, d)\in \Gamma(\rmZ_{d}(\overline{B})^{3}, \calO_{\undlamb})_{/\frakp^{[p]}_{\triplef,n}}$ agrees with the characteristic function 
\begin{equation*}
\mathbf{1}_{\overline{B}}\in \Gamma(Z_{d}(\overline{B})^{3}, \calO_{\undlamb})_{/\triplepnp} 
\end{equation*}
of the diagonal cycle 
\begin{equation}\label{def-diagonal}
\Delta_{d}(\overline{B})=\vartheta_{\ast}\rmZ_{d}(\overline{B}) 
\end{equation}
under the diagonal embedding $\vartheta:Z_{d}(\overline{B})\rightarrow \rmZ_{d}(\overline{B})^{3}$.  Once this is proved, we have
\begin{equation*}
\begin{aligned}
(\bar{\Theta}^{(1)}_{n}(\triplef, d), \phi_{1}\otimes\phi_{2}\otimes\phi_{3})&=(\mathbf{1}_{\overline{B}}, \phi_{1}\otimes\phi_{2}\otimes\phi_{3})\\
&=\sum_{z\in \Delta_{d}(\overline{B})} \phi_{1}(z)\otimes \phi_{2}(z)\otimes \phi_{3}(z).\\
\end{aligned}
\end{equation*} 

To prove this, we will use the isomorphism $\underline{\mathbf{a}}^{\ast}:\rmM_{n}(\triplef, d)(-1)\cong\rmM^{[p]}_{n}(\triplef, d)^{\bullet}(-1)$ in Lemma \ref{level-raising-decomp} and consider the space $\rmM^{[p]}_{n}(\triplef, d)^{\bullet}(-1)$ as a subspace of $\rmH^{3}(\overline{\rmX}^{3}_{d}(p)\otimes\FF^{\ac}_{p}, \rmR\Psi(\calO_{\undlamb})(2))_{/\frakp^{[p]}_{\triplef,n}}$.
This leads to the study the monodromy filtration on $\rmH^{3}(\overline{\rmX}^{3}_{d}(p)\otimes\FF^{\ac}_{p}, \rmR\Psi(\calO_{\undlamb})(2))_{\frakm^{[p]}_{\triplef}}$ using the localized weight spectral sequence which we will treat in the next section.

\section{Proof of the second reciprocity laws}
\subsection{Semistable model of triple product of Shimura curves}
Let $p$ be a prime. We recall the construction of a strict semistable model of the triple product $\rmX^{3}_{d}(p)$  over $\ZZ_{p^{2}}$ following \cite{GS}. 
Let $\mathfrak{X}^{3}_{d}(p)$ be the triple fiber product of $\mathfrak{X}_{d}(p)$ over $\ZZ_{p^{2}}$.  First, we analyze the reduction of  $\mathfrak{X}^{3}_{d}(p)$. We denote by $\overline{\rmX}^{3}_{d}(p)$ the special fiber of  $\mathfrak{X}^{3}_{d}(p)$. By Lemma \ref{gamma0p}, we know each $\overline{\rmX}_{d}(p)$ can be described as the union $\overline{\rmX}_{+}\cup \overline{\rmX}_{-}$ where $\rmXbar_{+}:=\rmXbar_{+,d}$ and $\rmXbar_{-}:=\rmXbar_{-,d}$ are both isomorphic to $\rmXbar_{d}$ glued along the supersingular locus $\rmXbar_{\pm}:=\rmXbar_{\pm,d}\cong \rmXbar^{\ss}_{d}$ and therefore the special fiber $\overline{\rmX}^{3}_{d}(p)$ can be described by the cube given below.
\begin{equation*}
\begin{tikzpicture}
  \matrix (m) [matrix of math nodes, row sep=2em,
    column sep=2em]{
     &  {\rmXbar}^{045}_{[+-+]} & &  {\rmXbar}^{024}_{[+++]}\\
         {\rmXbar}^{015}_{[--+]}& &    {\rmXbar}^{012}_{[-++]} & \\
    &  {\rmXbar}^{345}_{[+--]}& &  {\rmXbar}^{234}_{[++-]} \\
        {\rmXbar}^{135}_{[---]} & &     {\rmXbar}^{123}_{[-+-]} & \\}  ;
     \path
    (m-1-2) edge (m-1-4) edge (m-2-1) edge (m-3-2) 
    (m-1-4) edge (m-3-4) edge (m-2-3) 
    (m-2-1) edge [-,line width=6pt,draw=white] (m-2-3) edge (m-2-3) edge (m-4-1) 
    (m-3-2) edge  (m-3-4) edge (m-4-1) 
    (m-4-1) edge (m-4-3) 
    (m-3-4) edge (m-4-3) 
    (m-2-3) edge [-,line width=6pt,draw=white] (m-4-3) edge (m-4-3) ;
\end{tikzpicture}
\end{equation*}
We will explain the meaning of the simplices in this cube.
\begin{itemize}
\item The $0$-simplices are the vertices of the cube. They correspond to $3$-dimensional strata in $\rmXbar^{3}_{d}(p)$. Consider the vertex labeled by $\rmXbar^{123}_{[-+-]}$ for example. The superscript $123$ has no real meaning and is simply used for ordering the vertices. This ordering is inherited from Liu's paper \cite{Liu-cubic} where the labels have real meanings in terms of achimedean places of a cubic field. The subscript $[-+-]$ means that $\rmXbar^{123}_{[-+-]}$ is of the form
\begin{equation*}
\rmXbar_{-}\times \rmXbar_{+}\times \rmXbar_{-}. 
\end{equation*}
\item The $1$-simplices are the edges of the cube. They correspond to $2$-dimensional strata in $\rmXbar^{3}_{d}(p)$. For example, we will label the edge between  ${\rmXbar}^{135}_{[---]}$ and ${\rmXbar}^{123}_{[-+-]}$ by $\rmXbar^{1235}_{[-\pm-]}$. This means we will take the union on the superscript and take the intersection on the subscript. Then $\rmXbar^{1235}_{[-\pm-]}$ is of the form
\begin{equation*}
\rmXbar_{-}\times \rmXbar_{\pm} \times \rmXbar_{-}.
\end{equation*}
\item The $2$-simplices are the faces of the cube. They correspond to $1$-dimensional strata in $\rmXbar^{3}_{d}(p)$. We use similar convention as in the last point. For example
\begin{equation*}
\rmXbar^{01235}_{[-\pm\pm]}=\rmXbar_{-}\times \rmXbar_{\pm} \times  \rmXbar_{\pm}.
\end{equation*}
\item Finally, the $3$-simplex is the zero dimensional components given by 
\begin{equation*}
\rmXbar^{012345}_{[\pm\pm\pm]}=\rmXbar_{\pm}\times \rmXbar_{\pm} \times \rmXbar_{\pm}.
\end{equation*}
\end{itemize}
We will sometimes drop the subscripts or the superscripts to simplify the notations. By an easy computation on the local rings, we see that $\mathfrak{X}^{3}_{d}(p)$ is not regular. Following the procedure in \cite[Example 6.15]{GS}, we can obtain a strict semistable model denoted by $\mathfrak{Y}_{d}(p)$  of $\mathfrak{X}^{3}_{d}(p)$ over $\ZZ_{p^{2}}$. More precisely, to obtain $\mathfrak{Y}_{d}(p)$, we blow-up $\mathfrak{X}^{3}_{d}(p)$ along the closed subscheme $\rmXbar^{135}$, then we blow-up the strict transform of $\rmXbar^{024}$. We denote by $\pi: \mathfrak{Y}_{d}(p)\rightarrow \mathfrak{X}^{3}_{d}(p)$ the natural morphism between these two schemes given by the aforementioned process. The generic fiber $\mathfrak{Y}_{d}(p)$ agree with the generic fiber $\rmX^{3}_{d}(p)$ of $\mathfrak{X}^{3}_{d}(p)$. The special fiber of $\mathfrak{Y}_{d}(p)$ will be denoted by $\rmYbar=\rmYbar_{d}(p)$ and its reduction complex can be described by the following cube. The densely dotted line on the cube correspond to the new intersections between three dimensional strata caused by the blow-ups and they give new two dimensional strata. 
\begin{equation}\label{primitive-cube}
\begin{tikzpicture}
  \matrix (m) [matrix of math nodes, row sep=2em,
    column sep=2em]{
     &  {\rmYbar}^{045}_{[+-+]} \vphantom{f^{*}}  & &   {\rmYbar}^{024}_{[+++]}\hphantom{EEE}\\
      {\rmYbar}^{015}_{[--+]} \vphantom{f^{*}}& &         {\rmYbar}^{012}_{[-++]} \hphantom{Ee}& \\
    & {\rmYbar}^{345}_{[+--]} \vphantom{U}  & &       {\rmYbar}^{234}_{[++-]} \vphantom{E}\\
       {\rmYbar}^{135}_{[---]}  \vphantom{M} & &    {\rmYbar}^{123}_{[-+-]} \vphantom{N} & \\}  ;
     \path
    (m-1-2) edge (m-1-4) edge (m-2-1) edge (m-3-2) edge[densely dotted](m-4-1) 
    (m-1-4) edge (m-3-4) edge (m-2-3) edge[densely dotted] (m-3-2) edge[densely dotted](m-4-1)  edge[densely dotted] (m-2-1) edge[densely dotted] (m-4-3)
    (m-2-1) edge [-,line width=6pt,draw=white] (m-2-3) edge (m-2-3) edge (m-4-1) 
    (m-3-2) edge  (m-3-4) edge (m-4-1) 
    (m-4-1) edge (m-4-3) 
    (m-3-4) edge (m-4-3) edge[densely dotted](m-4-1)
    (m-2-3) edge [-,line width=6pt,draw=white] (m-4-3) edge (m-4-3) edge[densely dotted](m-4-1) ;
\end{tikzpicture}
\end{equation}
We will explain the meaning of some of the simplices of this cube. 
\begin{itemize}
\item The $0$-simplices are the vertices of the cube. They correspond to $3$-dimensional strata in ${\rmYbar}$. These are the strict transform of the corresponding strata in $\rmXbar^{3}_{d}(p)$. For example, $\rmYbar^{012}$ is the strict transform of $\rmXbar^{012}$ under $\pi$. 

\item The $1$-simplices are the edges on the cube. They correspond to $2$-dimensional strata in ${\rmYbar}$. Notice that there are three types of edges: 
\begin{enumerate}
\item those correspond to the original edges in the cube for $\rmXbar^{3}_{d}(p)$, for example $\rmYbar^{0125}=\rmYbar^{012}\cap \rmYbar^{015}$; 
\item those correspond to the faces in the  cube for $\rmXbar^{3}_{d}(p)$, for example $\rmY^{01235}=\rmYbar^{012}\cap \rmYbar^{135}$; 
\item the one correspond to the ``main diagonal" of the cube $\rmYbar^{012345}=\rmYbar^{024}\cap \rmYbar^{135}$.
\end{enumerate}
\end{itemize}

\begin{proposition} \label{Ybar}
We can describe the strata of $\rmYbar$ in terms of the strata of $\rmXbar^{3}_{d}(p)$. The three dimensional strata are given by the following list.
\begin{enumerate}
\item The three dimensional stratum 
\begin{equation*}
\rmYbar^{i(i+1)(i+2)}
\end{equation*}
is the blow-up of $\rmXbar^{i(i+1)(i+2)}$ along the one dimensional strata $\rmXbar^{i(i+1)(i+2)(i+3)(i+5)}$ for $i\in \{0,1, 2, 3, 4, 5\}$. For example, the stratum $\rmYbar^{012}$ is the blow-up of $\rmXbar^{012}=\rmXbar_{[-++]}$ along $\rmX^{01235}=\rmX_{[-+\pm]}$ whose exceptional divisor is given by the $\PP^{1}$-bundle over $\rmXbar_{[-\pm\pm]}$;

\item The three dimensional stratum $$\rmYbar^{024}$$
is the blow-up of $\rmXbar^{024}$ along the zero dimensional stratum $\rmXbar^{012345}$ followed by the blow-up of the strict transform of $\rmXbar^{01234}\cup \rmXbar^{01245}\cup \rmXbar^{02345}$. The exceptional divisor is given by a $\PP^{2}$-bundle over $\rmXbar_{[\pm\pm\pm]}$ and a $\PP^{1}$-bundle over $\rmXbar_{[+\pm\pm]}\cup\rmXbar_{[\pm+\pm]}\cup\rmXbar_{[\pm\pm+]}$;

\item The three dimensional stratum $$\rmYbar^{135}$$
is the blow-up of $\rmXbar^{135}$ along the zero dimensional stratum $\rmXbar^{012345}$ followed by the blow-up of the strict transform of $\rmXbar^{01235}\cup \rmXbar^{01345}\cup \rmXbar^{12345}$. The exceptional divisor is given by a $\PP^{2}$-bundle over $\rmXbar_{[\pm\pm\pm]}$ and a $\PP^{1}$-bundle over $\rmXbar_{[-\pm\pm]}\cup\rmXbar_{[\pm-\pm]}\cup\rmXbar_{[\pm\pm-]}$.
\end{enumerate}

The two dimensional strata are given by the following list.
\begin{enumerate}
\item The two dimensional stratum $$\rmYbar^{i(i+1)(i+2)(i+3)}$$
maps isomorphically to $\rmXbar^{i(i+1)(i+2)(i+3)}$ for $i\in\{0, 1, 2, 3, 4, 5\}$.

\item The two dimensional stratum $$\rmYbar^{i(i+1)(i+2)(i+4)}$$
is the blow-up of $\rmXbar^{i(i+1)(i+2)(i+4)}$ along $\rmXbar^{012345}$ for $i\in\{0, 1, 2, 3, 4, 5\}$. For example, the stratum $\rmYbar^{0124}$ is the blow-up of $\rmXbar^{0124}=\rmXbar_{[\pm++]}$ along $\rmXbar_{[\pm\pm\pm]}$ whose exceptional locus is given by a $\PP^{1}$-bundle over $\overline{\rmX}_{[\pm\pm\pm]}$.

\item The two dimensional stratum $$\rmYbar^{i(i+1)(i+2)(i+3)(i+5)}$$
is a $\PP^{1}$-bundle over $\rmXbar^{i(i+1)(i+2)(i+3)(i+5)}$ for $i\in\{0, 1, 2, 3, 4, 5\}$. In fact, $\rmYbar^{i(i+1)(i+2)(i+3)(i+5)}$
is the exceptional divisor of the blow-up 
$\pi: \rmYbar^{i(i+1)(i+2)}\rightarrow \rmXbar^{i(i+1)(i+2)}$. For example, $\rmYbar^{01235}$ is a $\PP^{1}$-bundle over $\rmXbar^{01235}=\rmXbar_{[-\pm\pm]}$.

\item The two dimensional stratum 
\begin{equation*}
\rmYbar^{012345}
\end{equation*}
is a $\PP^{2}$-bundle over $\rmXbar^{012345}=\rmXbar_{[\pm\pm\pm]}$.
\end{enumerate}
Moreover all the strata of dimension $2$ or dimension $3$ are given in the above list.

The one dimensional strata are given by the following list.
\begin{enumerate}
\item There is a stratum of the form $\rmXbar^{01234}=\rmXbar_{[\pm+\pm]}$ with multiplicity $2$, similarly for the stratum labelled by $01245$, $02345$, $01345$, $12345$, $01234$;
\item There is a stratum which is a $\PP^{1}$-bundle over $\rmXbar^{012345}=\rmXbar_{[\pm\pm\pm]}$ with multiplicity $6$.
\end{enumerate}
The zero dimensional strata are given by $6$ copies of $\rmXbar^{012345}=\rmXbar_{[\pm\pm\pm]}$.
\end{proposition}

\begin{proof}
The proof of these statements are exactly the same as these given in \cite[Proposition B.39]{Liu-cubic}. Although we are working with different Shimura varities, the underlying local models are the same. 
\end{proof}
\def\calOlamb{\calO_{\undlamb}}
\subsection{Monodromy filtration of the nearby cycle cohomology} In this subsection, we let $\Lambda=\calO_{\undlamb}$. Using the above Proposition, we can calculate all the terms in \ref{E1-primitive} for the localized weight spectral sequence for
\begin{equation*}
\rmH^{3}_{\Iw,\frakm}(\Lambda(2))=\rmH^{3}(\overline{\rmY}\otimes\FF^{\ac}_{p}, \rmR\Psi(\Lambda)(2))_{\frakm^{[p]}_{\triplef}}
\end{equation*}
at a triple of maximal ideals $\frakm^{[p]}_{\triplef}=(\frakm^{[p]}_{1}, \frakm^{[p]}_{2}, \frakm^{[p]}_{3})$ and hence we can also describe the monodromy filtration
\begin{equation*}
\begin{aligned}
0&\subset^{\Gr^{\rmM}_{3}\rmH^{3}_{\Iw,\frakm}(\Lambda(2))}\rmM_{3}\rmH^{3}_{\Iw,\frakm}(\Lambda(2))\subset^{\Gr^{\rmM}_{2}\rmH^{3}_{\Iw,\frakm}(\Lambda(2))} \rmM_{2}\rmH^{3}_{\Iw,\frakm}(\Lambda(2))\subset^{\Gr^{\rmM}_{1}\rmH^{3}_{\Iw,\frakm}(\Lambda(2))} \rmM_{1}\rmH^{3}_{\Iw,\frakm}(\Lambda(2))\\&\subset^{\Gr^{\rmM}_{0}\rmH^{3}_{\Iw,\frakm}(\Lambda(2))} \rmM_{0}\rmH^{3}_{\Iw,\frakm}(\Lambda(2))\subset^{\Gr^{\rmM}_{-1}\rmH^{3}_{\Iw,\frakm}(\Lambda(2))}\rmM_{-1}\rmH^{3}_{\Iw,\frakm}(\Lambda(2))\subset^{\Gr^{\rmM}_{-2}\rmH^{3}_{\Iw,\frakm}(\Lambda(2))} \rmM_{-2}\rmH^{3}_{\Iw,\frakm}(\Lambda(2))\\
&\subset^{\Gr^{\rmM}_{-3}\rmH^{3}_{\Iw,\frakm}(\Lambda(2))} \rmM_{-3}\rmH^{3}_{\Iw,\frakm}(\Lambda(2))=\rmH^{3}_{\Iw,\frakm}(\Lambda(2)) \\
\end{aligned}
\end{equation*}
of $\rmH^{3}_{\Iw,\frakm}(\Lambda(2))$. For the moment, we only assume that each maximal ideal in $\frakm^{[p]}_{\triplef}$ is absolutely irreducible.

\begin{proposition}\label{WS}
The localized weight spectral sequence for 
\begin{equation*}
\rmH^{3}_{\Iw,\frakm}(\Lambda(2))=\rmH^{3}(\overline{\rmY}\otimes\FF^{\ac}_{p}, \rmR\Psi(\Lambda)(2))_{\frakm^{[p]}_{\triplef}}
\end{equation*} 
converges at the $\rmE_{2}$-page. The graded pieces of the monodromy filtration $\rmM_{\bullet}\rmH^{3}_{\Iw,\frakm}(\Lambda(2))$ are given by
\begin{equation*}
\begin{aligned}
&\Gr^{\rmM}_{3}\rmH^{3}_{\Iw,\frakm}(\Lambda(2))=\rmH^{0}(\rmXbar_{[\pm\pm\pm]}\otimes\FF^{\ac}_{p}, \Lambda(2))_{\frakm^{[p]}_{\triplef}}\\
&\Gr^{\rmM}_{2}\rmH^{3}_{\Iw,\frakm}(\Lambda(2))=\underset{?\in\{\pm\}}\oplus\rmH^{1}(\rmXbar_{[?\pm\pm]}\otimes\FF^{\ac}_{p}, \Lambda(2))_{\frakm^{[p]}_{\triplef}}\oplus \underset{?\in\{\pm\}}\oplus\rmH^{1}(\rmXbar_{[\pm?\pm]}\otimes{\FF^{\ac}_{p}}, \Lambda(2))_{\frakm^{[p]}_{\triplef}}\oplus \\
&\phantom{aaaaaaaaaaaaa}\underset{?\in\{\pm\}}\oplus\rmH^{1}(\rmXbar_{[\pm\pm?]}\otimes{\FF^{\ac}_{p}}, \Lambda(2))_{\frakm^{[p]}_{\triplef}} \\
&\Gr^{\rmM}_{1}\rmH^{3}_{\Iw,\frakm}(\Lambda(2))=\underset{?_{1}, ?_{2}\in\{\pm\}}\oplus\rmH^{2}(\rmXbar_{[?_{1}?_{2}\pm]}\otimes{\FF^{\ac}_{p}}, \Lambda(2))_{\frakm^{[p]}_{\triplef}}\oplus \underset{?_{1}, ?_{2}\in\{\pm\}}\oplus\rmH^{2}(\rmXbar_{[?_{1}\pm?_{2}]}\otimes{\FF^{\ac}_{p}}, \Lambda(2))_{\frakm^{[p]}_{\triplef}}\oplus\\
&\phantom{aaaaaaaaaaaaa}\underset{?_{1}, ?_{2}\in\{\pm\}}\oplus\rmH^{2}(\rmXbar_{[\pm?_{1}?_{2}]}\otimes{\FF^{\ac}_{p}}, \Lambda(2))_{\frakm^{[p]}_{\triplef}}\oplus\rmH^{0}(\rmX_{[\pm\pm\pm]}\otimes\FF^{\ac}_{p}, \Lambda(1)))^{\oplus 3}_{\frakm^{[p]}_{\triplef}} \\
&\Gr^{\rmM}_{0}\rmH^{3}_{\Iw,\frakm}(\Lambda(2))=\underset{?_{1}, ?_{2}, ?_{3}\in\{\pm\}}\oplus\rmH^{3}(\rmXbar_{[?_{1}?_{2}?_{3}]}\otimes{\FF^{\ac}_{p}}, \Lambda(2))_{\frakm^{[p]}_{\triplef}}\oplus \underset{?\in\{\pm\}}\oplus\rmH^{1}(\rmXbar_{[?\pm\pm]}\otimes{\FF^{\ac}_{p}}, \Lambda(1))^{\oplus2}_{\frakm^{[p]}_{\triplef}}\oplus\\
&\phantom{aaaaaaaaaaaaa}\underset{?\in\{\pm\}}\oplus\rmH^{1}(\rmXbar_{[\pm?\pm]}\otimes{\FF^{\ac}_{p}}, \Lambda(1))^{\oplus2}_{\frakm^{[p]}_{\triplef}}\oplus\underset{?\in\{\pm\}}\oplus\rmH^{1}(\rmXbar_{[\pm\pm?]}\otimes{\FF^{\ac}_{p}}, \Lambda(1))^{\oplus2}_{\frakm^{[p]}_{\triplef}}\\
&\Gr^{\rmM}_{-1}\rmH^{3}_{\Iw,\frakm}(\Lambda(2))=\underset{?_{1}, ?_{2}\in\{\pm\}}\oplus\rmH^{2}(\rmXbar_{[?_{1}?_{2}\pm]}\otimes{\FF^{\ac}_{p}}, \Lambda(1))_{\frakm^{[p]}_{\triplef}}\oplus \underset{?_{1}, ?_{2}\in\{\pm\}}\oplus\rmH^{2}(\rmXbar_{[?_{1}\pm?_{2}]}\otimes{\FF^{\ac}_{p}}, \Lambda(1))_{\frakm^{[p]}_{\triplef}}\oplus\\
&\phantom{aaaaaaaaaaaaa}\underset{?_{1}, ?_{2}\in\{\pm\}}\oplus\rmH^{2}(\rmXbar_{[\pm?_{1}?_{2}]}\otimes{\FF^{\ac}_{p}}, \Lambda(1))_{\frakm^{[p]}_{\triplef}}\oplus\rmH^{0}(\rmX_{[\pm\pm\pm]}\otimes\FF^{\ac}_{p}, \Lambda)^{\oplus3}_{\frakm^{[p]}_{\triplef}} \\
&\Gr^{\rmM}_{-2}\rmH^{3}_{\Iw,\frakm}(\Lambda(2))=\underset{?\in\{\pm\}}\oplus\rmH^{1}(\rmXbar_{[?\pm\pm]}\otimes{\FF^{\ac}_{p}}, \Lambda)_{\frakm^{[p]}_{\triplef}}\oplus \underset{?\in\{\pm\}}\oplus\rmH^{1}(\rmXbar_{[\pm?\pm]}\otimes{\FF^{\ac}_{p}}, \Lambda)_{\frakm^{[p]}_{\triplef}}\oplus\\
&\phantom{aaaaaaaaaaaaaa}\underset{?\in\{\pm\}}\oplus\rmH^{1}(\rmXbar_{[\pm\pm?]}\otimes{\FF^{\ac}_{p}}, \Lambda)_{\frakm^{[p]}_{\triplef}} \\
&\Gr^{\rmM}_{-3}\rmH^{3}_{\Iw,\frakm}(\Lambda(2))=\rmH^{0}(\rmXbar_{[\pm\pm\pm]}\otimes{\FF^{\ac}_{p}}, \Lambda(-1))_{\frakm^{[p]}_{\triplef}}.\\
\end{aligned}
\end{equation*}
\end{proposition}
\begin{proof}
This follows from an explicit computation using the descriptions of $\rmYbar$ in Proposition \ref{Ybar} and \ref{E1-primitive}.
\end{proof}
\begin{remark}\label{extra-comp}
We note for later use that we can write the graded piece $\Gr^{\rmM}_{0}\rmH^{3}_{\Iw,\frakm}(\Lambda(2))$ more compactly as 
\begin{equation*}
\rmH^{3}(\rmYbar^{(0)}\otimes{\FF^{\ac}_{p}}, \Lambda(2))_{\frakm^{[p]}_{\triplef}}=\underset{?_{1}, ?_{2}, ?_{3}\in\{\pm\}}\oplus\rmH^{3}(\rmYbar_{[?_{1}?_{2}?_{3}]}\otimes{\FF^{\ac}_{p}}, \Lambda(2))_{\frakm^{[p]}_{\triplef}}.
\end{equation*}
In fact, we are identifying the cohomologies 
\begin{equation*}
\underset{?\in\{\pm\}}\oplus\rmH^{1}(\rmXbar_{[?\pm\pm]}\otimes{\FF^{\ac}_{p}}, \Lambda(1))^{\oplus2}_{\frakm^{[p]}_{\triplef}}\oplus\underset{?\in\{\pm\}}\oplus\rmH^{1}(\rmXbar_{[\pm?\pm]}\otimes{\FF^{\ac}_{p}}, \Lambda(1))^{\oplus2}_{\frakm^{[p]}_{\triplef}}\oplus\underset{?\in\{\pm\}}\oplus\rmH^{1}(\rmXbar_{[\pm\pm?]}\otimes{\FF^{\ac}_{p}}, \Lambda(1))^{\oplus2}_{\frakm^{[p]}_{\triplef}}
\end{equation*}
in $\Gr^{\rmM}_{0}\rmH^{3}_{\Iw,\frakm}(\Lambda(2))$ with the cohomologies
\begin{equation*}
\underset{?\in\{\pm\}}\oplus\rmH^{1}(\PP^{1}(\rmXbar_{[?\pm\pm]})\otimes{\FF^{\ac}_{p}}, \Lambda(1))^{\oplus2}_{\frakm^{[p]}_{\triplef}}\oplus\underset{?\in\{\pm\}}\oplus\rmH^{1}(\PP^{1}(\rmXbar_{[\pm?\pm]})\otimes{\FF^{\ac}_{p}}, \Lambda(1))^{\oplus2}_{\frakm^{[p]}_{\triplef}}\oplus\underset{?\in\{\pm\}}\oplus\rmH^{1}(\PP^{1}(\rmXbar_{[\pm\pm?]})\otimes{\FF^{\ac}_{p}}, \Lambda(1))^{\oplus2}_{\frakm^{[p]}_{\triplef}}
\end{equation*}
of the exceptional divisors in $\rmYbar^{(0)}$.
\end{remark}

Let $\rmH^{3}_{\Iw,\frakm}(\Lambda(2))^{\ur}=\rmH^{3}(\overline{\rmY}\otimes\FF^{\ac}_{p}, \rmR\Psi(\Lambda)(2))^{\ur}_{\frakm^{[p]}_{\triplef}}$ be the part of $\rmH^{3}_{\Iw,\frakm}(\Lambda(2))$ where the monodromy operator $N$ acts trivially.
\begin{corollary}\label{mono-fil-un}
The monodromy filtration $\rmM_{\bullet}\rmH^{3}_{\Iw,\frakm}(\Lambda(2))$ restricts to a filtration 
\begin{equation*}
\begin{aligned}
0&\subset^{\Gr^{\rmM}_{3}\rmH^{3}_{\Iw,\frakm}(\Lambda(2))^{\ur}}\rmM_{3}\rmH^{3}_{\Iw,\frakm}(\Lambda(2))^{\ur}\subset^{\Gr^{\rmM}_{2}\rmH^{3}_{\Iw,\frakm}(\Lambda(2))^{\ur}} \rmM_{2}\rmH^{3}_{\Iw,\frakm}(\Lambda(2))^{\ur}\subset^{\Gr^{\rmM}_{1}\rmH^{3}_{\Iw,\frakm}(\Lambda(2))^{\ur}} \rmM_{1}\rmH^{3}_{\Iw,\frakm}(\Lambda(2))^{\ur}\\&\subset^{\Gr^{\rmM}_{0}\rmH^{3}_{\Iw,\frakm}(\Lambda(2))^{\ur}} \rmM_{0}\rmH^{3}_{\Iw,\frakm}(\Lambda(2))^{\ur}=\rmH^{3}_{\Iw,\frakm}(\Lambda(2))^{\ur}\\
\end{aligned}
\end{equation*}
on $\rmH^{3}_{\Iw,\frakm}(\Lambda(2))^{\ur}$ whose graded pieces are given by
\begin{equation*}
\begin{aligned}
&\Gr^{\rmM}_{3}\rmH^{3}_{\Iw,\frakm}(\Lambda(2))^{\ur}=\rmH^{0}(\rmXbar_{[\pm\pm\pm]}\otimes\FF^{\ac}_{p}, \Lambda(2))_{\frakm^{[p]}_{\triplef}}\\
&\Gr^{\rmM}_{2}\rmH^{3}_{\Iw,\frakm}(\Lambda(2))^{\ur}=\underset{?\in\{\pm\}}\oplus\rmH^{1}(\rmXbar_{[?\pm\pm]}\otimes\FF^{\ac}_{p}, \Lambda(2))_{\frakm^{[p]}_{\triplef}}\oplus \underset{?\in\{\pm\}}\oplus\rmH^{1}(\rmXbar_{[\pm?\pm]}\otimes{\FF^{\ac}_{p}}, \Lambda(2))_{\frakm^{[p]}_{\triplef}}\oplus \\
&\phantom{aaaaaaaaaaaaaaaaa}\underset{?\in\{\pm\}}\oplus\rmH^{1}(\rmXbar_{[\pm\pm?]}\otimes{\FF^{\ac}_{p}}, \Lambda(2))_{\frakm^{[p]}_{\triplef}} \\
&\Gr^{\rmM}_{1}\rmH^{3}_{\Iw,\frakm}(\Lambda(2))^{\ur}=\underset{?_{1}, ?_{2}\in\{\pm\}}\oplus\rmH^{2}(\rmXbar_{[?_{1}?_{2}\pm]}\otimes{\FF^{\ac}_{p}}, \Lambda(2))_{\frakm^{[p]}_{\triplef}}\oplus \underset{?_{1}, ?_{2}\in\{\pm\}}\oplus\rmH^{2}(\rmXbar_{[?_{1}\pm?_{2}]}\otimes{\FF^{\ac}_{p}}, \Lambda(2))_{\frakm^{[p]}_{\triplef}}\oplus\\
&\phantom{aaaaaaaaaaaaaaaaa}\underset{?_{1}, ?_{2}\in\{\pm\}}\oplus\rmH^{2}(\rmXbar_{[\pm?_{1}?_{2}]}\otimes{\FF^{\ac}_{p}}, \Lambda(2))_{\frakm^{[p]}_{\triplef}}\oplus\rmH^{0}(\rmX_{[\pm\pm\pm]}\otimes\FF^{\ac}_{p}, \Lambda(1)))^{\oplus 2}_{\frakm^{[p]}_{\triplef}} \\
&\Gr^{\rmM}_{0}\rmH^{3}_{\Iw,\frakm}(\Lambda(2))^{\ur}=\underset{?_{1}, ?_{2}, ?_{3}\in\{\pm\}}\oplus\rmH^{3}(\rmXbar_{[?_{1}?_{2}?_{3}]}\otimes{\FF^{\ac}_{p}}, \Lambda(2))_{\frakm^{[p]}_{\triplef}}\oplus \underset{?\in\{\pm\}}\oplus\rmH^{1}(\rmXbar_{[?\pm\pm]}\otimes{\FF^{\ac}_{p}}, \Lambda(1))_{\frakm^{[p]}_{\triplef}}\oplus\\
&\phantom{aaaaaaaaaaaaaaaaaaa}\underset{?\in\{\pm\}}\oplus\rmH^{1}(\rmXbar_{[\pm?\pm]}\otimes{\FF^{\ac}_{p}}, \Lambda(1))_{\frakm^{[p]}_{\triplef}}\oplus\underset{?\in\{\pm\}}\oplus\rmH^{1}(\rmXbar_{[\pm\pm?]}\otimes{\FF^{\ac}_{p}}, \Lambda(1))_{\frakm^{[p]}_{\triplef}}.\\
\end{aligned}
\end{equation*}
\end{corollary}
\begin{proof}
This follows from the explicit descriptions of the monodromy operator on $\Gr^{\rmM}_{\bullet}\rmH^{3}_{\Iw,\frakm}(\Lambda(2))$ and Proposition \ref{WS}.
\end{proof}

\subsection{Proof of Theorem \ref{2-law}} Now we return to the proof of Theorem \ref{2-law}. We will consider the following composite map
\begin{equation*}
\begin{aligned}
\mathrm{CH}^{2}(\overline{\rmX}^{3}_{d})&\xrightarrow{{\AJ}_{\triplef, n}} \rmH^{1}(\FF_{p^{2}}, \rmM_{n}(\triplef, d)(-1))\xrightarrow{\mathbf{a}^{\ast}} \rmH^{1}(\FF_{p^{2}}, \rmM^{[p]}_{n}(\triplef, d)^{\bullet}(-1))\\
&\longrightarrow \rmH^{1}(\FF_{p^{2}}, \rmH^{3}(\overline{\rmX}^{3}_{d}(p)\otimes\FF^{\ac}_{p}, \rmR\Psi(\calO_{\undlamb})(2))_{/\frakp^{[p]}_{\triplef,n}})\cong \rmH^{1}(\FF_{p^{2}}, \rmH^{3}(\overline{\rmY}\otimes\FF^{\ac}_{p}, \rmR\Psi(\calO_{\undlamb})(2))_{/\frakp^{[p]}_{\triplef,n}}).\\
\end{aligned}
\end{equation*}

By Remark \ref{special-new}, $\rmM^{[p]}_{n}(\triplef, d)^{\bullet}(-1)$ can be identified with $\rmH^{3}(\overline{\rmX}^{3}_{d}(p)\otimes{\FF}^{\ac}_{p}, \calO_{\undlamb}(2))^{\bullet}_{/\frakp^{[p]}_{\triplef,n}}$.
It is clear that the specialization map $\rmH^{3}(\overline{\rmX}^{3}_{d}(p)\otimes{\FF}^{\ac}_{p}, \calO_{\undlamb}(2))\xrightarrow{sp} \rmH^{3}(\overline{\rmX}^{3}_{d}(p)\otimes\FF^{\ac}_{p}, \rmR\Psi(\calO_{\undlamb})(2))\cong \rmH^{1}(\FF_{p^{2}}, \rmH^{3}(\overline{\rmY}\otimes\FF^{\ac}_{p}, \rmR\Psi(\calO_{\undlamb})(2)))$ factors through
\begin{equation*}
\rmH^{3}(\overline{\rmY}\otimes\FF^{\ac}_{p}, \rmR\Psi(\calO_{\undlamb})(2))^{\ur}=\rmM_{0}\rmH^{3}(\overline{\rmY}\otimes\FF^{\ac}_{p}, \rmR\Psi(\calO_{\undlamb})(2))^{\ur}
\end{equation*}
and hence we obtain a map 
\begin{equation*}
\begin{aligned}
\mathrm{CH}^{2}(\overline{\rmX}^{3}_{d})&\xrightarrow{{\AJ}_{\triplef, n}} \rmH^{1}(\FF_{p^{2}}, \rmM_{n}(\triplef, d)(-1))\xrightarrow{\mathbf{a}^{\ast}} \rmH^{1}(\FF_{p^{2}}, \rmM^{[p]}_{n}(\triplef, d)^{\bullet}(-1))\\
&\longrightarrow \rmH^{1}(\FF_{p^{2}}, \rmM_{0}\rmH^{3}(\overline{\rmY}\otimes\FF^{\ac}_{p}, \rmR\Psi(\calO_{\undlamb})(2))^{\ur}_{/\frakp^{[p]}_{\triplef,n}})\\
&\longrightarrow \rmH^{1}(\FF_{p^{2}}, \Gr^{\rmM}_{0}\rmH^{3}(\overline{\rmY}\otimes\FF^{\ac}_{p}, \rmR\Psi(\calO_{\undlamb})(2))^{\ur}_{/\frakp^{[p]}_{\triplef,n}})\\
\end{aligned}
\end{equation*}
whose target can be identified with 
\begin{equation*}
\begin{aligned}
&\phantom{aa}\rmH^{1}(\FF_{p^{2}}, \underset{?_{1}, ?_{2}, ?_{3}\in\{\pm\}}\oplus\rmH^{3}(\rmXbar_{[?_{1}?_{2}?_{3}]}\otimes{\FF^{\ac}_{p}}, \calO_{\undlamb}(2))_{/\triplepnp})\\
&\oplus\rmH^{1}(\FF_{p^{2}}, \underset{?\in\{\pm\}}\oplus\rmH^{1}(\rmXbar_{[?\pm\pm]}\otimes{\FF^{\ac}_{p}}, \calO_{\undlamb}(1))^{\oplus2}_{/\triplepnp})\\
&\oplus \rmH^{1}(\FF_{p^{2}}, \underset{?\in\{\pm\}}\oplus\rmH^{1}(\rmXbar_{[\pm?\pm]}\otimes{\FF^{\ac}_{p}}, \calO_{\undlamb}(1))^{\oplus2}_{/\triplepnp})\\
&\oplus \rmH^{1}(\FF_{p^{2}}, \underset{?\in\{\pm\}}\oplus\rmH^{1}(\rmXbar_{[\pm\pm?]}\otimes{\FF^{\ac}_{p}}, \calO_{\undlamb}(1))^{\oplus2}_{/\triplepnp}).\\
\end{aligned}
\end{equation*}
Then we can further project it to its direct summand given by
\begin{equation*}
\begin{aligned}
&\phantom{aa}\rmH^{1}(\FF_{p^{2}}, \rmH^{1}(\rmXbar_{[+\pm\pm]}\otimes{\FF^{\ac}_{p}}, \calO_{\undlamb}(1))_{\triplepnp})\\
&\oplus \rmH^{1}(\FF_{p^{2}}, \rmH^{1}(\rmXbar_{[\pm+\pm]}\otimes{\FF^{\ac}_{p}}, \calO_{\undlamb}(1))_{\triplepnp})\\
&\oplus \rmH^{1}(\FF_{p^{2}}, \rmH^{1}(\rmXbar_{[\pm\pm+]}\otimes{\FF^{\ac}_{p}}, \calO_{\undlamb}(1))_{\triplepnp}).\\
\end{aligned}
\end{equation*} 

All in all, we obtain the following series of maps
\begin{equation*}
\begin{aligned}
\mathrm{CH}^{2}(\overline{\rmX}^{3}_{d})&\xrightarrow{{\AJ}_{\triplef, n}} \rmH^{1}(\FF_{p^{2}}, \rmM_{n}(\triplef, d)(-1))\xrightarrow{\mathbf{a}^{\ast}} \rmH^{1}(\FF_{p^{2}}, \rmM^{[p]}_{n}(\triplef, d)^{\bullet}(-1))\\
&\longrightarrow \rmH^{1}(\FF_{p^{2}}, \Gr^{\rmM}_{0}\rmH^{3}(\overline{\rmY}\otimes\FF^{\ac}_{p}, \rmR\Psi(\calO_{\undlamb})(2))^{\ur}_{/\triplepnp})\\
&\longrightarrow \rmH^{1}(\FF_{p^{2}}, \underset{?_{1}, ?_{2}, ?_{3}\in\{\pm\}}\oplus\rmH^{3}(\rmXbar_{[?_{1}?_{2}?_{3}]}\otimes{\FF^{\ac}_{p}}, \calO_{\lambda}(2))_{/\triplepnp})\\
&\phantom{aa}\oplus\rmH^{1}(\FF_{p^{2}}, \underset{?\in\{\pm\}}\oplus\rmH^{1}(\rmXbar_{[?\pm\pm]}\otimes{\FF^{\ac}_{p}}, \calO_{\undlamb}(1))^{\oplus2}_{/\triplepnp})\\
&\phantom{aa}\oplus \rmH^{1}(\FF_{p^{2}}, \underset{?\in\{\pm\}}\oplus\rmH^{1}(\rmXbar_{[\pm?\pm]}\otimes{\FF^{\ac}_{p}}, \calO_{\undlamb}(1))^{\oplus2}_{/\triplepnp})\\
&\phantom{aa}\oplus \rmH^{1}(\FF_{p^{2}}, \underset{?\in\{\pm\}}\oplus\rmH^{1}(\rmXbar_{[\pm\pm?]}\otimes{\FF^{\ac}_{p}}, \calO_{\undlamb}(1))^{\oplus2}_{/\triplepnp})\\
&\longrightarrow\rmH^{1}(\FF_{p^{2}}, \rmH^{1}(\rmXbar_{[+\pm\pm]}\otimes{\FF^{\ac}_{p}}, \calO_{\undlamb}(1))_{/\triplepnp})\\
&\phantom{aa}\oplus \rmH^{1}(\FF_{p^{2}}, \rmH^{1}(\rmXbar_{[\pm+\pm]}\otimes{\FF^{\ac}_{p}}, \calO_{\undlamb}(1))_{/\triplepnp})\\
&\phantom{aa}\oplus \rmH^{1}(\FF_{p^{2}}, \rmH^{1}(\rmXbar_{[\pm\pm+]}\otimes{\FF^{\ac}_{p}}, \calO_{\undlamb}(1))_{/\triplepnp}).\\
\end{aligned}
\end{equation*}
Let $\Xi_{n}$ be the composite of the above maps and let $\Xi_{?, n}$ be the composite of $\Xi_{n}$ with the projection to the component given by $\rmH^{1}(\FF_{p^{2}}, \rmH^{1}(\rmXbar_{?}\otimes{\FF^{\ac}_{p}}, \calO_{\undlamb})_{/\triplepnp})$ for $?\in\{[+\pm\pm], [\pm+\pm], [\pm\pm+]\}$.  We will analyze the image of the diagonal element $\overline{\Delta}_{d}$ under $\Xi_{?, n}$.  Recall the diagonal cycle 
\begin{equation*}
\Delta_{d}(\overline{B})=\vartheta_{\ast}[\rmZ_{d}(\overline{B})]\in \mathrm{CH}^{0}(\rmZ_{d}(\overline{B})^{3})
\end{equation*}
on the triple product of Shimura set $\rmZ_{d}(\overline{B})^{3}$ defined in \ref{def-diagonal}. Note the element $\Delta_{d}(\overline{B})$ can be also regarded as an element in the Chow group $\mathrm{CH}^{1}(\rmXbar_{?})$ for $?\in\{[+\pm\pm], [\pm+\pm], [\pm\pm+]\}$. Thus it gives rise to a cohomology class
\begin{equation*}
\vartheta_{?, n}\in  \rmH^{1}(\FF_{p^{2}}, \rmH^{1}(\rmXbar_{?}\otimes{\FF^{\ac}_{p}}, \calO_{\undlamb})_{/\triplepnp})
\end{equation*}
via the mod $\undlamb^{n}$ Abel--Jacobi map 
\begin{equation*}
\mathrm{CH}^{1}(\rmXbar_{?})\rightarrow\rmH^{1}(\FF_{p^{2}}, \rmH^{1}(\rmXbar_{?}\otimes{\FF^{\ac}_{p}}, \calO_{\undlamb})_{/\triplepnp}) 
\end{equation*}
constructed in the same way as \ref{AJ-n}.
\begin{proposition}\label{ss-diagonal}
The image of the diagonal element  $\overline{\Delta}_{d}$ under $\Xi_{?, n}$ is given by the class
\begin{equation*}
\vartheta_{?, n}\in  \rmH^{1}(\FF_{p^{2}}, \rmH^{1}(\rmXbar_{?}\otimes{\FF^{\ac}_{p}}, \calO_{\undlamb})_{/\triplepnp})
\end{equation*}
for $?\in \{[+\pm\pm], [\pm+\pm], [\pm\pm+]\}$.
\end{proposition}
\begin{proof}
Recall the class $\overline{\Theta}_{n}(\triplef, d)\in \rmH^{1}(\FF_{p^{2}}, \rmM_{n}(\triplef, d)(-1))$ as in \ref{modtheta}. Then by unwinding the definitions, see in particular \ref{1/2},  we find that
\begin{equation*}
\begin{aligned}
2^{3}\mathbf{a}^{\ast}(\overline{\Theta}_{n}(\triplef, d))&=\pi^{\ast}_{[111]}(\overline{\Theta}_{n}(\triplef, d))-\epsilon_{1}\pi^{\ast}_{[211]}(\overline{\Theta}_{n}(\triplef, d))-\epsilon_{2}\pi^{\ast}_{[121]}(\overline{\Theta}_{n}(\triplef, d))-\epsilon_{3}\pi^{\ast}_{[112]}(\overline{\Theta}_{n}(\triplef, d))\\
&\phantom{aa}+\epsilon_{2}\epsilon_{3}\pi^{\ast}_{[122]}(\overline{\Theta}_{n}(\triplef, d))+\epsilon_{1}\epsilon_{3}\pi^{\ast}_{[212]}(\overline{\Theta}_{n}(\triplef, d))+\epsilon_{1}\epsilon_{2}\pi^{\ast}_{[221]}(\overline{\Theta}_{n}(\triplef, d))\\
&\phantom{aa}-\epsilon_{1}\epsilon_{2}\epsilon_{3}\pi^{\ast}_{[222]}(\overline{\Theta}_{n}(\triplef, d))
\end{aligned}
\end{equation*}
where $\pi^{\ast}_{[?_{1}?_{2}?_{3}]}=(\pi^{\ast}_{?_{1},p}, \pi^{\ast}_{?_{2},p}, \pi^{\ast}_{?_{3},p})$ for $?_{1}, ?_{2}, ?_{3}\in\{1, 2\}$. 

This leads to understand the cycles $\pi^{\ast}_{[?_{1}?_{2}?_{3}]}(\overline{\Delta}_{d})$ in $\rmYbar^{(0)}$. In terms of their moduli interpretations in the notations of \S 3.1, these cycles are given by
\begin{equation*}
\begin{aligned}
&\pi^{\ast}_{[111]}(\overline{\Delta}_{d})=\{(\underline{A}_{1}, \underline{A}_{2}, \underline{A}_{3}, C_{1}\subset A_{1}[p], C_{2}\subset A_{2}[p], C_{3}\subset A_{3}[p]): \underline{A}_{1}\cong\underline{A}_{2}\cong\underline{A}_{3}\}\\
&\pi^{\ast}_{[112]}(\overline{\Delta}_{d})=\{(\underline{A}_{1}, \underline{A}_{2}, \underline{A}_{3}, C_{1}\subset A_{1}[p], C_{2}\subset A_{2}[p], C_{3}\subset A_{3}[p]): \underline{A}_{1}\cong \underline{A}_{2}\cong\underline{A}, \underline{A}_{3}/C_{3}\cong\underline{A}\}\\
&\pi^{\ast}_{[121]}(\overline{\Delta}_{d})=\{(\underline{A}_{1}, \underline{A}_{2}, \underline{A}_{3}, C_{1}\subset A_{1}[p], C_{2}\subset A_{2}[p], C_{3}\subset A_{3}[p]): \underline{A}_{1}\cong \underline{A}_{3}\cong\underline{A}, \underline{A}_{2}/C_{2}\cong\underline{A}\}\\
&\pi^{\ast}_{[112]}(\overline{\Delta}_{d})=\{(\underline{A}_{1}, \underline{A}_{2}, \underline{A}_{3}, C_{1}\subset A_{1}[p], C_{2}\subset A_{2}[p], C_{3}\subset A_{3}[p]): \underline{A}_{2}\cong \underline{A}_{3}\cong\underline{A}, \underline{A}_{1}/C_{1}\cong\underline{A}\}\\
&\pi^{\ast}_{[122]}(\overline{\Delta}_{d})=\{(\underline{A}_{1}, \underline{A}_{2}, \underline{A}_{3}, C_{1}\subset A_{1}[p], C_{2}\subset A_{2}[p], C_{3}\subset A_{3}[p]): \underline{A}_{2}/C_{2}\cong \underline{A}_{1}, \underline{A}_{3}/C_{3}\cong\underline{A}_{1}\}\\
&\pi^{\ast}_{[212]}(\overline{\Delta}_{d})=\{(\underline{A}_{1}, \underline{A}_{2}, \underline{A}_{3}, C_{1}\subset A_{1}[p], C_{2}\subset A_{2}[p], C_{3}\subset A_{3}[p]): \underline{A}_{1}/C_{1}\cong \underline{A}_{2}, \underline{A}_{3}/C_{3}\cong\underline{A}_{2}\}\\
&\pi^{\ast}_{[221]}(\overline{\Delta}_{d})=\{(\underline{A}_{1}, \underline{A}_{2}, \underline{A}_{3}, C_{1}\subset A_{1}[p], C_{2}\subset A_{2}[p], C_{3}\subset A_{3}[p]): \underline{A}_{1}/C_{1}\cong \underline{A}_{3}, \underline{A}_{2}/C_{2}\cong\underline{A}_{3}\}\\
&\pi^{\ast}_{[222]}(\overline{\Delta}_{d})=\{(\underline{A}_{1}, \underline{A}_{2}, \underline{A}_{3}, C_{1}\subset A_{1}[p], C_{2}\subset A_{2}[p], C_{3}\subset A_{3}[p]): \underline{A}_{1}/C_{1}\cong \underline{A}_{2}/C_{2}\cong\underline{A}_{3}/C_{3}\cong \underline{A}\}\\
\end{aligned}
\end{equation*}
where $\underline{A}_{i}$ is given by a tuple $(A_{i},  \iota_{i},  C_{N^{+}}, \alpha_{d})$ as in Definition \ref{Xm} and $C_{i}$ is a cyclic subgroup of order $p^{2}$ stable under the action of $\calO_{B}$ via $\iota_{i}$.

By definition, we need to understand the intersection of the cycle $\pi^{\ast}_{[?_{1}?_{2}?_{3}]}(\overline{\Delta}_{d})$ for $?_{1}, ?_{2}, ?_{3}\in\{1, 2\}$ with $\PP^{1}(\rmXbar_{?})$ for $?\in\{[+\pm\pm], [\pm+\pm], [\pm\pm+]\}$ in light of Remark \ref{extra-comp}. It is clear that for any point 
\begin{equation*}
(\underline{A}_{1}, \underline{A}_{2}, \underline{A}_{3}, C_{1}\subset A_{1}[p], C_{2}\subset A_{2}[p], C_{3}\subset A_{3}[p])\in \pi^{\ast}_{[?_{1}?_{2}?_{3}]}(\overline{\Delta}_{d})\cap \PP^{1}(\rmXbar_{[?]}), 
\end{equation*}
$A_{1}, A_{2}, A_{3} $ has to be supersingular and therefore $C_{1}, C_{2}, C_{3}$ are uniquely determined. For example, if this point lies in $\pi^{\ast}_{[112]}(\overline{\Delta}_{d})\cap \PP^{1}(\rmXbar_{?})$, then $A$ (and therefore ${A}_{1}$ and ${A}_{2}$) and ${A}_{3}$ are supersingular and $C_{1}\cong C_{2}$ is the kernel of Frobenius of $A$. Moreover ${A}_{3}\cong A^{(p)}$ and $C_{3}$ is the kernel of the map $A^{(p)}\rightarrow A$.

Then it is not difficult to see that
\begin{equation*}
\begin{aligned}
&\pi^{\ast}_{[111]}(\overline{\Delta}_{d})\cap \PP^{1}(\rmXbar_{[?]})=\Delta_{d}(\overline{B})\\
&\pi^{\ast}_{[112]}(\overline{\Delta}_{d})\cap \PP^{1}(\rmXbar_{[?]})=(1,1, w_{p})\Delta_{d}(\overline{B})\\
&\pi^{\ast}_{[121]}(\overline{\Delta}_{d})\cap\PP^{1}(\rmXbar_{[?]})=(1, w_{p}, 1)\Delta_{d}(\overline{B})\\
&\pi^{\ast}_{[112]}(\overline{\Delta}_{d})\cap \PP^{1}(\rmXbar_{[?]})=(w_{p}, 1, 1)\Delta_{d}(\overline{B})\\
&\pi^{\ast}_{[122]}(\overline{\Delta}_{d})\cap \PP^{1}(\rmXbar_{[?]})=(1, w_{p}, w_{p})\Delta_{d}(\overline{B})\\
&\pi^{\ast}_{[212]}(\overline{\Delta}_{d})\cap \PP^{1}(\rmXbar_{[?]})=(w_{p}, 1, w_{p})\Delta_{d}(\overline{B})\\
&\pi^{\ast}_{[221]}(\overline{\Delta}_{d})\cap \PP^{1}(\rmXbar_{[?]})=(w_{p},  w_{p}, 1)\Delta_{d}(\overline{B})\\
&\pi^{\ast}_{[222]}(\overline{\Delta}_{d})\cap \PP^{1}(\rmXbar_{[?]})= (w_{p},  w_{p}, w_{p})\Delta_{d}(\overline{B})\\
\end{aligned}
\end{equation*}
where $w_{p}$ is the Atkin-Lehner operator. By an easy adaption of the proof of \cite[Proposition 3.8]{Ri-100} in our setting, $w_{p}$ acts by $-\rmU_{p}$ on $\Gamma(\rmZ_{d}(\overline{B}), \calO_{\lambda_{i}})/\frakp^{[p]}_{i,n}$ which in turn agrees with $-\epsilon_{i}$ for $i\in\{1, 2, 3\}$. Hence $\Xi_{?, n}(\overline{\Delta}_{d})$ is given by the cycle class of 
\begin{equation*}
\begin{aligned}
\Delta_{d}(\overline{B})=&2^{-3}(\Delta_{d}(\overline{B})+\epsilon^{2}_{1}\Delta_{d}(\overline{B})+\epsilon^{2}_{2}\Delta_{d}(\overline{B})+\epsilon^{2}_{3}\Delta_{d}(\overline{B})\\
&+\epsilon^{2}_{2}\epsilon^{2}_{3}\Delta_{d}(\overline{B})+\epsilon^{2}_{1}\epsilon^{2}_{3}\Delta_{d}(\overline{B})+\epsilon^{2}_{1}\epsilon^{2}_{2}\Delta_{d}(\overline{B})\\
&+\epsilon^{2}_{1}\epsilon^{2}_{2}\epsilon^{2}_{3}\Delta_{d}(\overline{B})).
\end{aligned}
\end{equation*}
which is nothing but $\vartheta_{?, n}$.
\end{proof}

\begin{myproof}{Theorem \ref{2-law}}{}
Now we can finish the proof of Theorem \ref{2-law}. We note that we have the following commutative diagram by construction
\begin{equation}\label{key-diag}
\begin{tikzcd}
\rmH^{1}(\FF_{p^{2}}, \rmM_{n}(\triplef, d)(-1)) \arrow[r] \arrow[d] & \rmH^{1}(\FF_{p^{2}}, \rmM^{[p]}_{n}(\triplef, d)^{\bullet}(-1))  \arrow[r] \arrow[d] & \rmH^{1}(\FF_{p^{2}}, \Gr^{\rmM}_{0}\rmM^{[p]}_{n}(\triplef,d)^{\ur}(-1)) \arrow[d] &             \\
\rmT^{+}_{1, n}\otimes\rmT^{-}_{2, n}\otimes \rmT^{-}_{3, n}(-1) \arrow[r, "\Psi_{n}\otimes1\otimes1"]           & \rmT^{\bullet+}_{1, n}\otimes\rmT^{\bullet-}_{2, n}\otimes \rmT^{\bullet-}_{3, n}(-1)  \arrow[r, "1\otimes\Phi^{\ast}_{n}\otimes1" ]                   & \rmH^{1}(\FF_{p^{2}}, \rmH^{1}(\rmXbar_{[\pm\pm+]}\otimes{\FF^{\ac}_{p}}, \calO_{\lambda}(1))_{/\frakp_{\triplef^{[p]}, n}})\arrow[d, "1\otimes1\otimes\Phi^{\ast}_{n}"] \\
& & \threetensor\Gamma(\rmZ_{d}(\overline{B}), \calO_{\lambda_{i}})_{/\frakp^{[p]}_{i, n}}.\\
\end{tikzcd}
\end{equation}
For the second vertical map $\rmH^{1}(\FF_{p^{2}}, \rmM^{[p]}_{n}(\triplef, d)^{\bullet}(-1))\rightarrow \rmT^{\bullet+}_{1, n}\otimes\rmT^{\bullet-}_{2, n}\otimes \rmT^{\bullet-}_{3, n}(-1)$, we identify the target with
\begin{equation*}
\rmH^{0}(\rmXbar^{\ss}\otimes\FF^{\ac}_{p},\calO_{\lambda})^{\rmG_{\FF_{p^{2}}}}_{/\frakp^{[p]}_{1, n}}\otimes\rmH^{1}(\FF_{p^{2}}, \rmH^{1}(\rmXbar_{d}\otimes{\FF^{\ac}_{p}}, \calO_{\lambda}(1))_{/\frakp^{[p]}_{2, n}})\otimes\rmH^{1}(\FF_{p^{2}}, \rmH^{1}(\rmXbar_{d}\otimes{\FF^{\ac}_{p}}, \calO_{\lambda}(1))_{/\frakp^{[p]}_{3, n}}). \\
\end{equation*}
For the third vertical map  $\rmH^{1}(\FF_{p^{2}}, \Gr^{\rmM}_{0}\rmM^{[p]}_{n}(\triplef,d)^{\ur}(-1))\rightarrow  \rmH^{1}(\FF_{p^{2}}, \rmH^{1}(\rmXbar_{[\pm\pm+]}\otimes{\FF^{\ac}_{p}}, \calO_{\lambda}(1))_{/\frakp_{\triplef^{[p]}, n}})$, we identify the target with
\begin{equation*}
 \rmH^{0}(\rmXbar^{\ss}\otimes\FF^{\ac}_{p},\calO_{\lambda})^{\rmG_{\FF_{p^{2}}}}_{/\frakp^{[p]}_{1, n}}\otimes\rmH^{0}(\rmXbar^{\ss}\otimes\FF^{\ac}_{p},\calO_{\lambda})^{\rmG_{\FF_{p^{2}}}}_{/\frakp^{[p]}_{2, n}}\otimes\rmH^{1}(\FF_{p^{2}}, \rmH^{1}(\rmXbar_{d}\otimes{\FF^{\ac}_{p}}, \calO_{\lambda}(1))_{/\frakp^{[p]}_{3, n}})
\end{equation*}
where, in the triple tensor product on the right-hand-side, the first factor $\rmH^{0}(\rmXbar^{\ss}\otimes\FF^{\ac}_{p},\calO_{\lambda})^{\rmG_{\FF_{p^{2}}}}_{/\frakp^{[p]}_{1, n}}$ should be identified with $\rmT^{\bullet+}_{1, n}$ and the second factor $\rmH^{0}(\rmXbar^{\ss}\otimes\FF^{\ac}_{p},\calO_{\lambda})^{\rmG_{\FF_{p^{2}}}}_{/\frakp^{[p]}_{2, n}}$ should be identified with the space of vanishing cycles $\rmH^{1}(\overline{\rmX}_{2}(p)\otimes{\FF^{\ac}_{p}}, \rmR\Phi(\calO_{\lambda_{2}})(1))_{/\frakp^{[p]}_{2, n}}$ on $\overline{\rmX}_{2}(p)$.
The commutativity of the second square is the same as the commutativity of the following diagram
\begin{equation}
\begin{tikzcd}
\rmH^{1}(\FF_{p^{2}}, \rmM^{[p]}_{n}(\triplef, d)^{\bullet}(-1))\arrow[r]\arrow[d]&\rmH^{1}(\FF_{p^{2}}, \Gr^{\rmM}_{0}\rmM^{[p]}_{n}(\triplef,d)^{\ur}(-1)) \arrow[d] \\
\rmT^{\bullet+}_{1, n}\otimes\rmT^{\bullet-}_{2, n}\otimes \rmT^{\bullet-}_{3, n}(-1)   & \rmH^{1}(\FF_{p^{2}}, \rmH^{1}(\rmXbar_{[\pm\pm+]}\otimes{\FF^{\ac}_{p}}, \calO_{\lambda}(1))_{/\frakp^{[p]}_{\triplef, n}}) \arrow[l, "1\otimes\Phi_{n}\otimes1" ] \\
\end{tikzcd}
\end{equation}
which is given by the construction of the exact sequence \ref{fil-T-new} in Proposition \ref{new-old-part}.

If we follow the element 
\begin{equation*}
\overline{\Theta}_{n}(\triplef, d)\in\rmH^{1}(\FF_{p^{2}}, \rmM_{n}(\triplef, d)(-1)) 
\end{equation*}
to $\threetensor\Gamma(\rmZ_{d}(\overline{B}), \calO_{\lambda_{i}})_{/\frakp^{[p]}_{i, n}}$ through the second row of diagram \ref{key-diag}, we obtain  the element $\overline{\Theta}^{(1)}_{n}(\triplef, d)$ defined in \ref{j-theta} since the composite is exactly $\Psi_{n}\otimes\Phi^{\ast}_{n}\otimes\Phi^{\ast}_{n}$. On the other hand, if we follow the element $\overline{\Theta}_{n}(\triplef, d)\in\rmH^{1}(\FF_{p^{2}}, \rmM_{n}(\triplef, d)(-1))$ to $\threetensor\Gamma(\rmZ_{d}(\overline{B}), \calO_{\lambda_{i}})_{/\frakp^{[p]}_{i, n}}$ through the first row of this diagram, then we obtain the element  
\begin{equation*}
\mathbf{1}_{\overline{B}}\in \Gamma(\rmZ_{d}(\overline{B})^{3}, \calO_{\lambda})_{/\frakp^{[p]}_{\triplef,n}}\cong\threetensor\Gamma(\rmZ_{d}(\overline{B}), \calO_{\lambda_{i}})_{/\frakp^{[p]}_{i, n}} 
\end{equation*}
by Proposition \ref{ss-diagonal} where we recall $\mathbf{1}_{\overline{B}}$ is the characteristic function of the diagonal cycle $\Delta_{d}(\overline{B})=\vartheta_{\ast}\rmZ_{d}(\overline{B})$ in $\rmZ_{d}(\overline{B})^{3}$. Therefore we have
\begin{equation*}
\begin{aligned}
(\bar{\Theta}^{(1)}_{n}(\triplef, d), \phi_{1}\otimes\phi_{2}\otimes\phi_{3})&=(\mathbf{1}_{\overline{B}}, \phi_{1}\otimes\phi_{2}\otimes\phi_{3})\\
&=\sum_{z\in \Delta_{d}(\overline{B})} \phi_{1}(z)\otimes \phi_{2}(z)\otimes \phi_{3}(z).\\
\end{aligned}
\end{equation*} 
holds for any $\phi_{1}\otimes \phi_{2}\otimes \phi_{2}\in \otimes^{3}_{i=1}\Gamma(Z_{d}(\overline{B}),E_{\lambda_{i}}/ \calO_{\lambda_{i}})[\frakp^{[p]}_{i, n}]$ as desired.
\end{myproof}

\section{Ramified level raising on triple product of Shimura curves}
\subsection{Ramified arithmetic level raising for Shimura curves}
We consider the setting in $\S2.2$, let $(p,q)$ be a pair of $n$-admissible primes for the modular form $f$ and an integer $n\geq 1$. We attached to $f$ the Galois representation $\rho_{f,\lambda}$ and the Galois module $\rho_{\calO_{\lambda, n}}$ over $\calO_{\lambda, n}$. We consider the indefinite quaternion algebra $B^{\natural}$ of discriminant $N^{-}pq$. Then we can associate to it a Shimura curve over $\QQ$ denoted by $\rmX^{\natural}_{d}$ and its integral model $\interX^{\natural}_{d}$ over $\ZZ[1/Nd]$ as in \cite[3A]{Wang}. The curve $\interX^{\natural}_{d}$ admits the Cerednick--Drinfeld uniformization and its special fiber over $\FF_{q^{2}}$ can be explicitly described as in \cite[Proposition 2.2]{Wang}. In particular, the special fiber $\overline{\rmX}^{\natural}_{d}$ is a union $\PP^{1}(\rmZ^{+}_{d}(\overline{B}))\cup \PP^{1}(\rmZ^{-}_{d}(\overline{B}))$ with $\rmZ^{+}_{d}(\overline{B})\cong \rmZ^{-}_{d}(\overline{B})\cong\rmZ_{d}(\overline{B})$ and the intersection $\PP^{1}(\rmZ^{+}_{d}(\overline{B}))\cap \PP^{1}(\rmZ^{-}_{d}(\overline{B}))$ is the Shimura set $\rmZ_{d}(q)(\overline{B})$ obtained by adding an Iwahori level structure at $q$ to the Shimura set $\rmZ_{d}(\overline{B})$. Let $\overline{\pi}_{1,q}:\rmZ_{d}(q)(\overline{B})\rightarrow \rmZ_{d}(\overline{B})$ and $\overline{\pi}_{2,q}:\rmZ_{d}(q)(\overline{B})\rightarrow \rmZ_{d}(\overline{B})$ be the two natural degneracy maps at $q$.
We prove some parallel results in this setting to Proposition \ref{canonical-decomposition}, Proposition \ref{new-old-part} and Proposition \ref{new-old-quotient}.
\begin{lemma}\label{unramified-drinfeld}
Under Assumption \ref{assump}, we have an isomorphism 
\begin{equation*}
\rmH^{1}(\rmX^{\natural}_{d}\otimes{\QQ^{\ac}}, \calO_{\lambda}(1))_{/\frakp^{[pq]}_{n}}\cong \rho^{\oplus 2}_{\calO_{\lambda, n}}.
\end{equation*}
\end{lemma}
\begin{proof}
This follows from the same argument as in \cite[Theorem 5.17]{BD} combined with the argument in Lemma \ref{multi-one}. 
\end{proof}

\begin{definition}
We define the {new part} $\Gamma(\rmZ_{d}(q)(\overline{B}), \calO_{\lambda})^{\bullet}_{/\frakp^{[pq]}_{n}}$ of the space $\Gamma(\rmZ_{d}(q)(\overline{B}), \calO_{\lambda})_{/\frakp^{[pq]}_{n}}$ at $q$ by
\begin{equation*}
\Gamma(\rmZ_{d}(q)(\overline{B}), \calO_{\lambda})^{\bullet}_{/\frakp^{[pq]}_{n}}=\ker[\Gamma(\rmZ_{d}(q)(\overline{B}), \calO_{\lambda})_{/\frakp^{[pq]}_{n}}\xrightarrow{(-\overline{\pi}_{1,q,\ast},\overline{\pi}_{2,q,\ast})}\Gamma(\rmZ_{d}(\overline{B}), \calO_{\lambda, n})^{\oplus2}_{/\frakp^{[pq]}_{n}}]
\end{equation*}
and the {new quotient} $\Gamma(\rmZ_{d}(q)(\overline{B}), \calO_{\lambda})_{\bullet/\frakp^{[pq]}_{n}}$ of the space $\Gamma(\rmZ_{d}(q)(\overline{B}), \calO_{\lambda})_{/\frakp^{[pq]}_{n}}$ at $q$ by
\begin{equation*}
\Gamma(\rmZ_{d}(q)(\overline{B}), \calO_{\lambda})_{\bullet/\frakp^{[pq]}_{n}}=\coker[\Gamma(\rmZ_{d}(\overline{B}), \calO_{\lambda})^{\oplus2}_{/\frakp^{[pq]}_{n}}\xrightarrow{\overline{\pi}^{\ast}_{1,q}-\overline{\pi}^{\ast}_{2,q}}\Gamma(\rmZ_{d}(q)(\overline{B}), \calO_{\lambda})_{/\frakp^{[pq]}_{n}}].
\end{equation*}
Let $\lambda: \Gamma(\rmZ_{d}(q)(\overline{B}), \calO_{\lambda})^{\bullet}_{/\frakp^{[pq]}_{n}}\rightarrow \Gamma(\rmZ_{d}(q)(\overline{B}), \calO_{\lambda})_{\bullet/\frakp^{[pq]}_{n}}$ be the natural quotient map.
\end{definition}
We remark that the sign on the degeneracy map is added to be compatible with the signs in the  Rapoport--Zink spectral sequence.

\begin{lemma}\label{new-old-def}
Under the Assumption \ref{assump}, we have the following isomorphisms of $\calO_{\lambda, n}$-modules
\begin{equation*}
\Gamma(\rmZ_{d}(q)(\overline{B}), \calO_{\lambda})^{\bullet}_{/\frakp^{[pq]}_{n}}\cong \Gamma(\rmZ_{d}(\overline{B}), \calO_{\lambda})_{/\frakp^{[p]}_{n}}\cong \Gamma(\rmZ_{d}(q)(\overline{B}), \calO_{\lambda})_{\bullet/\frakp^{[pq]}_{n}}.
\end{equation*}
\end{lemma}
\begin{proof}
Since the composite 
\begin{equation*}
\Gamma(\rmZ_{d}(\overline{B}), \calO_{\lambda})^{\oplus2}_{/\frakp^{[pq]}_{n}}\xrightarrow{\overline{\pi}^{\ast}_{1, q}-\overline{\pi}^{\ast}_{2,q}} \Gamma(\rmZ_{d}(q)(\overline{B}), \calO_{\lambda})_{/\frakp^{[pq]}_{n}}\xrightarrow{(-\overline{\pi}_{1, q, \ast}, \overline{\pi}^{\ast}_{2, q,\ast})}\Gamma(\rmZ_{d}(\overline{B}), \calO_{\lambda})^{\oplus2}_{/\frakp^{[pq]}_{n}}
\end{equation*}
is given by the matrix 
\begin{equation*}
\Delta_{n}=
\begin{pmatrix}
-(q+1) &   \rmT_{q}\\
\rmT_{q} & -(q+1)\\ 
\end{pmatrix},
\end{equation*}
the same argument as in Proposition \ref{new-old-part} and Proposition \ref{new-old-quotient} imply that $\overline{\pi}^{\ast}_{1, q}-\overline{\pi}^{\ast}_{2,q}$ induces an isomorphism
\begin{equation*}
\Gamma(\rmZ_{d}(\overline{B}), \calO_{\lambda})_{/\frakp^{[pq]}_{n}}\cong \Gamma(\rmZ_{d}(q)(\overline{B}), \calO_{\lambda})^{\bullet}_{/\frakp^{[pq]}_{n}}
\end{equation*}
and that $(-\overline{\pi}_{1, q, \ast}, \overline{\pi}_{2, q, \ast})$ induces an isomorphism
\begin{equation*}
\Gamma(\rmZ_{d}(q)(\overline{B}), \calO_{\lambda})_{\bullet/\frakp^{[pq]}_{n}}\cong \Gamma(\rmZ_{d}(\overline{B}), \calO_{\lambda})_{/\frakp^{[pq]}_{n}}.
\end{equation*}
\end{proof}

\begin{proposition}\label{split-2}
Let $(p,q)$ be a pair of $n$-admissible primes for $f$. We assume that Assumption \ref{assump} holds, then there is a split exact sequence of $\calO_{\lambda, n}[\rmG_{\QQ_{q^{2}}}]$-modules
\begin{equation*}
0\rightarrow \Gamma(\rmZ_{d}(\overline{B}), \calO_{\lambda}(1))_{/\frakp^{[p]}_{n}}\rightarrow \rmH^{1}(\rmX^{\natural}_{d}\otimes{\QQ^{\ac}_{q}}, \calO_{\lambda}(1))_{/\frakp^{[pq]}_{n}}\rightarrow \Gamma(\rmZ_{d}(\overline{B}), \calO_{\lambda})_{/\frakp^{[p]}_{n}}\rightarrow 0.
\end{equation*} 
\end{proposition}
\begin{proof}
We consider the localized weight spectral sequence for $\rmH^{1}(\rmX^{\natural}_{d}\otimes{\QQ^{\ac}_{q}}, \calO_{\lambda}(1))_{\frakm^{[pq]}}$ and its induced monodromy filtration:
\begin{equation}\label{mono}
\begin{aligned}
&0\subset^{\rmE^{1,0}_{2,\frakm^{[pq]}}} \rmM_{1}\rmH^{1}(\rmX^{\natural}_{d}\otimes{\QQ^{\ac}_{q}},\calO_{\lambda}(1))_{\frakm^{[pq]}}\subset^{\rmE^{0,1}_{2, \frakm^{[pq]}}} \rmM_{0}\rmH^{1}(\rmX^{\natural}_{d}\otimes{\QQ^{\ac}_{q}},\calO_{\lambda}(1))_{\frakm^{[pq]}}\\
&\subset^{\rmE^{-1,2}_{2, \frakm^{[pq]}}} \rmM_{-1}\rmH^{1}(\rmX^{\natural}_{d}\otimes{\QQ^{\ac}_{q}},\calO_{\lambda}(1))_{\frakm^{[pq]}}.\\
\end{aligned}
\end{equation}
By discussions in example in \ref{1-dim},  we have 
\begin{equation*}
\begin{aligned}
&\rmE^{1,0}_{2,\frakm^{[pq]}}= \coker[\Gamma(\rmZ_{d}(\overline{B}), \calO_{\lambda}(1))^{\oplus2}_{\frakm^{[pq]}_{n}}\rightarrow\Gamma(\rmZ_{d}(q)(\overline{B}), \calO_{\lambda}(1))_{\frakm^{[pq]}}]\\
&\rmE^{0,1}_{2, \frakm^{[pq]}}=0\\
&\rmE^{-1,2}_{2, \frakm^{[pq]}}=\ker[\Gamma(\rmZ_{d}(q)(\overline{B}), \calO_{\lambda})_{\frakm^{[pq]}}\rightarrow\Gamma(\rmZ_{d}(\overline{B}), \calO_{\lambda})^{\oplus2}_{\frakm^{[pq]}}].\\
\end{aligned}
\end{equation*}
Then we have an exact sequence 
\begin{equation*}
0\rightarrow\Gamma(\rmZ_{d}(q)(\overline{B}), \calO_{\lambda}(1))_{\bullet/\frakp^{[pq]}_{n}}\rightarrow\rmH^{1}(\rmX^{\natural}_{d}\otimes{\QQ^{\ac}_{q}}, \calO_{\lambda}(1))_{/\frakp^{[pq]}_{n}}\rightarrow\Gamma(\rmZ_{d}(q)(\overline{B}), \calO_{\lambda})^{\bullet}_{/\frakp^{[pq]}_{n}}\rightarrow 0
\end{equation*}
of $\calO_{\lambda, n}[\rmG_{\QQ_{q}}]$-modules which clearly splits as $\rmH^{1}(\rmX^{\natural}_{d}\otimes{\QQ^{\ac}_{q}}, \calO_{\lambda}(1))_{/\frakp^{[pq]}_{n}}$ is unramified by Lemma \ref{unramified-drinfeld}. And this exact sequence is isomorphic to 
\begin{equation*}
0\rightarrow \Gamma(\rmZ_{d}(\overline{B}), \calO_{\lambda}(1))_{/\frakp^{[p]}_{n}}\rightarrow \rmH^{1}(\rmX^{\natural}_{d}\otimes{\QQ^{\ac}_{q}}, \calO_{\lambda}(1))_{/\frakp^{[pq]}_{n}}\rightarrow \Gamma(\rmZ_{d}(\overline{B}), \calO_{\lambda})_{/\frakp^{[p]}_{n}}\rightarrow 0.
\end{equation*} 
by Lemma \ref{new-old-def}.
\end{proof}
\begin{remark}
Note that
\begin{equation*}
\begin{aligned}
\rmH^{1}_{\sin}(\QQ_{q^{2}}, \rmH^{1}(\rmX^{\natural}_{d}\otimes{\QQ^{\ac}_{q}}, \calO_{\lambda}(1))_{/\frakp^{[pq]}_{n}})&=\rmH^{1}(\rmI_{q^{2}}, \rmH^{1}(\rmX^{\natural}_{d}\otimes{\QQ^{\ac}_{q}}, \calO_{\lambda}(1))_{/\frakp^{[pq]}_{n}})^{\rmG_{\FF_{q^{2}}}}\\
&\cong \Hom(\ZZ_{\ell}(1), \rmH^{1}(\rmX^{\natural}_{d}\otimes{\QQ^{\ac}_{q}}, \calO_{\lambda}(1))_{/\frakp^{[pq]}_{n}})^{\rmG_{\FF_{q^{2}}}}\\
&\cong  \rmH^{1}(\rmX^{\natural}_{d}\otimes{\QQ^{\ac}_{q}}, \calO_{\lambda})^{\rmG_{\FF_{q^{2}}}}_{/\frakp^{[pq]}_{n}}\\
&\cong \coker[\Gamma(\rmZ_{d}(q)(\overline{B}), \calO_{\lambda})^{\bullet}_{/\frakp^{[p]}_{n}}\xrightarrow{\lambda}\Gamma(\rmZ_{d}(\overline{B}), \calO_{\lambda})_{\bullet/\frakp^{[p]}_{n}}]\\
\end{aligned}
\end{equation*}
since $\rmH^{1}(\rmX^{\natural}_{d}\otimes{\QQ^{\ac}_{q}}, \calO_{\lambda}(1))_{/\frakp^{[pq]}_{n}})$ is unramified as an $\calO_{\lambda, n}[\rmG_{\QQ_{q^{2}}}]$-module.
Let $\rmH^{1}(\rmX^{\natural}_{d}\otimes{\QQ^{\ac}_{q}}, \calO_{\lambda}(1))^{\ur}_{/\frakp^{[pq]}_{n}}$ be the part of $\rmH^{1}(\rmX^{\natural}_{d}\otimes{\QQ^{\ac}_{q}}, \calO_{\lambda}(1))_{/\frakp^{[pq]}_{n}}$ where the monodromy operator $N$ acts by zero. By definition, we have
\begin{equation*}
\begin{aligned}
\rmH^{1}(\FF_{q^{2}}, \rmH^{1}(\rmX^{\natural}_{d}\otimes{\QQ^{\ac}_{q}}, \calO_{\lambda}(1))^{\ur}_{/\frakp^{[pq]}_{n}})&\cong \rmH^{1}(\FF_{q^{2}}, \rmH^{1}(\rmX^{\natural}_{d}\otimes{\QQ^{\ac}_{q}}, \calO_{\lambda}(1))_{/\frakp^{[pq]}_{n}})\\
&\cong \ker[\Gamma(\rmZ_{d}(q)(\overline{B}), \calO_{\lambda})^{\bullet}_{/\frakp^{[p]}_{n}}\xrightarrow{\lambda}\Gamma(\rmZ_{d}(\overline{B}), \calO_{\lambda})_{\bullet/\frakp^{[p]}_{n}}]\\
\end{aligned}
\end{equation*}
where the first isomorphism follows from the fact that $\rmH^{1}(\rmX^{\natural}_{d}\otimes{\QQ^{\ac}_{q}}, \calO_{\lambda}(1))_{/\frakp^{[pq]}_{n}}$ is unramified at $q$. Consider the following commutative diagram
\begin{equation*}
\begin{tikzcd}
0 \arrow[r] & \ker(\Delta^{\natural}_{n}) \arrow[d] \arrow[r] & \Gamma(\rmZ_{d}(q)(\overline{B}), \calO_{\lambda})^{\bullet}_{/\frakp^{[pq]}_{n}} \arrow[d] \arrow[r, "\lambda"] & \Gamma(\rmZ_{d}(q)(\overline{B}), \calO_{\lambda})_{\bullet/\frakp^{[pq]}_{n}} \arrow[d]           &   \\
0 \arrow[r] & \Gamma(\rmZ_{d}(\overline{B}), \calO_{\lambda})^{\oplus2}_{/\frakp^{[pq]}_{n}} \arrow[r] \arrow[d, "\Delta^{\natural}_{n}"] & \Gamma(\rmZ_{d}(q)(\overline{B}), \calO_{\lambda})_{/\frakp^{[pq]}_{n}} \arrow[r] \arrow[d] & \Gamma(\rmZ_{d}(q)(\overline{B}), \calO_{\lambda})_{\bullet/\frakp^{[pq]}_{n}}\arrow[r] \arrow[d] & 0 \\
0 \arrow[r] & \Gamma(\rmZ_{d}(\overline{B}), \calO_{\lambda})^{\oplus2}_{/\frakp^{[pq]}_{n}} \arrow[r] \arrow[d] & \Gamma(\rmZ_{d}(\overline{B}), \calO_{\lambda})^{\oplus2}_{/\frakp^{[pq]}_{n}}  \arrow[r] \arrow[d] & 0                     &   \\
            & \coker(\Delta^{\natural}_{n}) \arrow[r]           & 0                     &                       &  
\end{tikzcd}
\end{equation*}
where $\Delta^{\natural}_{n}=\begin{pmatrix}-(q+1) &   \rmT_{q}\\\rmT_{q} & -(q+1)\\ \end{pmatrix}$.  Then it follows that 
\begin{equation*}\label{fil-quot-drinfeld}
\begin{aligned}
&\ker[\Gamma(\rmZ_{d}(q)(\overline{B}), \calO_{\lambda})^{\bullet}_{/\frakp^{[pq]}_{n}}\xrightarrow{\lambda}\Gamma(\rmZ_{d}(\overline{B}), \calO_{\lambda})_{\bullet/\frakp^{[pq]}_{n}}]\cong \ker(\Delta^{\natural}_{n})\cong \Gamma(\rmZ_{d}(\overline{B}), \calO_{\lambda})_{/\frakp^{[pq]}_{n}}\cong \Gamma(\rmZ_{d}(\overline{B}), \calO_{\lambda})_{/\frakp^{[p]}_{n}}\\
&\coker[\Gamma(\rmZ_{d}(q)(\overline{B}), \calO_{\lambda})^{\bullet}_{/\frakp^{[pq]}_{n}}\xrightarrow{\lambda}\Gamma(\rmZ_{d}(\overline{B}), \calO_{\lambda})_{\bullet/\frakp^{[pq]}_{n}}]\cong \coker(\Delta^{\natural}_{n})\cong \Gamma(\rmZ_{d}(\overline{B}), \calO_{\lambda})_{/\frakp^{[pq]}_{n}}\cong \Gamma(\rmZ_{d}(\overline{B}), \calO_{\lambda})_{/\frakp^{[p]}_{n}}.\\
\end{aligned}
\end{equation*}

By the $n$-admissibility of $q$, we have the following exact sequence
\begin{equation*}
0\rightarrow  \rmH^{1}(\rmX^{\natural}_{d}\otimes{\QQ^{\ac}_{q}}, \calO_{\lambda})^{\rmG_{\FF_{q^{2}}}}_{/\frakp^{[pq]}_{n}}(1)\rightarrow  \rmH^{1}(\rmX^{\natural}_{d}\otimes{\QQ^{\ac}_{q}}, \calO_{\lambda}(1))_{/\frakp^{[pq]}_{n}}\rightarrow \rmH^{1}(\FF_{q^{2}}, \rmH^{1}(\rmX^{\natural}_{d}\otimes{\QQ^{\ac}_{q}}, \calO_{\lambda}(1))_{/\frakp^{[pq]}_{n}})\rightarrow 0
\end{equation*}
which can be identified with
\begin{equation*}
0\rightarrow \Gamma(Z_{d}(\overline{B}), \calO_{\lambda}(1))_{/\frakp_{n}^{[p]}}\rightarrow \rmH^{1}(\rmX^{\natural}_{d}\otimes{\QQ^{\ac}_{q}}, \calO_{\lambda}(1))_{/\frakp^{[pq]}_{n}}\rightarrow \Gamma(Z_{d}(\overline{B}), \calO_{\lambda})_{/\frakp^{[p]}_{n}}\rightarrow 0
\end{equation*}
by the above discussions. By construction, this split exact sequence is the same exact sequence as in Proposition \ref{split-2}.
\end{remark}

\subsection{Ramified arithmetic level raising for triple product of Shimura curves}
Now we shift to the triple product setting. We consider a pair of $n$-admissible primes $(p, q)$ for $\triplef$.  In \cite[Theorem 2]{Wang}, we proved the ramified arithmetic level raising theorem for the triple product of Shimura curves $\rmX^{\natural 3}_{d}$ under slightly different assumptions, we now review this result and indicate necessary modifications to incorporate our assumptions. We define the $\calO_{\undlamb, n}[\rmG_{\QQ}]$-module  $\rmM^{[{pq}]}_{n}(\triplef)$ over $\calO_{\underline{\lambda}, n}$ by 
\begin{equation*}
\rmM^{[{pq}]}_{n}(\triplef, d)=\threetensor\rmH^{1}(\rmX^{\natural}_{d}\otimes{\QQ^{\ac}}, \calO_{\lambda_{i}}(1))_{/\frakp^{[pq]}_{i,n}}.
\end{equation*} 
There is an isomorphism $\rmH^{1}(\rmX^{\natural}_{d}\otimes{\QQ^{\ac}}, \calO_{\lambda_{i}}(1))_{/\frakp^{[pq]}_{i,n}}\cong \rho^{\oplus2}_{\calO_{\lambda_{i}, n}}$ by Lemma \ref{unramified-drinfeld}. Hence we have an isomorphism 
\begin{equation*}
\rmM^{[{pq}]}_{n}(\triplef, d)\cong \rmM_{n}(\triplef, d) 
\end{equation*}
as $\calO_{\underline{\lambda}, n}[\rmG_{\QQ}]$-modules. 
Recall we have defined 
\begin{equation*}
\rmZ_{n}(\triplef, d)= \threetensor\Gamma(\rmZ_{d}(\overline{B}),\calO_{\lambda_{i}})_{/\frakp^{[p]}_{i,n}} 
\end{equation*}
in \S\ref{unram-raise}. The results in the last subsection imply immediately the following proposition.

\begin{proposition}[Ramified level raising]\label{rami-level-raising}
Let $(p, q)$ be a pair of $n$-admissible primes for $\triplef$. Suppose each maximal ideal in $\frakm_{\triplef}$ satisfies Assumption \ref{assump}. 
\begin{enumerate}
\item Then $\rmM^{[pq]}_{n}(\triplef, d)$ is unramified at $q$ and we have a natural isomorphism
\begin{equation*}
 \rmM^{[pq]}_{n}(\triplef, d)\cong\rmZ_{n}(\triplef, d)\oplus \rmZ^{\oplus 3}_{n}(\triplef, d)(1) \oplus \rmZ^{\oplus 3}_{n}(\triplef, d)(2)\oplus \rmZ_{n}(\triplef, d)(3)
\end{equation*}
as $\calO_{\undlamb,n}[\rmG_{\FF_{q^{2}}}]$-modules. 

\item There is an isomorphism 
\begin{equation*}
\threesum(\threetensor\Gamma(Z_{d}(\overline{B}),\calO_{\lambda_{i}})_{/\frakp^{[p]}_{i, n}})\cong\rmH^{1}_{\sing}(\QQ_{q},  \rmM^{[pq]}_{n}(\triplef, d)(-1))
\end{equation*}
of $\calO_{\underline{\lambda}, n}$-modules.
\end{enumerate}
\end{proposition}
\begin{proof}
The first statement follows from immediately from Proposition \ref{split-2}.  For the second statement, we have
\begin{equation*}
\begin{aligned}
\rmH^{1}_{\sing}(\QQ_{q^{2}},  \rmM^{[pq]}_{n}(\triplef, d)(-1))&=\Hom(\ZZ_{\ell}(1),  \rmM^{[pq]}_{n}(\triplef, d)(-1)))^{\rmG_{\FF_{q^{2}}}}\\
&\cong \rmM^{[pq]}_{n}(\triplef, d)(-2)^{\rmG_{\FF_{q^{2}}}}\\
&\cong (\rmZ_{n}(\triplef,d)(-2)\oplus \rmZ^{\oplus 3}_{n}(\triplef, d)(-1) \oplus \rmZ^{\oplus 3}_{n}(\triplef, d)\oplus \rmZ_{n}(\triplef,d)(1))^{\rmG_{\FF_{q^{2}}}}\\
&\cong \rmZ^{\oplus 3}_{n}(\triplef,d).\\
\end{aligned}
\end{equation*}
Then one proceeds as in  \cite[Corollary 4.11]{Wang} to descend the result to $\QQ_{q}$. 
\end{proof}

Let $\Delta^{\natural}_{d}=[\theta^{\natural}_{\ast}\rmX^{\natural}_{d}]\in \mathrm{CH}^{2}( \rmX^{\natural3}_{d})$ be the diagonal cycle for the natural diagonal morphism $\theta^{\natural}: \rmX^{\natural}_{d}\rightarrow \rmX^{\natural3}_{d}$.
Consider the Abel--Jacobi map 
\begin{equation}\label{AJ-p-n}
\mathrm{AJ}^{[pq]}_{\triplef, n}: \mathrm{CH}^{2}(\rmX^{\natural}_{d})\rightarrow \rmH^{1}(\QQ,  \rmM^{[{pq}]}_{n}(\triplef, d)(-1))
\end{equation}
for $\rmM^{[{pq}]}_{n}(\triplef, d)(-1)$ defined similarly as in \eqref{AJ-n}. We define 
\begin{equation*}
\Theta^{[{pq}]}_{n}(\triplef, d)= \mathrm{AJ}^{[pq]}_{\triplef, n}(\Delta^{\natural}_{d})
\end{equation*}
as an element in $\rmH^{1}(\QQ,  \rmM^{[{pq}]}_{n}(\triplef, d)(-1))$. Consider the singular residue $\partial_{q}\Theta^{[pq]} _{n}(\triplef, d)\in  \rmH^{1}_{\sing}(\QQ_{q}, \rmM^{[pq]}_{n}(\triplef)(-1))$ of $\Theta^{[{pq}]}_{n}(\triplef, d)$ at $q$. Proposition \ref{rami-level-raising} allows us to view the element $\partial_{q}\Theta^{[pq]} _{n}(\triplef, d)$ as an element in the space $\threesum(\threetensor\Gamma(Z_{d}(\overline{B}),\calO_{\lambda_{i}})_{/\frakp^{[p]}_{i, n}})$. We denote by $\partial^{(j)}_{q}\Theta^{[pq]} _{n}(\triplef, d)$ the component of $\partial_{q}\Theta^{[pq]} _{n}(\triplef, d)$ in the $j$-th copy of 
 $\threesum(\threetensor\Gamma(Z_{d}(\overline{B}),\calO_{\lambda_{i}})_{/\frakp^{[p]}_{i, n}})$.  
 \begin{theorem}[The first reciprocity law]\label{1-law}
 Let $(p, q)$ be a pair of $n$-admissible primes for $\triplef$. Suppose each maximal ideal in $\frakm_{\triplef}$ satisfies Assumption \ref{assump}.  
Then the formula 
 \begin{equation*}
 (\partial^{(j)}_{q}\Theta^{[pq]} _{n}(\triplef, d), \phi_{1}\otimes \phi_{2}\otimes \phi_{3})=(q+1)^{3}\sum_{z\in \Delta_{d}(\overline{B})} \phi_{1}(z)\otimes \phi_{2}(z)\otimes \phi_{3}(z)
 \end{equation*}
 holds for any $\phi_{1}\otimes \phi_{2}\otimes \phi_{3}\in \threetensor\Gamma(\rmZ_{d}(\overline{B}),E_{\lambda_{i}}/\calO_{\lambda_{i}})[\frakp^{[p]}_{i, n}]$ and any $j\in\{1, 2, 3\}$. 
 \end{theorem}
 \begin{proof}
This follows from \cite[Theorem 4.12]{Wang} by using Proposition \ref{rami-level-raising} instead of \cite[Theorem 4.7, Corollary 4.11]{Wang}.  
 \end{proof}
 
 \begin{remark}
 Theorem \ref{1-law} and Theorem \ref{2-law} imply the following  relation
\begin{equation*}
\begin{aligned}
(\partial^{(j)}_{q}\Theta^{ [pq]} _{n}(\triplef,d), \phi_{1}\otimes \phi_{2}\otimes \phi_{3})
&=(q+1)^{3}(\loc^{(j)}_{p}(\Theta_{n}(\triplef, d)), \phi_{1}\otimes\phi_{2}\otimes\phi_{3})\\ 
&= (q+1)^{3}\mathbf{I}(\phi_{1}, \phi_{2}, \phi_{3})\\
\end{aligned}
\end{equation*}
where 
\begin{equation*}
\mathbf{I}(\phi_{1}, \phi_{2}, \phi_{3})=\sum_{z\in \Delta_{d}(\overline{B})} \phi_{1}(z)\otimes \phi_{2}(z)\otimes \phi_{3}(z) 
\end{equation*}
for any $\phi_{1}\otimes\phi_{2}\otimes \phi_{3}\in \threetensor\Gamma(\rmZ(\overline{B}),E_{\lambda_{i}}/\calO_{\lambda_{i}})[\frakp^{[p]}_{i, n}]$ is the Gross--Kudla period. This shows that 
\begin{equation*}
(\Theta_{n}(\triplef, d), \Theta^{[pq]}_{n}(\triplef, d)) 
\end{equation*}
along with $\mathbf{I}(\phi_{1}, \phi_{2}, \phi_{3})$ form a bipartite Euler system in a weaker sense: these classes do satisfy all the required reciprocity laws, however the Frobenius eigenvalues on $\rmM_{n}(\triplef, d)$ is not just $p$ and $1$ which is required by the definition in \cite{How} for a bipartite Euler system. This also means the finite part and singular part of the local Galois cohomology group of the triple product Galois representation at a level raising prime have rank greater than one.
 
 \end{remark}
 
\section{Applications to the Bloch--Kato conjectures}
\subsection{The rank $1$ case of the Bloch--Kato conjecture} 
We consider the triple tensor product Galois representation $\rho_{f_{1}, \lambda_{1}}\otimes\rho_{f_{2}, \lambda_{2}}\otimes\rho_{f_{3}, \lambda_{3}}$ over $E_{\undlamb}:=E_{\lambda_{1}}\otimes E_{\lambda_{2}}\otimes E_{\lambda_{3}}$ attached to $\triplef=(f_{1}, f_{2}, f_{3})$. Recall that
\begin{equation*}
\mathrm{V}(\triplef)=\rho_{f_{1}, \lambda_{1}}\otimes\rho_{f_{2}, \lambda_{2}}\otimes\rho_{f_{3}, \lambda_{3}}.
\end{equation*}
By \cite[Lemma 5.1]{Wang}, we have $\rmH^{1}(\QQ_{v}, \rmV(\triplef)(-1))=0$ for all $v\nmid \ell$. Therefore it makes sense to consider the following Selmer group.
\begin{definition}\label{BK-def}
The \emph{Bloch--Kato Selmer} group $\rmH^{1}_{f}(\QQ, \mathrm{V}(\triplef)(-1))$ of the representation  $\mathrm{V}(\triplef)(-1)$ is the subspace of $\rmH^{1}(\QQ, \mathrm{V}(\triplef)(-1))$ consisting of those classes $s$ such that $\loc_{\ell}(s)\in \rmH^{1}_{f}(\QQ_{\ell}, \mathrm{V}(\triplef)(-1))$ where 
\begin{equation*}
\rmH^{1}_{f}(\QQ_{\ell}, \mathrm{V}(\triplef)(-1))= \ker[\rmH^{1}_{f}(\QQ_{\ell}, \mathrm{V}(\triplef)(-1))\rightarrow \rmH^{1}_{f}(\QQ_{\ell}, \mathrm{V}(\triplef)\otimes \mathrm{B}_{\mathrm{cris}}(-1))].
\end{equation*}
\end{definition}
Let $(\pi_{1}, \pi_{2}, \pi_{3})$ be the triple of irreducible cuspidal automorphic representation of $\GL_{2}(\Adel)$ associated to the triple $\triplef=(f_{1}, f_{2},f_{3})$. We have the \emph{Garrett--Rankin triple  product $L$-function}
\begin{equation*}
L(f_{1}\otimes f_{2}\otimes f_{3}, s)=L(s-\frac{3}{2}, \pi_{1}\otimes\pi_{2}\otimes \pi_{3}, r)
\end{equation*}
where $r$ is the natural $8$-dimensional representation of the $L$-group of $\GL_{2}\times \GL_{2}\times \GL_{2}$ and $L(s-\frac{3}{2}, \pi_{1}\otimes\pi_{2}\otimes \pi_{3}, r)$ is the Langlands $L$-function for $r$. The parity of the order of vanishing of $L(f_{1}\otimes f_{2}\otimes f_{3}, s)$ at the central critical point $s={2}$ is controlled by the \emph{global root number} $\epsilon(\pi_{1}\otimes\pi_{2}\otimes\pi_{3}, r)\in\{\pm1\}$.  In this article we will focus on the case when $\epsilon(\pi_{1}\otimes\pi_{2}\otimes\pi_{3}, r)=-1$. Consider the following Abel--Jacobi map 
\begin{equation}\label{AJ-Q}
\mathrm{AJ}_{\triplef, E_{\undlamb}}: \mathrm{CH}^{2}(\rmX^{3}_{d})\rightarrow \rmH^{1}(\QQ,  \rmH^{3}(\rmX^{3}_{d}\otimes{\Qbar}, E_{\underline{\lambda}}(2))_{\frakm_{\triplef}}).
\end{equation} 
constructed in the same way as \ref{AJ}.
We denote by $\Theta(\triplef, d) \in \rmH^{1}(\QQ,  \rmV(\triplef)(-1))$ the image of $\Delta_{d}=\theta_{*}[X_{d}]$ under the composite of $\mathrm{AJ}_{\triplef, E_{\undlamb}}$ and the natural projection from $\rmH^{1}(\QQ,  \rmH^{3}(\rmX^{3}_{d}\otimes{\QQ^{\ac}}, E_{\underline{\lambda}}(2))_{\frakm_{\triplef}})$ to $\rmH^{1}(\QQ,  \rmV(\triplef)(-1))$.  Since $\interX^{3}_{d}$ has good reduction at $\ell$, $\Theta(\triplef, d)$ in fact lies in the Bloch--Kato Selmer group $\rmH^{1}_{f}(\QQ, \mathrm{V}(\triplef)(-1))$ by \cite[Theorem 3.1]{Nekovar}. 

In light of the Gross--Zagier formula for the Gross--Kudla--Schoen diagonal cycles \cite{YZZ-dia} and assuming the conjectural injectivity of the Abel--Jacobi map and the non-degeneracy of the height-pairing, there are infinitely many $l$ such that $\Theta(\triplef, d)$ is non-zero as long as the first derivative $L^{\prime}(f_{1}\otimes f_{2}\otimes f_{3}, 2)$ is non-vanishing. Therefore we can view $\Theta(\triplef, d)$ as an algebraic incarnation of the first derivative $L^{\prime}(f_{1}\otimes f_{2}\otimes f_{3}, 2)$ and it makes sense to formulate the following conjecture towards the rank one case of the Bloch--Kato conjecture for the triple product motive attached to $\triplef$.

\begin{conjecture}
Suppose the class $\Theta(\triplef, d)\in \rmH^{1}(\QQ,  \rmV(\triplef)(-1))$ is non-zero. Then the Bloch--Kato Selmer group $\rmH^{1}_{f}(\QQ_{\ell}, \mathrm{V}(\triplef)(-1))$ is of rank $1$ over ${E_{\undlamb}}$.
\end{conjecture}

\subsection{Bipartite Euler system for the symmetric cube motive} In this subsection, we apply the results in this article to the case when the triple product motive degenerates. Let $\triplef=(f, f, f)$ for a single modular form $f\in S^{\mathrm{new}}_{2}(\Gamma_{0}(N))$ such that $N=N^{+}N^{-}$ with $(N^{+}, N^{-})=1$ and $N^{-}$ is square-free with even number of prime factors. We denote by $\pi$ the automorphic representation of $\GL_{2}(\mathbf{A})$ associated to $f$. Recall that we have the Galois representation $\rho_{f, \lambda}$ and its residual representation $\overline{\rho}_{f, \lambda}$ attached to $f$. Let $\rmV=\rmV_{f,\lambda}$ be the representation space of $\rho_{f, \lambda}$. The triple tensor product representation $\rho^{\otimes 3}_{f, \lambda}$ admits the following factorization
\begin{equation*}
\rho^{\otimes3}_{f, \lambda}(-1)=\mathrm{Sym}^{3}\rho_{f, \lambda}(-1)\oplus \rho_{f, \lambda}\oplus \rho_{f, \lambda}
\end{equation*}
by the Schur functor construction. In the notations of previous sections, this factorization is given by $\rmV({\triplef})(-1)=\mathrm{Sym}^{3}\rmV_{f, \lambda}(-1)\oplus \rmV_{f, \lambda}\oplus \rmV_{f, \lambda}$. The triple tensor product $L$-function $L(f\otimes f\otimes f, s)$ factors accordingly as 
\begin{equation*}
L(f\otimes f\otimes f, s)=L(\mathrm{Sym}^{3}f, s)L(f, s-1)^{2}.
\end{equation*}
In fact by a result of Kim and Shahadi \cite{KS}, the $L(\mathrm{Sym}^{3}f, s)$ is entire and  $L(f\otimes f\otimes f, s)$ is divisible by $L(f, s-1)$. In the case where the global root number $\epsilon(\pi\otimes\pi\otimes \pi, r)$ is $-1$ for $L(f\otimes f\otimes f, s)$, the $L$-function $L(\mathrm{Sym}^{3}f, s)$ has also global root number $-1$ at $s=2$. 

We project the class  $\Theta(\triplef, d) \in \rmH^{1}(\QQ,  \rmV(\triplef)(-1))$ to the symmetic cube component according to the factorization 
\begin{equation*}
\rmV({\triplef})(-1)=\mathrm{Sym}^{3}\rmV_{f, \lambda}(-1)\oplus \rmV_{f, \lambda}\oplus  \rmV_{f, \lambda}. 
\end{equation*}
The resulting class will be denoted by  $\Theta^{\diamond}(\triplef, d)$. The symmetric cube component $\mathrm{Sym}^{3}\rmV_{f, \lambda}(-1)$ of  $\rmV({\triplef})(-1)$ will be denoted by  $\rmV^{\diamond}({\triplef})(-1)$ and thus $\Theta^{\diamond}(\triplef, d)\in \rmH^{1}(\QQ, \rmV^{\diamond}({\triplef})(-1))$. The class  $\Theta^{\diamond}(\triplef, d)$ can be considered as an algebraic incarnation of the first derivative $L^{\prime}(\mathrm{Sym}^{3}f, s)$ at $s=2$.  We define the \emph{ symmetric cube Bloch--Kato Selmer} group $\rmH^{1}_{f}(\QQ, \mathrm{V}^{\diamond}(\triplef)(-1))$  the same way as in Definition \ref{BK-def}.  Using the reciprocity laws proved in this article, we will prove the following theorem towards the rank $1$ case of the Bloch--Kato conjecture for the symmetric cube motive of the modular form $f$. 
\begin{theorem}\label{main-rank-1}
Suppose that the modular form $f$ satisfies the following assumptions:
\begin{enumerate}
\item The residual Galois representation $\bar{\rho}_{f, \lambda}\vert_{G_{\QQ(\zeta_{\ell})}}$ is absolutely irreducible;
\item The residual Galois representation $\bar{\rho}_{f, \lambda}$ is minimally ramified at primes in $\Sigma^{+}\cup\Sigma^{-}_{\ram}$. Moreover $\bar{\rho}_{f, \lambda}$ is ramified at primes in $\Sigma^{-}_{\ram}$;
\item The image of $\overline{\rho}_{f, \lambda}$ contains $\GL_{2}(\FF_{l})$.
\end{enumerate}
If the class $\Theta^{\diamond}(\triplef, d) \in \rmH^{1}(\QQ, \rmV^{\diamond}({\triplef})(-1))$ is non-zero, then the symmetric cube Bloch--Kato Selmer group 
\begin{equation*}
\rmH^{1}_{f}(\QQ, \mathrm{V}^{\diamond}(\triplef)(-1))
\end{equation*}
is of dimension $1$ over $E_{\lambda}$.
\end{theorem}

To prove this theorem, we review a few results on local Tate dualities. Let $\rmT$ be the $\calO_{\lambda}$-lattice in $\rho_{f, \lambda}$ determined by the isomorphism $\rmH^{1}(\rmX_{d}\otimes{\QQ^{\ac}}, \calO_{\lambda}(1))_{\frakm}\cong \rmT^{\oplus 2}$ in Proposition \ref{multi-one}. Let $\rmA=\rmV/\rmT$ be the divisible $\rmG_{\QQ}$-module associated to $f$. Let $n\geq 1$ be an integer. We denote  the $\lambda^{n}$ torsion subgroup of $\rmA$ by $\rmA_{n}$ and set $\rmT_{n}=\rmT/\lambda^{n}$. We put $\rmN^{\diamond}(\triplef)(-1)=\mathrm{Sym}^{3}\mathrm{A}(-1)$ and $\rmN^{\diamond}_{n}(\triplef)(-1)=\Sym^{3}\rmA_{n}(-1)$. Similarly, we let  $\rmM^{\diamond}(\triplef)(-1)=\Sym^{3}\rmT(-1)$ and $\rmM^{\diamond}_{n}(\triplef)(-1)=\Sym^{3}\rmT_{n}(-1)$. Note that $\rmN^{\diamond}_{n}(\triplef)(-1)$ and $\rmM^{\diamond}_{n}(\triplef)(-1)$ are Kummer dual to each other and are in fact naturally isomorphic to each other. In order to simplify the notations, we will let $\calO=\calO_{\lambda}$ and $\calO_{n}=\calO_{\lambda, n}$.

\begin{definition}\label{n-adm}
Let $n\geq 1$ be an integer. We say that a prime $p$ is \emph{$(n, 1)$-admissible} for $f$ if 
\begin{enumerate}
\item $p\nmid N\ell$;
\item $\ell\nmid p^{2}-1 $;
\item $\varpi^{n}\mid p+1-\epsilon_{p}(f)a_{p}(f)$ with $\epsilon_{p}(f)=1$. 
\end{enumerate}
\end{definition}

The $\rmG_{\QQ}$-equivariant pairing $\mathrm{N}^{\diamond}_{n}(\triplef)(-1)\times \rmM^{\diamond}_{n}(\triplef)(-1)\rightarrow \calO_{n}(1)$ induces for each place $v$ of $\QQ$ a local Tate pairing 
\begin{equation*}
(\hphantom{a}, \hphantom{b})_{v}: \rmH^{1}(\QQ_{v}, \rmN^{\diamond}_{n}(\triplef)(-1))\times \rmH^{1}(\QQ_{v},\rmM^{\diamond}_{n}(\triplef)(-1)) \rightarrow \rmH^{1}(\QQ_{v}, \calO_{n}(1))\cong \calO_{n}.
\end{equation*}
We will write $(s, t)_{v}$ for $s\in  \rmH^{1}(\QQ, \rmN^{\diamond}_{n}(\triplef)(-1))$ and $t\in  \rmH^{1}(\QQ,\rmM^{\diamond}_{n}(\triplef)(-1))$ instead of $(\loc_{v}(s), \loc_{v}(t))_{v}$.  Let $\rmH^{1}_{f}(\QQ_{\ell},\rmM^{\diamond}(\triplef)(-1))$ be the pullback of  $\rmH^{1}_{f}(\QQ_{\ell}, \rmV^{\diamond}(\triplef)(-1))$ under the natural map 
\begin{equation*}
\rmH^{1}(\QQ_{\ell},\rmM^{\diamond}(\triplef)(-1))\rightarrow \rmH^{1}(\QQ_{\ell}, \rmV^{\diamond}(\triplef)(-1)). 
\end{equation*}
We define $\rmH^{1}_{f}(\QQ_{\ell},\rmM^{\diamond}_{n}(\triplef)(-1))$ to be the reduction of $\rmH^{1}(\QQ_{\ell},\rmM^{\diamond}(\triplef)(-1))$ modulo $\lambda^{n}$. Similarly, we let $\rmH^{1}_{f}(\QQ_{\ell}, \rmN^{\diamond}(\triplef)(-1))$ be the image of $\rmH^{1}_{f}(\QQ_{\ell}, \rmV^{\diamond}(\triplef)(-1))$ in $\rmH^{1}(\QQ_{\ell}, \rmN^{\diamond}(\triplef)(-1))$. Then we define $\rmH^{1}_{f}(\QQ_{\ell}, \rmN^{\diamond}_{n}(\triplef)(-1))$ to be the pullback of $\rmH^{1}_{f}(\QQ_{\ell}, \rmN^{\diamond}(\triplef)(-1))$ under the natural map 
\begin{equation*}
\rmH^{1}(\QQ_{\ell}, \rmN^{\diamond}_{n}(\triplef)(-1))\rightarrow \rmH^{1}(\QQ_{\ell}, \rmN^{\diamond}(\triplef)(-1)).
\end{equation*}
The following lemma summarizes well known properties of the local Tate pairing.

\begin{lemma}\label{sel-pairing}
We have the following statements.
\begin{enumerate}
\item The sum $\sum_{v}(\hphantom{a},\hphantom{b})_{v}$ restricted to $\rmH^{1}(\QQ, \rmN^{\diamond}_{n}(\triplef)(-1))\times \rmH^{1}(\QQ,\rmM^{\diamond}_{n}(\triplef)(-1))$ is trivial. Here $v$ runs through all places in $\QQ$. 
\item  For every $v\neq \ell$, there exists an integer $n_{v}\geq 1$ independent of $n$ such that the image of the pairing 
\begin{equation*}
(\hphantom{a}, \hphantom{b})_{v}: \rmH^{1}(\QQ_{v}, \rmN^{\diamond}_{n}(\triplef)(-1))\times \rmH^{1}(\QQ_{v},\rmM^{\diamond}_{n}(\triplef)(-1)) \rightarrow \rmH^{1}(\QQ_{v}, \calO_{n}(1))\cong \calO_{n}
\end{equation*} 
is annihilated by $\varpi^{n_{v}}$.
\item For every $v\neq \ell$, the finite part $\rmH^{1}_{\mathrm{fin}}(\QQ_{v}, \rmN^{\diamond}_{n}(\triplef)(-1))$ is orthogonal to $\rmH^{1}_{\mathrm{fin}}(\QQ_{v},\rmM^{\diamond}_{n}(\triplef)(-1))$ under the pairing $(\hphantom{a}, \hphantom{b})_{v}$. Similarly,  $\rmH^{1}_{f}(\QQ_{\ell}, \rmN^{\diamond}_{n}(\triplef)(-1))$  is orthogonal to $\rmH^{1}_{f}(\QQ_{\ell},\rmM^{\diamond}_{n}(\triplef)(-1))$.
\item Let $p$ be a $(n, 1)$-admissible prime for $f$, then we have a perfect pairing 
\begin{equation*}
\rmH^{1}_{\mathrm{fin}}(\QQ_{p}, \rmN^{\diamond}_{n}(\triplef)(-1))\times \rmH^{1}_{\mathrm{sin}}(\QQ_{p},\rmM^{\diamond}_{n}(\triplef)(-1))\rightarrow \calO_{n}
\end{equation*}
of free $\calO_{n}$-modules of rank $1$. 
\end{enumerate}
\end{lemma}
\begin{proof}
The statement $(1)$ follows from global class field theory.  Part $(2)$ follows from the fact that $\rmH^{1}(\QQ_{v}, \rmV^{\diamond}(\triplef))=0$ for all $v\nmid \ell$ and thus $\rmH^{1}(\QQ_{v},\rmM^{\diamond}(\triplef))$ is torsion for all $v\nmid \ell$, see  \cite[Lemma 4.3]{Liu-HZ}. Part $(3)$ is well known, see \cite[Theorem 2.17(e)]{DDT} for the first statement and \cite[Lemma 4.8]{Liu-HZ} for the second statement. 

For $(4)$, it follows from the definition of an $(n, 1)$-admissible prime for $f$ that $\rmM_{n}(\triplef)$ is unramified at $p$ and
$\rmM_{n}(\triplef)\cong\calO_{n}\oplus \calO^{\oplus 3}_{n}(1) \oplus \calO^{\oplus 3}_{n}(2) \oplus\calO_{n}(3)$
as a Galois module for $\rmG_{\QQ_{p}}$. Then it follows from a simple computation that $\rmM^{\diamond}_{n}(\triplef)\cong\calO_{n}\oplus \calO_{n}(1) \oplus \calO_{n}(2) \oplus\calO_{n}(3)$. From this, it follows immediately that both $\rmH^{1}_{\mathrm{sin}}(\QQ_{p},\rmM^{\diamond}_{n}(\triplef)(-1))$ and $\rmH^{1}_{\mathrm{fin}}(\QQ_{p},\rmM^{\diamond}_{n}(\triplef)(-1))$ are free of rank $1$ over $\calO_{n}$. The last claim also follows form this. 
\end{proof}

We have the Abel--Jacobi map 
\begin{equation*}
\mathrm{AJ}^{\diamond}_{\triplef, n}: \mathrm{CH}^{2}(\rmX^{3}_{d})\rightarrow \rmH^{1}(\QQ, \rmM^{\diamond}_{n}(\triplef)(-1))
\end{equation*}
for $\rmM^{\diamond}_{n}(\triplef)(-1)$ constructed by composing the Abel--Jacobi map \eqref{AJ-n} with the natural projection map from $\rmM_{n}(\triplef, d)(-1)$ to its symmetric cube component $\rmM^{\diamond}_{n}(\triplef)(-1)$. We will denote by $\Theta^{\diamond}_{n}(\triplef, d) \in \rmH^{1}(\QQ, \rmM^{\diamond}_{n}(\triplef)(-1))$ the image of the Gross--Kudla--Schoen diagonal cycle $\Delta_{d}=\theta_{*}[\rmX_{d}]$ under the map $\mathrm{AJ}^{\diamond}_{\triplef, n}$.

Let $(p, q)$ be a pair of $(n, 1)$-admissible primes for $f$. Then we have the Shimura curve $\rmX^{\natural}_{d}$ and its integral model $\interX^{\natural}_{d}$ over $\ZZ[1/Ndp]$ considered in the setting of ramified arithmetic level-raising. We have another Abel--Jacobi map
\begin{equation*}
\mathrm{AJ}^{\diamond[pq]}_{\triplef, n}: \mathrm{CH}^{2}(\rmX^{\natural3}_{d})\rightarrow \rmH^{1}(\QQ,  \rmM^{\diamond}_{n}(\triplef)(-1))
\end{equation*}
for $\rmM^{\diamond}_{n}(\triplef)(-1)$ constructed by composing the Abel-Jacobi map \eqref{AJ-p-n} with the projection from $\rmM^{[pq]}_{n}(\triplef, d)(-1)\cong\rmM_{n}(\triplef, d)(-1)$ to its symmtric cube component $\rmM^{\diamond}_{n}(\triplef)(-1)$. We denote by 
\begin{equation*}
\Theta^{\diamond[pq]}_{n}(\triplef, d) \in \rmH^{1}(\QQ,  \rmM^{\diamond}_{n}(\triplef)(-1))
\end{equation*}
the image of $\Delta^{\natural}_{d}=\theta^{\natural}_{*}[\rmX^{\natural}_{d}]$ under 
the map $\mathrm{AJ}^{\diamond[pq]}_{\triplef, n}$. 
Since the pair $(\Theta^{\diamond}_{n}(\triplef, d), \Theta^{\diamond[pq]}_{n}(\triplef,d))$
satisfy  reciprocity laws as in  Theorem \ref{1-law} and Theorem \ref{2-law}, the collection of $(\Theta^{\diamond}_{n}(\triplef,d), \Theta^{\diamond[pq]}_{n}(\triplef,d))$ along with the Gross--Kudla periods form a bipartite Euler system in a slightly broader sense of \cite{How}: the Galois representations are assumed to be two-dimensional in \cite{How}.

\begin{myproof}{Theorem}{\ref{main-rank-1}}

We proceed by assuming that the dimension of $\rmH^{1}_{f}(\QQ, \mathrm{V}^{\diamond}(\triplef)(-1))$ is greater than $1$ and we will derive a contradiction from this. Consider the element $\Theta^{\diamond}(\triplef, d)$ in the statement of the theorem. Assume it is non-zero, then we can alway find an integer $n\geq 1$ large enough such that the element $\Theta^{\diamond}_{n}(\triplef, d)$ is non-zero in $\rmH^{1}(\QQ,  \rmN^{\diamond}_{n}(\triplef)(-1))$.  By \cite[Lemma 5.9]{Liu-HZ}, we can find a free $\calO_{n}$-module $S_{n}$ of rank $2$ in $\rmH^{1}_{f}(\QQ,  \rmN^{\diamond}_{n}(\triplef)(-1))$ with a basis $\{s, s^{\prime}\}$ such that $\Theta^{\diamond}_{n}(\triplef,d)=\varpi^{n_{0}}s$ for some $n_{0}< n$. 

Under the assumptions in the theorem, there are infinitely many $(n, 1)$-admissible primes for $f$ by the proof of \cite[Theorem 3.2]{BD}. By the same argument as in \cite[Theorem  5.7]{LT} which relies on \cite[Lemma 4.16, Lemma 4.11]{Liu-cubic}, we can choose a pair of admissible primes $(p, q)$ such that 
\begin{itemize}
\item the image of $\loc_{p}(s^{\prime})$ in $\rmH^{1}_{\mathrm{fin}}(\QQ_{p},  \rmN^{\diamond}_{1}(\triplef)(-1))$ is $0$;
\item the image of  $\loc_{p}(s)$ in  $\rmH^{1}_{\mathrm{fin}}(\QQ_{p},  \rmN^{\diamond}_{1}(\triplef)(-1))$ is non-zero;
\item the image of $\loc_{q}(s^{\prime})$ in $\rmH^{1}_{\mathrm{fin}}(\QQ_{q},  \rmN^{\diamond}_{1}(\triplef)(-1))$ is  non-zero.
\end{itemize}
By \cite[Lemma 3.4]{Liu-HZ}, we know
\begin{itemize}
\item $\loc_{v}(s^{\prime})\in \rmH^{1}_{\mathrm{fin}}(\QQ_{v}, \rmN^{\diamond}_{n}(\triplef)(-1))$ for all $v\nmid \ell N$;
\item $\loc_{\ell}(s^{\prime})\in \rmH^{1}_{f}(\QQ_{\ell}, \rmN^{\diamond}_{n}(\triplef)(-1))$. 
\end{itemize} 
We consider the element $\Theta^{\diamond [pq]}_{n}(\triplef, d)\in \rmH^{1}(\QQ, \rmM^{\diamond}_{n}(\triplef)(-1))$. This element has the following properties
\begin{itemize}
\item $\loc_{v}(\Theta^{\diamond[pq]}_{n}(\triplef,d))\in \rmH^{1}_{\mathrm{fin}}(\QQ_{v}, \rmM^{\diamond}_{n}(\triplef)(-1))$ for all $v\nmid \ell pqNd$;
\item $\loc_{\ell}(\Theta^{\diamond[pq]}_{n}(\triplef, d))\in \rmH^{1}_{f}(\QQ_{\ell}, \rmM^{\diamond}_{n}(\triplef)(-1))$.
\end{itemize}
These properties follow from the fact that $\interX^{\natural 3}_{d}$ has good reduction away from $pqNd$. By Lemma \ref{sel-pairing} $(2)$ and $(3)$, there is an integer $n_{N}\geq 1$  such that
\begin{equation}\label{use1}
\varpi^{n}\mid \sum_{v\not\in \{p, q\}} \varpi^{n_{N}}(s^{\prime}, \Theta^{\diamond[pq]}_{n}(\triplef, d))_{v}.
\end{equation}
Since  $\loc_{p}(s^{\prime})=0$ in $\rmH^{1}_{\mathrm{fin}}(\QQ_{p},  \rmN^{\diamond}_{1}(\triplef)(-1))$, we have $(s^{\prime}, \Theta^{\diamond[pq]}_{n}(\triplef,d))_{p}=0$. 
Let  $\phi\in  \Gamma(\rmZ_{d}(\overline{B}),E_{\lambda}/\calO_{\lambda})[\frakp^{[p]}_{n}]$ be an element such that its image 
under the map $\Gamma(\rmZ_{d}(\overline{B}),E_{\lambda}/\calO_{\lambda})[\frakp^{[p]}_{n}]\xrightarrow{\overline{\pi}_{1,d,\ast}} \Gamma(\rmZ(\overline{B}),E_{\lambda}/\calO_{\lambda})[\frakp^{[p]}_{n}]$
generates this rank one module.  By the choice of $s$, we have 
\begin{equation*}
\varpi^{n_{0}}\mid (\loc_{p}(\Theta^{\diamond}_{n}), \phi\otimes\phi\otimes\phi) \phantom{aa}\text{ but }\phantom{bb} \varpi^{n_{0}+1}\nmid (\loc_{p}(\Theta^{\diamond}_{n}(\triplef,d)), \phi\otimes\phi\otimes\phi).
\end{equation*}
Since we have
\begin{equation*} 
(\partial_{q}\Theta^{\diamond [pq]}_{n}(\triplef,d), \phi\otimes \phi\otimes \phi)
=(1+q)^{3}(\loc_{p}(\Theta^{\diamond}_{n}(\triplef,d)), \phi\otimes\phi\otimes\phi) 
\end{equation*}
by Theorem \ref{1-law} and $\ell\nmid 2(q+1)$, 
\begin{equation}\label{use2}
\varpi^{n_{0}}\mid(s^{\prime}, \Theta^{\diamond[pq]}_{n}(\triplef,d))_{q} \hphantom{aa}\text{but}\hphantom{bb} \varpi^{n_{0}+1}\nmid(s^{\prime}, \Theta^{\diamond[pq]}_{n}(\triplef,d))_{q}.
\end{equation}
We can choose $n, n_{0}, n_{N}$ such that $n> n_{0}+n_{N}$. By Lemma \ref{sel-pairing} $(1)$, 
\begin{equation*}
\varpi^{n}\mid \sum_{v}(s^{\prime}, \Theta^{\diamond[pq]}_{n}(\triplef,d))_{v}
\end{equation*}
where the sum runs through all the places $v$ of $\QQ$. This implies 
\begin{equation*}
\varpi^{n-n_{N}}\mid(s^{\prime}, \Theta^{\diamond[pq]}_{n}(\triplef,d))_{q} 
\end{equation*}
by \eqref{use1}. This is a contradiction to \eqref{use2}.
\end{myproof}

\end{document}